\documentclass{amsart}
\usepackage{amsmath,amssymb,amsthm,color,enumerate,latexsym}
\usepackage[bookmarks=true, colorlinks=true, linkcolor=blue, citecolor=cyan]{hyperref}
\usepackage[margin=1in]{geometry}
\usepackage{tikz}
\usepackage{tikz-cd}
\usepackage{pgflibraryarrows}
\usepackage{pgflibrarysnakes}

\newtheorem*{introremark}{Remark}
\newtheorem*{introexample}{Example}
\newtheorem{theorem}{Theorem}[section]
\newtheorem*{maintheorem}{Main Theorem}

\newtheorem{corollary}[theorem]{Corollary}
\newtheorem{definition}[theorem]{Definition}
\newtheorem{example}[theorem]{Example}
\newtheorem{lemma}[theorem]{Lemma}

\newtheorem{proposition}[theorem]{Proposition}
\newtheorem{remark}[theorem]{Remark}

\newcommand{\bfa}{\mathbf{a}}
\newcommand{\bfe}{\mathbf{e}}
\newcommand{\bfv}{\mathbf{v}}
\newcommand{\cA}{\mathcal{A}}
\newcommand{\cC}{\mathcal{C}}
\newcommand{\cF}{\mathcal{F}}
\newcommand{\cG}{\mathcal{G}}
\newcommand{\cK}{\mathcal{K}}

\newcommand{\cT}{\mathcal{T}}
\newcommand{\FF}{\mathbb{F}}
\newcommand{\kk}{\Bbbk}
\renewcommand{\AA}{\mathbb{A}}

\newcommand{\KK}{\mathbb{K}}
\newcommand{\QQ}{\mathbb{Q}}

\newcommand{\ZZ}{\mathbb{Z}}

\newcommand{\Ext}{\operatorname{Ext}}
\newcommand{\hgt}{{\operatorname{ht}}}
\newcommand{\dpt}{{\operatorname{dp}}}
\newcommand{\supp}{\operatorname{supp}}
\newcommand{\rsh}{{\operatorname{rsh}}}
\newcommand{\rep}{{\operatorname{rep}}}
\newcommand{\sh}{{\operatorname{sh}}}
\newcommand{\wt}{\operatorname{wt}}

\newcommand{\exclude}[1]{{}}

\newenvironment{enumeratea}{\begin{enumerate}[\upshape (a)]}
                           {\end{enumerate}}
\newenvironment{enumeratei}{\begin{enumerate}[\upshape (i)]}
                           {\end{enumerate}}
                           
\newcommand{\erase}[1]{{}}

\title[Rank Two Non-Commutative Laurent Phenomenon and Pseudo-Positivity]{Rank Two Non-Commutative Laurent Phenomenon\\and Pseudo-Positivity}
\author{Dylan Rupel}
\address[Dylan Rupel]{University of Notre Dame, Department of Mathematics, Notre Dame, IN 46556, USA}
\email{drupel@nd.edu}

\begin{document}
  \begin{abstract}
    We study polynomial generalizations of the Kontsevich automorphisms acting on the skew-field of formal rational expressions in two non-commuting variables.  
    Our main result is the Laurentness and pseudo-positivity of iterations of these automorphisms.  
    The resulting expressions are described combinatorially using a generalization (studied in \cite{rupel2}) of the combinatorics of compatible pairs in a maximal Dyck path developed by Lee, Li, and Zelevinsky in \cite{lee-li-zelevinsky}.  

    By specializing to quasi-commuting variables we obtain pseudo-positive expressions for rank 2 quantum generalized cluster variables.  
    In the \emph{binomial case} when all internal exchange coefficients are zero, this quantum specialization provides a positive combinatorial construction of counting polynomials for Grassmannians of submodules in exceptional representations of valued quivers with two vertices.
  \end{abstract}
  \maketitle


  \setcounter{section}{1}

  Let $\kk$ be any field of characteristic zero.  Write $\KK=\kk(X,Y)$ for the skew-field of rational functions in non-commuting variables $X$ and $Y$.  Intuitively, writing $\pi:\kk(X,Y)\to\kk(x,y)$ for the commutative specialization, we may formally invert any element $W\in\KK$ for which $\pi(W)\ne0$; this idea has been made precise in \cite{usnich1} by considering iterated localizations of the free algebra $\kk\langle X,Y\rangle$.

  For any nonzero polynomial $P\in\kk[z]$, consider the following $\kk$-linear endomorphism of $\KK$:
  \[F_P:\begin{cases} X\mapsto XYX^{-1}\\ Y\mapsto P(Y)X^{-1}.\end{cases}\]
  In fact $F_P$ is an automorphism of $\KK$ as $F_P^{-1}$ is given by $X\mapsto P(X)Y^{-1}$ and $Y\mapsto YXY^{-1}$.
  We remark for later use that the element $Q:=XYX^{-1}Y^{-1}$ is fixed by $F_P$ for any nonzero polynomial $P$.

  Fix nonzero monic polynomials $P_1,P_2\in\kk[z]$ such that $P_1(0)=1=P_2(0)$, say
  \[P_1(z)=p_{1,0}+p_{1,1}z+\cdots+p_{1,d_1-1}z^{d_1-1}+p_{1,d_1}z^{d_1}\quad\!\text{ and }\!\quad P_2(z)=p_{2,0}+p_{2,1}z+\cdots+p_{2,d_2-1}z^{d_2-1}+p_{2,d_2}z^{d_2}\]
  with $p_{1,0}=p_{1,d_1}=p_{2,0}=p_{2,d_2}=1$.  Set $\AA_+=\ZZ_{\ge0}[p_{1,i},p_{2,j}:0\le i\le d_1,0\le j\le d_2]$ and call this the \emph{pseudo-positive semiring} associated to $P_1$ and $P_2$.  

  We will write $\bar{P}_1(z):=z^{d_1}P_1(z^{-1})$ and $\bar{P}_2(z):=z^{d_2}P_2(z^{-1})$ for the polynomials obtained from $P_1$ and $P_2$ by reversing the order of the coefficients.  Note that these are again polynomials of the same form.  For notational convenience, for $k\in\ZZ$ we define 
  \begin{equation}
    \label{eq:reversed polynomials}
    P_k(z)=p_{k,0}+p_{k,1}z+\cdots+p_{k,d_k-1}z^{d_k-1}+p_{k,d_k}z^{d_k}:=\begin{cases}\bar{P}_2(z) & \text{if $k\equiv 0\mod 4$;}\\P_1(z) & \text{if $k\equiv 1\mod 4$;}\\P_2(z) & \text{if $k\equiv 2\mod 4$;}\\\bar{P}_1(z) & \text{if $k\equiv 3\mod 4$.}\end{cases}
  \end{equation}
  Here we use the notation $d_k:=d_1$ if $k$ is odd and $d_k:=d_2$ if $k$ is even.

  Set $X_0:=X$ and $Y_0:=Y$.
  \begin{maintheorem}
    \label{th:main}
    For any $m\ge0$, the elements $ X_m, Y_m\in\KK$ given by
    \begin{equation}
      \label{eq:non-commutative initial cluster mutation}
      X_m:=F_{P_1}F_{P_2}\cdots F_{P_m}(X)\quad\text{and}\quad Y_m:=F_{P_1}F_{P_2}\cdots F_{P_m}(Y)
    \end{equation}
    are contained in the semiring $\AA_+\langle X^{\pm1},Y^{\pm1}\rangle\subset\KK$ of pseudo-positive non-commutative Laurent polynomials.
  \end{maintheorem}
  \begin{introremark}
    When $P_1$ and $P_2$ are monic and of the same degree but $P_1(0)=P_2(0)\ne1$, this result also holds and can be deduced from the Main Theorem by passing to an appropriate algebraic extension of $\kk$, then rescaling all variables.  
    The same is true when the coefficients $p_{1,0},p_{1,d_1},p_{2,0},p_{2,d_2}\ne0$ are arbitrary but satisfy a balancing condition which we leave as an exercise for the reader to work out.  
    In the absence of such a balancing condition the definitions of the polynomials $P_k(z)$ should be adjusted according to the exchange polynomial mutation rules developed by Chekhov and Shapiro \cite{chekhov-shapiro}, though it is not clear that the combinatorial construction below can be adapted to this setting.
    Also, since $F_P(X)=QY$ for any polynomial $P$, we have $X_{m+1}=QY_m$ for $m\ge0$; in particular, the claim for the $X_m$ follows from the claim for the $Y_m$.
  \end{introremark}
  \begin{introremark}
    If $d_1d_2\le3$, the Main Theorem can be observed quite explicitly by computing $X_1,X_2,\ldots,X_m$ by hand for 
    \[m=\begin{cases} 4 & \text{if $d_1d_2=0$;}\\ 5 & \text{if $d_1d_2=1$;}\\ 6 & \text{if $d_1d_2=2$;}\\ 8 & \text{if $d_1d_2=3$;}\end{cases}\]
    and observing in each case that these are given by pseudo-positive non-commutative Laurent polynomials with $X_m=QXQ^{-1}$.
    The combinatorics below can be adapted to these cases, however in everything that follows we assume $d_1d_2\ge4$ as such cases may be treated more uniformly.
  \end{introremark}

  For the following example, observe that the $Y_m$ for $m\ge2$ may alternatively be computed via the following non-commutative analogue of generalized cluster exchange relations:
  \begin{equation}
    \label{eq:non-commutative exchange relation}
    Y_mQY_{m-2}=1+p_{m,1}Y_{m-1}+\cdots+p_{m,d_m-1}Y_{m-1}^{d_m-1}+Y_{m-1}^{d_m}.
  \end{equation}
  \begin{introexample}
    Let $P_1=1+p_{1,1}z+p_{1,2}z^2+z^3$ and $P_2=1+p_{2,1}z+z^2$.
    Then the first few non-commutative generalized cluster variables $Y_m$ are given by:
    \[ Y_1=(1+p_{1,1}Y+p_{1,2}Y^2+Y^3)X^{-1}, \quad Y_2=(1+p_{2,1}Y_1+Y_1^2)Y^{-1}Q^{-1}, \quad Y_3=(1+p_{1,2}Y_2+p_{1,1}Y_2^2+Y_2^3)Y_1^{-1}Q^{-1}.\]
    While $Y_2$ is manifestly an element of $\AA_+\langle X^{\pm1},Y^{\pm1}\rangle$, a highly nontrivial cancellation must occur in the expansion of $Y_3$ in order for it to be a pseudo-positive non-commutative Laurent polynomial.
    Such cancellations indeed occur and we obtain the expansion
    \begin{align*}
      Y_3&=
      \Big(XY^{-3}+p_{1,2}(p_{2,1}+Y_1)Y^{-1}+p_{1,1}(p_{2,1}+Y_1)Y^{-2}+p_{1,1}(1+p_{2,1}Y_1+Y_1^2)XY^{-1}X^{-1}(p_{2,1}+Y_1)Y^{-1}+\\
      &\qquad +(p_{2,1}+Y_1)Y^{-3}+(1+p_{2,1}Y_1+Y_1^2)XY^{-1}X^{-1}(p_{2,1}+Y_1)Y^{-2}+\\
      &\qquad +(1+p_{2,1}Y_1+Y_1^2)XY^{-1}X^{-1}(1+p_{2,1}Y_1+Y_1^2)XY^{-1}X^{-1}(p_{2,1}+Y_1)Y^{-1}\Big)Q^{-1}.
    \end{align*}
  \end{introexample}

  The automorphisms $F_{P_k}$ are generalizations of automorphisms of $\KK$ introduced by Kontsevich \cite{kontsevich} which are recovered in the \emph{binomial case} when $p_{1,i}=0=p_{2,j}$ for $1\le i\le d_1-1$ and $1\le j\le d_2-1$.
  In this binomial case, Kontsevich conjectured the Laurentness and positivity of the \emph{non-commutative cluster variables} $X_m$ and $Y_m$.
  This terminology is justified by specializing to commutative variables through which we recover the initial cluster mutations in the rank two cluster algebra \cite{fomin-zelevinsky} associated to the exchange matrix $\left[\begin{array}{cc} 0 & d_2\\ -d_1 & 0\end{array}\right]$ after composing with the transposition of initial cluster variables.
  In the binomial case, Laurentness was established by Usnich \cite{usnich0} when $d_1=d_2=2$, and by Berenstein and Retakh \cite{berenstein-retakh} for arbitrary polynomial degrees.
  Positivity in the binomial case was proven by Di Francesco and Kedem \cite{difrancesco-kedem} when $d_1d_2=4$, by Lee and Schiffler \cite{lee-schiffler} for $d_1=d_2$, and by the author \cite{rupel1} for arbitrary polynomial degrees.
  The proofs in \cite{lee-schiffler,rupel1} use a Dyck path combinatorics which is rather different than that employed in the present work.

  The Laurentness of $X_m$ and $Y_m$ was established by Usnich \cite{usnich2} in the special case where $P_k=P_1$ for all $k\in\ZZ$.
  We will prove the Main Theorem by providing a combinatorial construction of the elements $Y_m$, called \emph{non-commutative generalized cluster variables}.  
  This combinatorics was studied by the author \cite{rupel2} to construct greedy bases for (commutative) rank two generalized cluster algebras by building upon the notion of compatible pairs in a maximal Dyck path developed by Lee, Li, and Zelevinsky \cite{lee-li-zelevinsky} for constructing greedy bases of rank two cluster algebras.

  For $\bfa=(a_1,a_2)\in\ZZ_{\ge0}^2$, let $D:=D_\bfa$ denote the lattice path in the rectangle $[0,a_1]\times[0,a_2]$ which begins at $(0,0)$ takes unit length East and North steps to end at $(a_1,a_2)$ and is maximal among all such \emph{Dyck paths} that never pass above the main diagonal of the rectangle $[0,a_1]\times[0,a_2]$.  
  In other words, no lattice point of $D$ lies strictly above the main diagonal and any lattice point which lies strictly above $D$ also lies strictly above the main diagonal.
  Label the edges of $D$ as $E=\{1,\ldots,a_1+a_2\}$, where this bijection of ordered sets respects the natural order on edges from $(0,0)$ to $(a_1,a_2)$.
  There is a partition $E=H\sqcup V$, where $H$ (resp. $V$) denotes the set of horizontal (resp. vertical) edges of $D$. 

  For edges $e,e'\in E$, we write $ee'$ for the subpath of $D$ beginning with $e$ traveling North-East and ending with $e'$.  
  By convention, this path will be empty if $e$ is to the North-East of $e'$, while the path $ee$ contains the single edge $e$.  
  Let $\overline{e}e'$ (resp. $e\overline{e}'$) denote the path obtained from $ee'$ by removing the edge $e$ (resp. $e'$).  
  Write $(ee')_H$ (resp. $(ee')_V$) for the set of horizontal (resp. vertical) edges in the path $ee'$.
  We abbreviate $|ee'|_H:=|(ee')_H|$ and $|ee'|_V:=|(ee')_V|$.  
  
  \begin{remark}
    In \cite{lee-li-zelevinsky} and \cite{rupel2}, the definition for subpaths $ee'$ of $D$ includes a ``wrap-around'' condition whereby $ee'$ is non-empty for $e'<e$, however following \cite[Remark 2.21]{rupel2} such a condition will not be necessary in our situation and all relevant results quoted from \cite{rupel2} will be modified accordingly.
  \end{remark}

  \begin{definition}
    \label{def:compatibility}
    \cite[Definition 4.1]{rupel2}
    A \emph{grading} $\omega:E\to\ZZ_{\ge0}$ (on the edges) of $D$ is called \emph{compatible} if: for every $h\in H$ and $v\in V$ with $h<v$, there exists an edge $e$ along the path $hv$ so that at least one of the following holds:
    \begin{align}
      \tag{HGC}\label{eq:hgc} e\ne v\quad&\text{ and }\quad |he|_V=\sum\limits_{h'\in(he)_H} \omega(h');\\
      \tag{VGC}\label{eq:vgc} e\ne h\quad&\text{ and }\quad |ev|_H=\sum\limits_{v'\in(ev)_V} \omega(v').
    \end{align}
  \end{definition}

  Recall that $d_1,d_2\in\ZZ_{\ge0}$ denote the degrees of the exchange polynomials $P_1$ and $P_2$ respectively.  
  We say that a grading $\omega$ of $D$ is \emph{$(d_1,d_2)$-bounded} if $\omega(h)\le d_1$ for all $h\in H$ and $\omega(v)\le d_2$ for all $v\in V$.  
  For the remainder of the paper we will restrict to such bounded gradings $\omega$, though we continue to write $\omega:E\to\ZZ_{\ge0}$ throughout.
  This notion of compatible gradings was introduced in \cite{rupel2} building upon the notion of compatible subsets of $E$ developed in \cite{lee-li-zelevinsky} which can be recovered when $\omega(h)\in\{0,d_1\}$ for $h\in H$ and $\omega(v)\in\{0,d_2\}$ for $v\in V$.

  For a $(d_1,d_2)$-bounded grading $\omega$, we associate the non-commutative monomial $\wt_\omega(e)$ to each edge $e\in E$ as follows:
  \begin{equation}
    \label{eq:edge weights}
    \wt_\omega(e)=\begin{cases}
                    p_{1,\omega(e)}Y^{\omega(e)}X^{-1} & \text{ if $e\in H$;}\\
                    p_{2,d_2-\omega(e)}X^{\omega(e)+1}Y^{-1}X^{-1} & \text{ if $e\in V$.}\\
                  \end{cases}
  \end{equation}
  Thus we may associate a non-commutative Laurent monomial to each $(d_1,d_2)$-bounded grading $\omega$ by taking the product of the associated non-commutative weights in the natural order along the path $D$:
  \begin{equation}
    \label{eq:total path weights}
    Y_D(\omega):=\wt_\omega(1)\wt_\omega(2)\cdots\wt_\omega(a_1+a_2).
  \end{equation}  
  Define $Y_D:=\sum\limits_\omega Y_D(\omega)$, where the sum ranges over all $(d_1,d_2)$-bounded compatible gradings $\omega$ of $D$.
  This construction is analogous to those employed in \cite{lee-schiffler,rupel1} using a different Dyck path combinatorics.

  We will mainly be interested in the maximal Dyck paths $D_m:=D_{\bfa_m}$ for integer vectors $\bfa_m\in\ZZ^2$, $m\ge1$, defined recursively by
  \begin{equation}
    \label{eq:roots recursive}
    \bfa_0=(0,-1),\quad
    \bfa_1=(1,0),\quad
    \bfa_{m-1}+\bfa_{m+1}=\begin{cases}d_2\bfa_m & \text{if $m$ is odd;}\\ d_1\bfa_m & \text{if $m$ is even.}\end{cases}
  \end{equation}
  These vectors are precisely (half of) the \emph{almost positive roots} in the root system associated to the Cartan matrix $\left[\begin{array}{cc} 2 & -d_2\\ -d_1 & 2\end{array}\right]$ which describe the denominator vectors of cluster variables.
  The Main Theorem is an immediate consequence of the following combinatorial construction of the non-commutative generalized cluster variables $Y_m$.
  \begin{theorem}
    \label{th:combinatorial construction}
    For $m\ge1$, we have $Y_{D_m}=Y_m$.
  \end{theorem}

  \begin{example}
    \label{ex:combinatorial cluster variables}
    We continue the example from above with $P_1=1+p_{1,1}z+p_{1,2}z^2+z^3$ and $P_2=1+p_{2,1}z+z^2$.
    
    For $m=1$, we get $\bfa_1=(1,0)$ so that $D_1=$  
    \raisebox{0.15em}{\begin{tikzpicture}
      \begin{scope}[scale=0.5]
      \draw[step=1.0, black, thin] (0,0) grid (1,0);
      \draw[black, thin] (0,0) -- (1,0);
      \draw[black, fill=black] (0,0) circle (3pt);
      \draw[black, fill=black] (1,0) circle (3pt);
      \end{scope}
    \end{tikzpicture}}.
    This maximal Dyck path consists of a single horizontal edge which may be assigned any of the weights $0,1,2,3$, a situation which we denote by the dashed edge
    \raisebox{0.15em}{\begin{tikzpicture}
      \begin{scope}[scale=0.5]
      \draw[step=1.0, black, thin] (0,0) grid (1,0);
      \draw[black, thin] (0,0) -- (1,0);
      \draw[black, fill=black] (0,0) circle (3pt);
      \draw[black, line width=1.5pt, dash pattern=on 2.5pt off 1.5pt] (0,0) -- (1,0);
      \draw[black, fill=black] (1,0) circle (3pt);
      \end{scope}
    \end{tikzpicture}}.
    Summing the monomial contributions coming from \eqref{eq:edge weights} for each choice of weight, we get
    \[Y_{D_1}=X^{-1}+p_{1,1}YX^{-1}+p_{1,2}Y^2X^{-1}+Y^3X^{-1}=Y_1\]
    and this same equality holds for any dashed edge below.

    For $m=2$, we get $\bfa_2=(2,1)$ so that $D_2=$
    \raisebox{-0.5em}{\begin{tikzpicture}
      \begin{scope}[scale=0.5]
      \draw[step=1.0, black, thin] (0,0) grid (2,1);
      \draw[black, thin] (0,0) -- (2,1);
      \draw[black, fill=black] (0,0) circle (3pt);
      \draw[black, fill=black] (1,0) circle (3pt);
      \draw[black, fill=black] (2,0) circle (3pt);
      \draw[black, fill=black] (2,1) circle (3pt);
      \end{scope}
    \end{tikzpicture}}.
    In this case, the compatible gradings of the edges in $D_2$ are given by
    
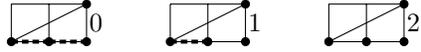
\[\begin{tikzpicture}
        \begin{scope}[scale=0.5]
        \draw[step=1.0, black, thin] (0,0) grid (2,1);
        \draw[black, thin] (0,0) -- (2,1);
        \draw[black, fill=black] (0,0) circle (3pt);
        \draw[black, line width=1.5pt, dash pattern=on 2.5pt off 1.5pt] (0,0) -- (1,0);
        \draw[black, fill=black] (1,0) circle (3pt);
        \draw[black, line width=1.5pt, dash pattern=on 2.5pt off 1.5pt] (1,0) -- (2,0);
        \draw[black, fill=black] (2,0) circle (3pt);
        \draw[black, fill=black] (2,1) circle (3pt);
        \node at (2.25,0.5) {$0$};
        \end{scope}
      \end{tikzpicture}
      \qquad
      \begin{tikzpicture}
        \begin{scope}[scale=0.5]
        \draw[step=1.0, black, thin] (0,0) grid (2,1);
        \draw[black, thin] (0,0) -- (2,1);
        \draw[black, fill=black] (0,0) circle (3pt);
        \draw[black, line width=1.5pt, dash pattern=on 2.5pt off 1.5pt] (0,0) -- (1,0);
        \draw[black, fill=black] (1,0) circle (3pt);
        \draw[black, fill=black] (2,0) circle (3pt);
        \draw[black, fill=black] (2,1) circle (3pt);
        \node at (2.25,0.5) {$1$};
        \end{scope}
      \end{tikzpicture}
      \qquad
      \begin{tikzpicture}
        \begin{scope}[scale=0.5]
        \draw[step=1.0, black, thin] (0,0) grid (2,1);
        \draw[black, thin] (0,0) -- (2,1);
        \draw[black, fill=black] (0,0) circle (3pt);
        \draw[black, fill=black] (1,0) circle (3pt);
        \draw[black, fill=black] (2,0) circle (3pt);
        \draw[black, fill=black] (2,1) circle (3pt);
        \node at (2.25,0.5) {$2$};
        \end{scope}
    \end{tikzpicture}\]
    where we again use a dashed line to indicate that a horizontal edge may be assigned any of the weights $0,1,2,3$ without affecting compatibility and a horizontal edge with no weight indicates that the edge must be assigned weight $0$ in order to have a compatible grading.
    Summing the monomial contributions coming from \eqref{eq:total path weights} for each choice of compatible grading, we get
    \[Y_{D_2}=Y_1^2XY^{-1}X^{-1}+Y_1X^{-1}p_{2,1}X^2Y^{-1}X^{-1}+X^{-1}X^{-1}X^3Y^{-1}X^{-1}=Y_2.\]

    For $m=3$, we get $\bfa_3=(5,3)$ so that $D_3=$
    \raisebox{-1.9em}{\begin{tikzpicture}
      \begin{scope}[scale=0.5]
      \draw[step=1.0, black, thin] (0,0) grid (5,3);
      \draw[black, thin] (0,0) -- (5,3);
      \draw[black, fill=black] (0,0) circle (3pt);
      \draw[black, fill=black] (1,0) circle (3pt);
      \draw[black, fill=black] (2,0) circle (3pt);
      \draw[black, fill=black] (2,1) circle (3pt);
      \draw[black, fill=black] (3,1) circle (3pt);
      \draw[black, fill=black] (4,1) circle (3pt);
      \draw[black, fill=black] (4,2) circle (3pt);
      \draw[black, fill=black] (5,2) circle (3pt);
      \draw[black, fill=black] (5,3) circle (3pt);
      \end{scope}
    \end{tikzpicture}}.
    In this case, the compatible gradings of the edges in $D_3$ are given by
    
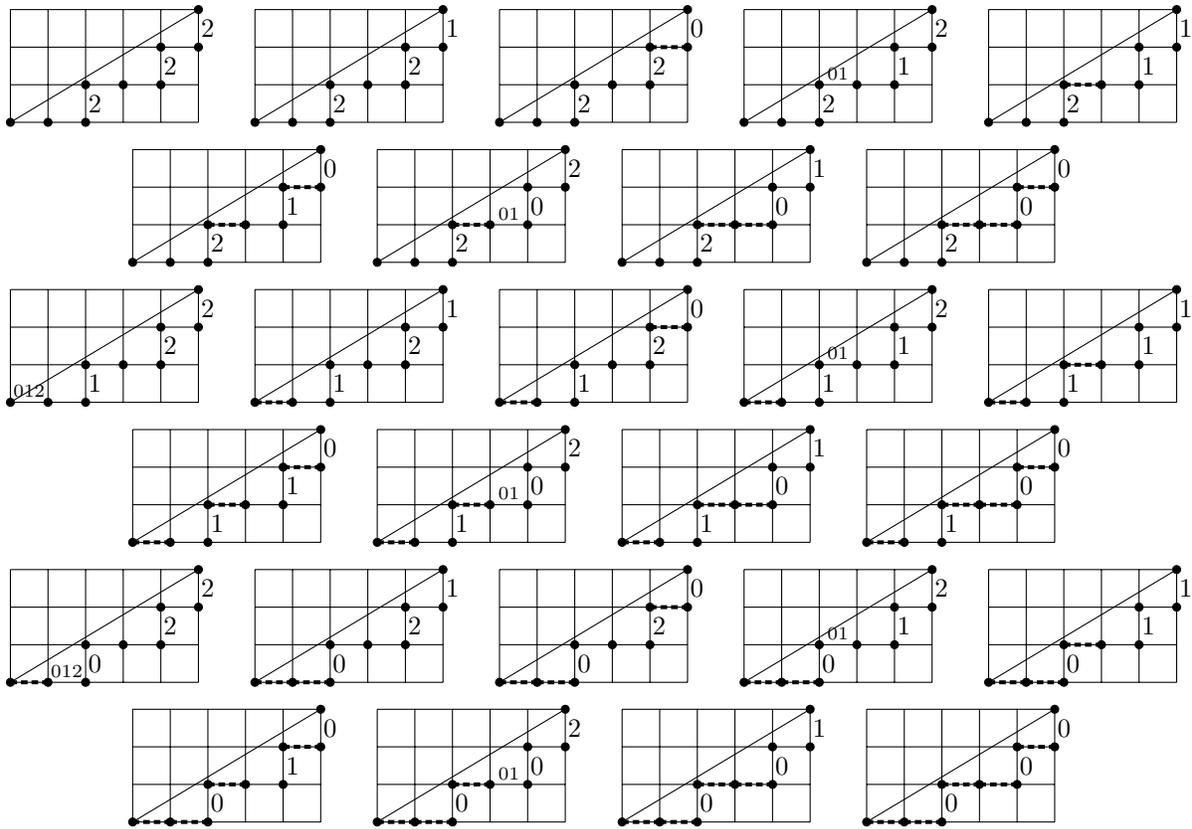
\begin{align*}
      \begin{tikzpicture}
        \begin{scope}[scale=0.5]
        \draw[step=1.0, black, thin] (0,0) grid (5,3);
        \draw[black, thin] (0,0) -- (5,3);
        \draw[black, fill=black] (0,0) circle (3pt);
        \draw[black, fill=black] (1,0) circle (3pt);
        \draw[black, fill=black] (2,0) circle (3pt);
        \node at (2.25,0.5) {$2$};
        \draw[black, fill=black] (2,1) circle (3pt);
        \draw[black, fill=black] (3,1) circle (3pt);
        \draw[black, fill=black] (4,1) circle (3pt);
        \node at (4.25,1.5) {$2$};
        \draw[black, fill=black] (4,2) circle (3pt);
        \draw[black, fill=black] (5,2) circle (3pt);
        \node at (5.25,2.5) {$2$};
        \draw[black, fill=black] (5,3) circle (3pt);
        \end{scope}
      \end{tikzpicture}
      \quad
      \begin{tikzpicture}
        \begin{scope}[scale=0.5]
        \draw[step=1.0, black, thin] (0,0) grid (5,3);
        \draw[black, thin] (0,0) -- (5,3);
        \draw[black, fill=black] (0,0) circle (3pt);
        \draw[black, fill=black] (1,0) circle (3pt);
        \draw[black, fill=black] (2,0) circle (3pt);
        \node at (2.25,0.5) {$2$};
        \draw[black, fill=black] (2,1) circle (3pt);
        \draw[black, fill=black] (3,1) circle (3pt);
        \draw[black, fill=black] (4,1) circle (3pt);
        \node at (4.25,1.5) {$2$};
        \draw[black, fill=black] (4,2) circle (3pt);
        \draw[black, fill=black] (5,2) circle (3pt);
        \node at (5.25,2.5) {$1$};
        \draw[black, fill=black] (5,3) circle (3pt);
        \end{scope}
      \end{tikzpicture}
      \quad
      \begin{tikzpicture}
        \begin{scope}[scale=0.5]
        \draw[step=1.0, black, thin] (0,0) grid (5,3);
        \draw[black, thin] (0,0) -- (5,3);
        \draw[black, fill=black] (0,0) circle (3pt);
        \draw[black, fill=black] (1,0) circle (3pt);
        \draw[black, fill=black] (2,0) circle (3pt);
        \node at (2.25,0.5) {$2$};
        \draw[black, fill=black] (2,1) circle (3pt);
        \draw[black, fill=black] (3,1) circle (3pt);
        \draw[black, fill=black] (4,1) circle (3pt);
        \node at (4.25,1.5) {$2$};
        \draw[black, fill=black] (4,2) circle (3pt);
        \draw[black, line width=1.5pt, dash pattern=on 2.5pt off 1.5pt] (4,2) -- (5,2);
        \draw[black, fill=black] (5,2) circle (3pt);
        \node at (5.25,2.5) {$0$};
        \draw[black, fill=black] (5,3) circle (3pt);
        \end{scope}
      \end{tikzpicture}
      \quad
      \begin{tikzpicture}
        \begin{scope}[scale=0.5]
        \draw[step=1.0, black, thin] (0,0) grid (5,3);
        \draw[black, thin] (0,0) -- (5,3);
        \draw[black, fill=black] (0,0) circle (3pt);
        \draw[black, fill=black] (1,0) circle (3pt);
        \draw[black, fill=black] (2,0) circle (3pt);
        \node at (2.25,0.5) {$2$};
        \draw[black, fill=black] (2,1) circle (3pt);
        \node at (2.5,1.3) {\scriptsize $01$};
        \draw[black, fill=black] (3,1) circle (3pt);
        \draw[black, fill=black] (4,1) circle (3pt);
        \node at (4.25,1.5) {$1$};
        \draw[black, fill=black] (4,2) circle (3pt);
        \draw[black, fill=black] (5,2) circle (3pt);
        \node at (5.25,2.5) {$2$};
        \draw[black, fill=black] (5,3) circle (3pt);
        \end{scope}
      \end{tikzpicture}
      \quad
      \begin{tikzpicture}
        \begin{scope}[scale=0.5]
        \draw[step=1.0, black, thin] (0,0) grid (5,3);
        \draw[black, thin] (0,0) -- (5,3);
        \draw[black, fill=black] (0,0) circle (3pt);
        \draw[black, fill=black] (1,0) circle (3pt);
        \draw[black, fill=black] (2,0) circle (3pt);
        \node at (2.25,0.5) {$2$};
        \draw[black, fill=black] (2,1) circle (3pt);
        \draw[black, line width=1.5pt, dash pattern=on 2.5pt off 1.5pt] (2,1) -- (3,1);
        \draw[black, fill=black] (3,1) circle (3pt);
        \draw[black, fill=black] (4,1) circle (3pt);
        \node at (4.25,1.5) {$1$};
        \draw[black, fill=black] (4,2) circle (3pt);
        \draw[black, fill=black] (5,2) circle (3pt);
        \node at (5.25,2.5) {$1$};
        \draw[black, fill=black] (5,3) circle (3pt);
        \end{scope}
      \end{tikzpicture}\\
      \begin{tikzpicture}
        \begin{scope}[scale=0.5]
        \draw[step=1.0, black, thin] (0,0) grid (5,3);
        \draw[black, thin] (0,0) -- (5,3);
        \draw[black, fill=black] (0,0) circle (3pt);
        \draw[black, fill=black] (1,0) circle (3pt);
        \draw[black, fill=black] (2,0) circle (3pt);
        \node at (2.25,0.5) {$2$};
        \draw[black, fill=black] (2,1) circle (3pt);
        \draw[black, line width=1.5pt, dash pattern=on 2.5pt off 1.5pt] (2,1) -- (3,1);
        \draw[black, fill=black] (3,1) circle (3pt);
        \draw[black, fill=black] (4,1) circle (3pt);
        \node at (4.25,1.5) {$1$};
        \draw[black, fill=black] (4,2) circle (3pt);
        \draw[black, line width=1.5pt, dash pattern=on 2.5pt off 1.5pt] (4,2) -- (5,2);
        \draw[black, fill=black] (5,2) circle (3pt);
        \node at (5.25,2.5) {$0$};
        \draw[black, fill=black] (5,3) circle (3pt);
        \end{scope}
      \end{tikzpicture}
      \quad
      \begin{tikzpicture}
        \begin{scope}[scale=0.5]
        \draw[step=1.0, black, thin] (0,0) grid (5,3);
        \draw[black, thin] (0,0) -- (5,3);
        \draw[black, fill=black] (0,0) circle (3pt);
        \draw[black, fill=black] (1,0) circle (3pt);
        \draw[black, fill=black] (2,0) circle (3pt);
        \node at (2.25,0.5) {$2$};
        \draw[black, fill=black] (2,1) circle (3pt);
        \draw[black, line width=1.5pt, dash pattern=on 2.5pt off 1.5pt] (2,1) -- (3,1);
        \draw[black, fill=black] (3,1) circle (3pt);
        \node at (3.5,1.3) {\scriptsize $01$};
        \draw[black, fill=black] (4,1) circle (3pt);
        \node at (4.25,1.5) {$0$};
        \draw[black, fill=black] (4,2) circle (3pt);
        \draw[black, fill=black] (5,2) circle (3pt);
        \node at (5.25,2.5) {$2$};
        \draw[black, fill=black] (5,3) circle (3pt);
        \end{scope}
      \end{tikzpicture}
      \quad
      \begin{tikzpicture}
        \begin{scope}[scale=0.5]
        \draw[step=1.0, black, thin] (0,0) grid (5,3);
        \draw[black, thin] (0,0) -- (5,3);
        \draw[black, fill=black] (0,0) circle (3pt);
        \draw[black, fill=black] (1,0) circle (3pt);
        \draw[black, fill=black] (2,0) circle (3pt);
        \node at (2.25,0.5) {$2$};
        \draw[black, fill=black] (2,1) circle (3pt);
        \draw[black, line width=1.5pt, dash pattern=on 2.5pt off 1.5pt] (2,1) -- (3,1);
        \draw[black, fill=black] (3,1) circle (3pt);
        \draw[black, line width=1.5pt, dash pattern=on 2.5pt off 1.5pt] (3,1) -- (4,1);
        \draw[black, fill=black] (4,1) circle (3pt);
        \node at (4.25,1.5) {$0$};
        \draw[black, fill=black] (4,2) circle (3pt);
        \draw[black, fill=black] (5,2) circle (3pt);
        \node at (5.25,2.5) {$1$};
        \draw[black, fill=black] (5,3) circle (3pt);
        \end{scope}
      \end{tikzpicture}
      \quad
      \begin{tikzpicture}
        \begin{scope}[scale=0.5]
        \draw[step=1.0, black, thin] (0,0) grid (5,3);
        \draw[black, thin] (0,0) -- (5,3);
        \draw[black, fill=black] (0,0) circle (3pt);
        \draw[black, fill=black] (1,0) circle (3pt);
        \draw[black, fill=black] (2,0) circle (3pt);
        \node at (2.25,0.5) {$2$};
        \draw[black, fill=black] (2,1) circle (3pt);
        \draw[black, line width=1.5pt, dash pattern=on 2.5pt off 1.5pt] (2,1) -- (3,1);
        \draw[black, fill=black] (3,1) circle (3pt);
        \draw[black, line width=1.5pt, dash pattern=on 2.5pt off 1.5pt] (3,1) -- (4,1);
        \draw[black, fill=black] (4,1) circle (3pt);
        \node at (4.25,1.5) {$0$};
        \draw[black, fill=black] (4,2) circle (3pt);
        \draw[black, line width=1.5pt, dash pattern=on 2.5pt off 1.5pt] (4,2) -- (5,2);
        \draw[black, fill=black] (5,2) circle (3pt);
        \node at (5.25,2.5) {$0$};
        \draw[black, fill=black] (5,3) circle (3pt);
        \end{scope}
      \end{tikzpicture}
      \quad
      \begin{tikzpicture}
        \begin{scope}[scale=0.25]
        \draw[step=1.0, white, thin] (0,0) grid (5,3);
        \end{scope}
      \end{tikzpicture}\\
      \begin{tikzpicture}
        \begin{scope}[scale=0.5]
        \draw[step=1.0, black, thin] (0,0) grid (5,3);
        \draw[black, thin] (0,0) -- (5,3);
        \draw[black, fill=black] (0,0) circle (3pt);
        \node at (0.5,0.3) {\scriptsize $012$};
        \draw[black, fill=black] (1,0) circle (3pt);
        \draw[black, fill=black] (2,0) circle (3pt);
        \node at (2.25,0.5) {$1$};
        \draw[black, fill=black] (2,1) circle (3pt);
        \draw[black, fill=black] (3,1) circle (3pt);
        \draw[black, fill=black] (4,1) circle (3pt);
        \node at (4.25,1.5) {$2$};
        \draw[black, fill=black] (4,2) circle (3pt);
        \draw[black, fill=black] (5,2) circle (3pt);
        \node at (5.25,2.5) {$2$};
        \draw[black, fill=black] (5,3) circle (3pt);
        \end{scope}
      \end{tikzpicture}
      \quad
      \begin{tikzpicture}
        \begin{scope}[scale=0.5]
        \draw[step=1.0, black, thin] (0,0) grid (5,3);
        \draw[black, thin] (0,0) -- (5,3);
        \draw[black, fill=black] (0,0) circle (3pt);
        \draw[black, line width=1.5pt, dash pattern=on 2.5pt off 1.5pt] (0,0) -- (1,0);
        \draw[black, fill=black] (1,0) circle (3pt);
        \draw[black, fill=black] (2,0) circle (3pt);
        \node at (2.25,0.5) {$1$};
        \draw[black, fill=black] (2,1) circle (3pt);
        \draw[black, fill=black] (3,1) circle (3pt);
        \draw[black, fill=black] (4,1) circle (3pt);
        \node at (4.25,1.5) {$2$};
        \draw[black, fill=black] (4,2) circle (3pt);
        \draw[black, fill=black] (5,2) circle (3pt);
        \node at (5.25,2.5) {$1$};
        \draw[black, fill=black] (5,3) circle (3pt);
        \end{scope}
      \end{tikzpicture}
      \quad
      \begin{tikzpicture}
        \begin{scope}[scale=0.5]
        \draw[step=1.0, black, thin] (0,0) grid (5,3);
        \draw[black, thin] (0,0) -- (5,3);
        \draw[black, fill=black] (0,0) circle (3pt);
        \draw[black, line width=1.5pt, dash pattern=on 2.5pt off 1.5pt] (0,0) -- (1,0);
        \draw[black, fill=black] (1,0) circle (3pt);
        \draw[black, fill=black] (2,0) circle (3pt);
        \node at (2.25,0.5) {$1$};
        \draw[black, fill=black] (2,1) circle (3pt);
        \draw[black, fill=black] (3,1) circle (3pt);
        \draw[black, fill=black] (4,1) circle (3pt);
        \node at (4.25,1.5) {$2$};
        \draw[black, fill=black] (4,2) circle (3pt);
        \draw[black, line width=1.5pt, dash pattern=on 2.5pt off 1.5pt] (4,2) -- (5,2);
        \draw[black, fill=black] (5,2) circle (3pt);
        \node at (5.25,2.5) {$0$};
        \draw[black, fill=black] (5,3) circle (3pt);
        \end{scope}
      \end{tikzpicture}
      \quad
      \begin{tikzpicture}
        \begin{scope}[scale=0.5]
        \draw[step=1.0, black, thin] (0,0) grid (5,3);
        \draw[black, thin] (0,0) -- (5,3);
        \draw[black, fill=black] (0,0) circle (3pt);
        \draw[black, line width=1.5pt, dash pattern=on 2.5pt off 1.5pt] (0,0) -- (1,0);
        \draw[black, fill=black] (1,0) circle (3pt);
        \draw[black, fill=black] (2,0) circle (3pt);
        \node at (2.25,0.5) {$1$};
        \draw[black, fill=black] (2,1) circle (3pt);
        \node at (2.5,1.3) {\scriptsize $01$};
        \draw[black, fill=black] (3,1) circle (3pt);
        \draw[black, fill=black] (4,1) circle (3pt);
        \node at (4.25,1.5) {$1$};
        \draw[black, fill=black] (4,2) circle (3pt);
        \draw[black, fill=black] (5,2) circle (3pt);
        \node at (5.25,2.5) {$2$};
        \draw[black, fill=black] (5,3) circle (3pt);
        \end{scope}
      \end{tikzpicture}
      \quad
      \begin{tikzpicture}
        \begin{scope}[scale=0.5]
        \draw[step=1.0, black, thin] (0,0) grid (5,3);
        \draw[black, thin] (0,0) -- (5,3);
        \draw[black, fill=black] (0,0) circle (3pt);
        \draw[black, line width=1.5pt, dash pattern=on 2.5pt off 1.5pt] (0,0) -- (1,0);
        \draw[black, fill=black] (1,0) circle (3pt);
        \draw[black, fill=black] (2,0) circle (3pt);
        \node at (2.25,0.5) {$1$};
        \draw[black, fill=black] (2,1) circle (3pt);
        \draw[black, line width=1.5pt, dash pattern=on 2.5pt off 1.5pt] (2,1) -- (3,1);
        \draw[black, fill=black] (3,1) circle (3pt);
        \draw[black, fill=black] (4,1) circle (3pt);
        \node at (4.25,1.5) {$1$};
        \draw[black, fill=black] (4,2) circle (3pt);
        \draw[black, fill=black] (5,2) circle (3pt);
        \node at (5.25,2.5) {$1$};
        \draw[black, fill=black] (5,3) circle (3pt);
        \end{scope}
      \end{tikzpicture}\\
      \begin{tikzpicture}
        \begin{scope}[scale=0.5]
        \draw[step=1.0, black, thin] (0,0) grid (5,3);
        \draw[black, thin] (0,0) -- (5,3);
        \draw[black, fill=black] (0,0) circle (3pt);
        \draw[black, line width=1.5pt, dash pattern=on 2.5pt off 1.5pt] (0,0) -- (1,0);
        \draw[black, fill=black] (1,0) circle (3pt);
        \draw[black, fill=black] (2,0) circle (3pt);
        \node at (2.25,0.5) {$1$};
        \draw[black, fill=black] (2,1) circle (3pt);
        \draw[black, line width=1.5pt, dash pattern=on 2.5pt off 1.5pt] (2,1) -- (3,1);
        \draw[black, fill=black] (3,1) circle (3pt);
        \draw[black, fill=black] (4,1) circle (3pt);
        \node at (4.25,1.5) {$1$};
        \draw[black, fill=black] (4,2) circle (3pt);
        \draw[black, line width=1.5pt, dash pattern=on 2.5pt off 1.5pt] (4,2) -- (5,2);
        \draw[black, fill=black] (5,2) circle (3pt);
        \node at (5.25,2.5) {$0$};
        \draw[black, fill=black] (5,3) circle (3pt);
        \end{scope}
      \end{tikzpicture}
      \quad
      \begin{tikzpicture}
        \begin{scope}[scale=0.5]
        \draw[step=1.0, black, thin] (0,0) grid (5,3);
        \draw[black, thin] (0,0) -- (5,3);
        \draw[black, fill=black] (0,0) circle (3pt);
        \draw[black, line width=1.5pt, dash pattern=on 2.5pt off 1.5pt] (0,0) -- (1,0);
        \draw[black, fill=black] (1,0) circle (3pt);
        \draw[black, fill=black] (2,0) circle (3pt);
        \node at (2.25,0.5) {$1$};
        \draw[black, fill=black] (2,1) circle (3pt);
        \draw[black, line width=1.5pt, dash pattern=on 2.5pt off 1.5pt] (2,1) -- (3,1);
        \draw[black, fill=black] (3,1) circle (3pt);
        \node at (3.5,1.3) {\scriptsize $01$};
        \draw[black, fill=black] (4,1) circle (3pt);
        \node at (4.25,1.5) {$0$};
        \draw[black, fill=black] (4,2) circle (3pt);
        \draw[black, fill=black] (5,2) circle (3pt);
        \node at (5.25,2.5) {$2$};
        \draw[black, fill=black] (5,3) circle (3pt);
        \end{scope}
      \end{tikzpicture}
      \quad
      \begin{tikzpicture}
        \begin{scope}[scale=0.5]
        \draw[step=1.0, black, thin] (0,0) grid (5,3);
        \draw[black, thin] (0,0) -- (5,3);
        \draw[black, fill=black] (0,0) circle (3pt);
        \draw[black, line width=1.5pt, dash pattern=on 2.5pt off 1.5pt] (0,0) -- (1,0);
        \draw[black, fill=black] (1,0) circle (3pt);
        \draw[black, fill=black] (2,0) circle (3pt);
        \node at (2.25,0.5) {$1$};
        \draw[black, fill=black] (2,1) circle (3pt);
        \draw[black, line width=1.5pt, dash pattern=on 2.5pt off 1.5pt] (2,1) -- (3,1);
        \draw[black, fill=black] (3,1) circle (3pt);
        \draw[black, line width=1.5pt, dash pattern=on 2.5pt off 1.5pt] (3,1) -- (4,1);
        \draw[black, fill=black] (4,1) circle (3pt);
        \node at (4.25,1.5) {$0$};
        \draw[black, fill=black] (4,2) circle (3pt);
        \draw[black, fill=black] (5,2) circle (3pt);
        \node at (5.25,2.5) {$1$};
        \draw[black, fill=black] (5,3) circle (3pt);
        \end{scope}
      \end{tikzpicture}
      \quad
      \begin{tikzpicture}
        \begin{scope}[scale=0.5]
        \draw[step=1.0, black, thin] (0,0) grid (5,3);
        \draw[black, thin] (0,0) -- (5,3);
        \draw[black, fill=black] (0,0) circle (3pt);
        \draw[black, line width=1.5pt, dash pattern=on 2.5pt off 1.5pt] (0,0) -- (1,0);
        \draw[black, fill=black] (1,0) circle (3pt);
        \draw[black, fill=black] (2,0) circle (3pt);
        \node at (2.25,0.5) {$1$};
        \draw[black, fill=black] (2,1) circle (3pt);
        \draw[black, line width=1.5pt, dash pattern=on 2.5pt off 1.5pt] (2,1) -- (3,1);
        \draw[black, fill=black] (3,1) circle (3pt);
        \draw[black, line width=1.5pt, dash pattern=on 2.5pt off 1.5pt] (3,1) -- (4,1);
        \draw[black, fill=black] (4,1) circle (3pt);
        \node at (4.25,1.5) {$0$};
        \draw[black, fill=black] (4,2) circle (3pt);
        \draw[black, line width=1.5pt, dash pattern=on 2.5pt off 1.5pt] (4,2) -- (5,2);
        \draw[black, fill=black] (5,2) circle (3pt);
        \node at (5.25,2.5) {$0$};
        \draw[black, fill=black] (5,3) circle (3pt);
        \end{scope}
      \end{tikzpicture}
      \quad
      \begin{tikzpicture}
        \begin{scope}[scale=0.25]
        \draw[step=1.0, white, thin] (0,0) grid (5,3);
        \end{scope}
      \end{tikzpicture}\\
      \begin{tikzpicture}
        \begin{scope}[scale=0.5]
        \draw[step=1.0, black, thin] (0,0) grid (5,3);
        \draw[black, thin] (0,0) -- (5,3);
        \draw[black, fill=black] (0,0) circle (3pt);
        \draw[black, line width=1.5pt, dash pattern=on 2.5pt off 1.5pt] (0,0) -- (1,0);
        \draw[black, fill=black] (1,0) circle (3pt);
        \node at (1.5,0.3) {\scriptsize $012$};
        \draw[black, fill=black] (2,0) circle (3pt);
        \node at (2.25,0.5) {$0$};
        \draw[black, fill=black] (2,1) circle (3pt);
        \draw[black, fill=black] (3,1) circle (3pt);
        \draw[black, fill=black] (4,1) circle (3pt);
        \node at (4.25,1.5) {$2$};
        \draw[black, fill=black] (4,2) circle (3pt);
        \draw[black, fill=black] (5,2) circle (3pt);
        \node at (5.25,2.5) {$2$};
        \draw[black, fill=black] (5,3) circle (3pt);
        \end{scope}
      \end{tikzpicture}
      \quad
      \begin{tikzpicture}
        \begin{scope}[scale=0.5]
        \draw[step=1.0, black, thin] (0,0) grid (5,3);
        \draw[black, thin] (0,0) -- (5,3);
        \draw[black, fill=black] (0,0) circle (3pt);
        \draw[black, line width=1.5pt, dash pattern=on 2.5pt off 1.5pt] (0,0) -- (1,0);
        \draw[black, fill=black] (1,0) circle (3pt);
        \draw[black, line width=1.5pt, dash pattern=on 2.5pt off 1.5pt] (1,0) -- (2,0);
        \draw[black, fill=black] (2,0) circle (3pt);
        \node at (2.25,0.5) {$0$};
        \draw[black, fill=black] (2,1) circle (3pt);
        \draw[black, fill=black] (3,1) circle (3pt);
        \draw[black, fill=black] (4,1) circle (3pt);
        \node at (4.25,1.5) {$2$};
        \draw[black, fill=black] (4,2) circle (3pt);
        \draw[black, fill=black] (5,2) circle (3pt);
        \node at (5.25,2.5) {$1$};
        \draw[black, fill=black] (5,3) circle (3pt);
        \end{scope}
      \end{tikzpicture}
      \quad
      \begin{tikzpicture}
        \begin{scope}[scale=0.5]
        \draw[step=1.0, black, thin] (0,0) grid (5,3);
        \draw[black, thin] (0,0) -- (5,3);
        \draw[black, fill=black] (0,0) circle (3pt);
        \draw[black, line width=1.5pt, dash pattern=on 2.5pt off 1.5pt] (0,0) -- (1,0);
        \draw[black, fill=black] (1,0) circle (3pt);
        \draw[black, line width=1.5pt, dash pattern=on 2.5pt off 1.5pt] (1,0) -- (2,0);
        \draw[black, fill=black] (2,0) circle (3pt);
        \node at (2.25,0.5) {$0$};
        \draw[black, fill=black] (2,1) circle (3pt);
        \draw[black, fill=black] (3,1) circle (3pt);
        \draw[black, fill=black] (4,1) circle (3pt);
        \node at (4.25,1.5) {$2$};
        \draw[black, fill=black] (4,2) circle (3pt);
        \draw[black, line width=1.5pt, dash pattern=on 2.5pt off 1.5pt] (4,2) -- (5,2);
        \draw[black, fill=black] (5,2) circle (3pt);
        \node at (5.25,2.5) {$0$};
        \draw[black, fill=black] (5,3) circle (3pt);
        \end{scope}
      \end{tikzpicture}
      \quad
      \begin{tikzpicture}
        \begin{scope}[scale=0.5]
        \draw[step=1.0, black, thin] (0,0) grid (5,3);
        \draw[black, thin] (0,0) -- (5,3);
        \draw[black, fill=black] (0,0) circle (3pt);
        \draw[black, line width=1.5pt, dash pattern=on 2.5pt off 1.5pt] (0,0) -- (1,0);
        \draw[black, fill=black] (1,0) circle (3pt);
        \draw[black, line width=1.5pt, dash pattern=on 2.5pt off 1.5pt] (1,0) -- (2,0);
        \draw[black, fill=black] (2,0) circle (3pt);
        \node at (2.25,0.5) {$0$};
        \draw[black, fill=black] (2,1) circle (3pt);
        \node at (2.5,1.3) {\scriptsize $01$};
        \draw[black, fill=black] (3,1) circle (3pt);
        \draw[black, fill=black] (4,1) circle (3pt);
        \node at (4.25,1.5) {$1$};
        \draw[black, fill=black] (4,2) circle (3pt);
        \draw[black, fill=black] (5,2) circle (3pt);
        \node at (5.25,2.5) {$2$};
        \draw[black, fill=black] (5,3) circle (3pt);
        \end{scope}
      \end{tikzpicture}
      \quad
      \begin{tikzpicture}
        \begin{scope}[scale=0.5]
        \draw[step=1.0, black, thin] (0,0) grid (5,3);
        \draw[black, thin] (0,0) -- (5,3);
        \draw[black, fill=black] (0,0) circle (3pt);
        \draw[black, line width=1.5pt, dash pattern=on 2.5pt off 1.5pt] (0,0) -- (1,0);
        \draw[black, fill=black] (1,0) circle (3pt);
        \draw[black, line width=1.5pt, dash pattern=on 2.5pt off 1.5pt] (1,0) -- (2,0);
        \draw[black, fill=black] (2,0) circle (3pt);
        \node at (2.25,0.5) {$0$};
        \draw[black, fill=black] (2,1) circle (3pt);
        \draw[black, line width=1.5pt, dash pattern=on 2.5pt off 1.5pt] (2,1) -- (3,1);
        \draw[black, fill=black] (3,1) circle (3pt);
        \draw[black, fill=black] (4,1) circle (3pt);
        \node at (4.25,1.5) {$1$};
        \draw[black, fill=black] (4,2) circle (3pt);
        \draw[black, fill=black] (5,2) circle (3pt);
        \node at (5.25,2.5) {$1$};
        \draw[black, fill=black] (5,3) circle (3pt);
        \end{scope}
      \end{tikzpicture}\\
      \begin{tikzpicture}
        \begin{scope}[scale=0.5]
        \draw[step=1.0, black, thin] (0,0) grid (5,3);
        \draw[black, thin] (0,0) -- (5,3);
        \draw[black, fill=black] (0,0) circle (3pt);
        \draw[black, line width=1.5pt, dash pattern=on 2.5pt off 1.5pt] (0,0) -- (1,0);
        \draw[black, fill=black] (1,0) circle (3pt);
        \draw[black, line width=1.5pt, dash pattern=on 2.5pt off 1.5pt] (1,0) -- (2,0);
        \draw[black, fill=black] (2,0) circle (3pt);
        \node at (2.25,0.5) {$0$};
        \draw[black, fill=black] (2,1) circle (3pt);
        \draw[black, line width=1.5pt, dash pattern=on 2.5pt off 1.5pt] (2,1) -- (3,1);
        \draw[black, fill=black] (3,1) circle (3pt);
        \draw[black, fill=black] (4,1) circle (3pt);
        \node at (4.25,1.5) {$1$};
        \draw[black, fill=black] (4,2) circle (3pt);
        \draw[black, line width=1.5pt, dash pattern=on 2.5pt off 1.5pt] (4,2) -- (5,2);
        \draw[black, fill=black] (5,2) circle (3pt);
        \node at (5.25,2.5) {$0$};
        \draw[black, fill=black] (5,3) circle (3pt);
        \end{scope}
      \end{tikzpicture}
      \quad
      \begin{tikzpicture}
        \begin{scope}[scale=0.5]
        \draw[step=1.0, black, thin] (0,0) grid (5,3);
        \draw[black, thin] (0,0) -- (5,3);
        \draw[black, fill=black] (0,0) circle (3pt);
        \draw[black, line width=1.5pt, dash pattern=on 2.5pt off 1.5pt] (0,0) -- (1,0);
        \draw[black, fill=black] (1,0) circle (3pt);
        \draw[black, line width=1.5pt, dash pattern=on 2.5pt off 1.5pt] (1,0) -- (2,0);
        \draw[black, fill=black] (2,0) circle (3pt);
        \node at (2.25,0.5) {$0$};
        \draw[black, fill=black] (2,1) circle (3pt);
        \draw[black, line width=1.5pt, dash pattern=on 2.5pt off 1.5pt] (2,1) -- (3,1);
        \draw[black, fill=black] (3,1) circle (3pt);
        \node at (3.5,1.3) {\scriptsize $01$};
        \draw[black, fill=black] (4,1) circle (3pt);
        \node at (4.25,1.5) {$0$};
        \draw[black, fill=black] (4,2) circle (3pt);
        \draw[black, fill=black] (5,2) circle (3pt);
        \node at (5.25,2.5) {$2$};
        \draw[black, fill=black] (5,3) circle (3pt);
        \end{scope}
      \end{tikzpicture}
      \quad
      \begin{tikzpicture}
        \begin{scope}[scale=0.5]
        \draw[step=1.0, black, thin] (0,0) grid (5,3);
        \draw[black, thin] (0,0) -- (5,3);
        \draw[black, fill=black] (0,0) circle (3pt);
        \draw[black, line width=1.5pt, dash pattern=on 2.5pt off 1.5pt] (0,0) -- (1,0);
        \draw[black, fill=black] (1,0) circle (3pt);
        \draw[black, line width=1.5pt, dash pattern=on 2.5pt off 1.5pt] (1,0) -- (2,0);
        \draw[black, fill=black] (2,0) circle (3pt);
        \node at (2.25,0.5) {$0$};
        \draw[black, fill=black] (2,1) circle (3pt);
        \draw[black, line width=1.5pt, dash pattern=on 2.5pt off 1.5pt] (2,1) -- (3,1);
        \draw[black, fill=black] (3,1) circle (3pt);
        \draw[black, line width=1.5pt, dash pattern=on 2.5pt off 1.5pt] (3,1) -- (4,1);
        \draw[black, fill=black] (4,1) circle (3pt);
        \node at (4.25,1.5) {$0$};
        \draw[black, fill=black] (4,2) circle (3pt);
        \draw[black, fill=black] (5,2) circle (3pt);
        \node at (5.25,2.5) {$1$};
        \draw[black, fill=black] (5,3) circle (3pt);
        \end{scope}
      \end{tikzpicture}
      \quad
      \begin{tikzpicture}
        \begin{scope}[scale=0.5]
        \draw[step=1.0, black, thin] (0,0) grid (5,3);
        \draw[black, thin] (0,0) -- (5,3);
        \draw[black, fill=black] (0,0) circle (3pt);
        \draw[black, line width=1.5pt, dash pattern=on 2.5pt off 1.5pt] (0,0) -- (1,0);
        \draw[black, fill=black] (1,0) circle (3pt);
        \draw[black, line width=1.5pt, dash pattern=on 2.5pt off 1.5pt] (1,0) -- (2,0);
        \draw[black, fill=black] (2,0) circle (3pt);
        \node at (2.25,0.5) {$0$};
        \draw[black, fill=black] (2,1) circle (3pt);
        \draw[black, line width=1.5pt, dash pattern=on 2.5pt off 1.5pt] (2,1) -- (3,1);
        \draw[black, fill=black] (3,1) circle (3pt);
        \draw[black, line width=1.5pt, dash pattern=on 2.5pt off 1.5pt] (3,1) -- (4,1);
        \draw[black, fill=black] (4,1) circle (3pt);
        \node at (4.25,1.5) {$0$};
        \draw[black, fill=black] (4,2) circle (3pt);
        \draw[black, line width=1.5pt, dash pattern=on 2.5pt off 1.5pt] (4,2) -- (5,2);
        \draw[black, fill=black] (5,2) circle (3pt);
        \node at (5.25,2.5) {$0$};
        \draw[black, fill=black] (5,3) circle (3pt);
        \end{scope}
      \end{tikzpicture}
      \quad
      \begin{tikzpicture}
        \begin{scope}[scale=0.25]
        \draw[step=1.0, white, thin] (0,0) grid (5,3);
        \end{scope}
      \end{tikzpicture}
    \end{align*}
    where we continue to use the conventions above and we indicate allowable weights when there are restrictions.
    We leave it as an exercise for the reader to verify that summing the monomial contributions coming from \eqref{eq:total path weights} for each choice of weights indeed gives $Y_3$.
  \end{example}

  Our proof of Theorem~\ref{th:combinatorial construction} requires a careful understanding of the recursive structure of the maximal Dyck paths $D_m$ which we will establish in the next section.  
  In Section~\ref{sec:compatible pairs}, we further develop the combinatorics of compatible gradings of $D_m$ introduced in \cite{rupel2}.  
  The main aim there is to understand gradings which behave nicely with respect to the recursive structure developed in Section~\ref{sec:dyck paths}.
  These results produce nicely factorizable summands of $Y_{D_m}$, facilitating an inductive proof of Theorem~\ref{th:combinatorial construction} which can be viewed as analogous to the arguments employed in \cite{lee-schiffler,rupel1}.
  Section~\ref{sec:proof of main} puts these combinatorial results together to establish Theorem~\ref{th:combinatorial construction}.
  We finish with Section~\ref{sec:specialization} discussing the specialization from non-commutative variables to quasi-commuting variables.
  A main goal of this section is proving Corollary~\ref{cor:combinatorial polynomials} which gives a positive combinatorial construction of counting polynomials for Grassmannians of subrepresentations in rigid indecomposable representations of a rank two valued quiver, this directly establishes a conjecture from \cite{rupel3} in the rank two case.
  These results lay the foundation for the work \cite{rupel-weist} which explains the reason for such a combinatorial construction of counting polynomials in the case $d_1=d_2$ by establishing the existence of cell decompositions for the Grassmannians of subrepresentations in rigid indecomposable representations of rank two quivers where cells are naturally labelled by compatible weightings of the maximal Dyck paths $D_m$.

  \subsection*{Notation}
  We adopt the following notational conventions throughout the paper.
  \begin{itemize}
    \item For integers $a<b$, set $[a,b]=\{a,a+1,\ldots,b\}$.
    \item Given any quantity $\alpha$ defined using the tuple $(d_1,d_2)$ or the pair of polynomials $(P_1,P_2)$, let $\alpha'$ denote the same quantity defined using the tuple $(d'_1,d'_2)=(d_0,d_1)$ or the polynomials $(P'_1,P'_2)=(P_0,P_1)$.  
      In particular, $p'_{1,j}=p_{2,d_2-j}$ and $p'_{2,j}=p_{1,j}$ when equation~\eqref{eq:edge weights} is applied to a $(d'_1,d'_2)$-bounded grading $\omega'$ on $D_{\bfa'_m}$.
    \item Equations that will be referenced globally will be assigned numbers, those that are referenced only locally (i.e.\ within a single proof) will be assigned symbols (e.g.\ $\dagger$ or $\ddagger$). 
      In particular, symbols labeling equations will be reused but this should not lead to any confusion. 
  \end{itemize}

  \subsection*{Acknowledgements}
  The author would like to thank the anonymous referees for suggestions which greatly improved the readability of this work.

\section{Maximal Dyck Paths}
\label{sec:dyck paths}

In this section we study the recursive structure present in the maximal Dyck paths $D_m$.  
To accomplish this, we note that the vectors $\bfa_m$ can be written more explicitly in terms of two-parameter Chebyshev polynomials $u_{m,k}$ ($m,k\in \ZZ$) defined recursively by: 
\begin{equation}
  \label{eq:chebyshev}
  u_{0,k}=0,\quad
  u_{1,k}=1,\quad
  u_{m+1,k+1}=d_ku_{m,k}-u_{m-1,k-1},
\end{equation}
where $d_k$ denotes the degree of the polynomial $P_k$ in equation~\eqref{eq:reversed polynomials}. 
Then, for $m\ge 1$, we have $\bfa_m=(u_{m,1},u_{m-1,2})$.  Write $\bfa'_m=(u'_{m,1},u'_{m-1,2})=(u_{m,2},u_{m-1,1})$ and set $D'_m=D_{\bfa'_m}$ for $m\ge1$.  
\begin{remark}
  \label{rem:chebyshev}
  To see the equivalence with equation~\eqref{eq:roots recursive}, one must use the identities $u_{m,k}=u_{m,k+1}$ for $m$ odd and $d_ku_{m,k}=d_{k+1}u_{m,k+1}$ for $m$ even.
\end{remark}
We record the next simple observations for future use.
\begin{lemma}
  \label{le:Dyck path inequality}
   For positive integers $d_1,d_2$ and any integers $m,k$, we have $u_{m,k+1}u_{m-2,k}<u_{m-1,k+1}u_{m-1,k}$.
\end{lemma}
\begin{proof}
  We work by induction on $m$, the case $m=2$ being the trivial inequality $0<1$.  
  For $m\ge3$, we have
  \[u_{m,k+1}u_{m-2,k}=d_ku_{m-1,k}u_{m-2,k}-u_{m-2,k-1}u_{m-2,k}<d_ku_{m-1,k}u_{m-2,k}-u_{m-3,k-1}u_{m-1,k}=u_{m-1,k+1}u_{m-1,k},\]
  where the inequality above uses induction.
  The case $m\le 1$ can be handled similarly.
\end{proof}

In order to establish a recursive structure for $D_m$, we will show that the maximal Dyck paths $D_m$ and $D'_m$ are intimately related as can be seen in the following.
\begin{example}
  \label{ex:dyck paths}
  For $d_1=3$ and $d_2=2$, we have the following maximal Dyck paths:
  \[\begin{array}{|c|c|c|c|c|}
    \hline
    m & \bfa_m & D_m & \bfa'_m & D'_m \\
    \hline
    1 & (1,0) & 
    \raisebox{0.15em}{\begin{tikzpicture}
      \begin{scope}[scale=0.5]
      \draw[step=1.0, black, thin] (0,0) grid (1,0);
      \draw[black, thin] (0,0) -- (1,0);
      \draw[black, fill=black] (0,0) circle (3pt);
      \draw[black, fill=black] (1,0) circle (3pt);
      \end{scope}
    \end{tikzpicture}}
    & (1,0) & 
    \raisebox{0.15em}{\begin{tikzpicture}
      \begin{scope}[scale=0.5]
      \draw[step=1.0, black, thin] (0,0) grid (1,0);
      \draw[black, thin] (0,0) -- (1,0);
      \draw[black, fill=black] (0,0) circle (3pt);
      \draw[black, fill=black] (1,0) circle (3pt);
      \end{scope}
    \end{tikzpicture}} \\
    \hline
    2 & (2,1) & 
    \raisebox{-0.5em}{\begin{tikzpicture}
      \begin{scope}[scale=0.5]
      \draw[step=1.0, black, thin] (0,0) grid (2,1);
      \draw[black, thin] (0,0) -- (2,1);
      \draw[black, fill=black] (0,0) circle (3pt);
      \draw[black, fill=black] (1,0) circle (3pt);
      \draw[black, fill=black] (2,0) circle (3pt);
      \draw[black, fill=black] (2,1) circle (3pt);
      \end{scope}
    \end{tikzpicture}}
    & (3,1) & 
    \raisebox{-0.5em}{\begin{tikzpicture}
      \begin{scope}[scale=0.5]
      \draw[step=1.0, black, thin] (0,0) grid (3,1);
      \draw[black, thin] (0,0) -- (3,1);
      \draw[black, fill=black] (0,0) circle (3pt);
      \draw[black, fill=black] (1,0) circle (3pt);
      \draw[black, fill=black] (2,0) circle (3pt);
      \draw[black, fill=black] (3,0) circle (3pt);
      \draw[black, fill=black] (3,1) circle (3pt);
      \end{scope}
    \end{tikzpicture}} \\
    \hline
    3 & (5,3) & 
    \raisebox{-1.9em}{\begin{tikzpicture}
      \begin{scope}[scale=0.5]
      \draw[step=1.0, black, thin] (0,0) grid (5,3);
      \draw[black, thin] (0,0) -- (5,3);
      \draw[black, fill=black] (0,0) circle (3pt);
      \draw[black, fill=black] (1,0) circle (3pt);
      \draw[black, fill=black] (2,0) circle (3pt);
      \draw[black, fill=black] (2,1) circle (3pt);
      \draw[black, fill=black] (3,1) circle (3pt);
      \draw[black, fill=black] (4,1) circle (3pt);
      \draw[black, fill=black] (4,2) circle (3pt);
      \draw[black, fill=black] (5,2) circle (3pt);
      \draw[black, fill=black] (5,3) circle (3pt);
      \end{scope}
    \end{tikzpicture}}
    & (5,2) &
    \raisebox{-1.3em}{\begin{tikzpicture}
      \begin{scope}[scale=0.5]
      \draw[step=1.0, black, thin] (0,0) grid (5,2);
      \draw[black, thin] (0,0) -- (5,2);
      \draw[black, fill=black] (0,0) circle (3pt);
      \draw[black, fill=black] (1,0) circle (3pt);
      \draw[black, fill=black] (2,0) circle (3pt);
      \draw[black, fill=black] (3,0) circle (3pt);
      \draw[black, fill=black] (3,1) circle (3pt);
      \draw[black, fill=black] (4,1) circle (3pt);
      \draw[black, fill=black] (5,1) circle (3pt);
      \draw[black, fill=black] (5,2) circle (3pt);
      \end{scope}
    \end{tikzpicture}} \\
    \hline
    4 & (8,5) & 
    \raisebox{-3.3em}{\begin{tikzpicture}
      \begin{scope}[scale=0.5]
      \draw[step=1.0, black, thin] (0,0) grid (8,5);
      \draw[black, thin] (0,0) -- (8,5);
      \draw[black, fill=black] (0,0) circle (3pt);
      \draw[black, fill=black] (1,0) circle (3pt);
      \draw[black, fill=black] (2,0) circle (3pt);
      \draw[black, fill=black] (2,1) circle (3pt);
      \draw[black, fill=black] (3,1) circle (3pt);
      \draw[black, fill=black] (4,1) circle (3pt);
      \draw[black, fill=black] (4,2) circle (3pt);
      \draw[black, fill=black] (5,2) circle (3pt);
      \draw[black, fill=black] (5,3) circle (3pt);
      \draw[black, fill=black] (6,3) circle (3pt);
      \draw[black, fill=black] (7,3) circle (3pt);
      \draw[black, fill=black] (7,4) circle (3pt);
      \draw[black, fill=black] (8,4) circle (3pt);
      \draw[black, fill=black] (8,5) circle (3pt);
      \end{scope}
    \end{tikzpicture}}
    & (12,5) &
    \raisebox{-3.3em}{\begin{tikzpicture}
      \begin{scope}[scale=0.5]
      \draw[step=1.0, black, thin] (0,0) grid (12,5);
      \draw[black, thin] (0,0) -- (12,5);
      \draw[black, fill=black] (0,0) circle (3pt);
      \draw[black, fill=black] (1,0) circle (3pt);
      \draw[black, fill=black] (2,0) circle (3pt);
      \draw[black, fill=black] (3,0) circle (3pt);
      \draw[black, fill=black] (3,1) circle (3pt);
      \draw[black, fill=black] (4,1) circle (3pt);
      \draw[black, fill=black] (5,1) circle (3pt);
      \draw[black, fill=black] (5,2) circle (3pt);
      \draw[black, fill=black] (6,2) circle (3pt);
      \draw[black, fill=black] (7,2) circle (3pt);
      \draw[black, fill=black] (8,2) circle (3pt);
      \draw[black, fill=black] (8,3) circle (3pt);
      \draw[black, fill=black] (9,3) circle (3pt);
      \draw[black, fill=black] (10,3) circle (3pt);
      \draw[black, fill=black] (10,4) circle (3pt);
      \draw[black, fill=black] (11,4) circle (3pt);
      \draw[black, fill=black] (12,4) circle (3pt);
      \draw[black, fill=black] (12,5) circle (3pt);
      \end{scope}
    \end{tikzpicture}} \\
    \hline
  \end{array}\]
\end{example}

\begin{lemma}
  \label{le:Dyck path recursion}
  For $m\ge1$, the following hold.
  \begin{enumeratea}
    \item The maximal Dyck path $D'_{m+1}$ can be obtained from $D_m$ via replacing each horizontal edge, together with the $\ell$ vertical edges which immediately follow it, by $d_1-\ell$ horizontal edges followed by a vertical edge.
    \item The maximal Dyck path $D_{m+1}$ can be obtained from $D'_m$ via replacing each horizontal edge, together with the $\ell$ vertical edges which immediately follow it, by $d_2-\ell$ horizontal edges followed by a vertical edge.
  \end{enumeratea}
\end{lemma}
\begin{proof}
  We only prove (a) as (b) will immediately follow by interchanging the roles of $d_1$ and $d_2$.  
  Let $D'$ denote the lattice path obtained from $D_m$ as in (a).  
  It follows from the definition that $D'$ will contain $d_1u_{m,1}-u_{m-1,2}=u_{m+1,2}$ horizontal edges and $u_{m,1}$ vertical edges.  
  We need to show that $D'$ does not cross above the main diagonal and that it is maximal with this property. 

  Write $v'_1,\ldots,v'_{u_{m,1}}$ for the vertical edges of $D'$ and for $1\le r\le u_{m,1}$ suppose $v'_r$ is immediately preceded by exactly $\ell_r$ horizontal edges of the same height.
  Suppose there exists $t$ so that $v'_t$ passes above the main diagonal, this is equivalent to the inequality $\frac{t}{\sum\limits_{r=1}^t\ell_r}>\frac{u_{m,1}}{u_{m+1,2}}$.  
  Using the equality $u_{m+1,2}=d_1u_{m,1}-u_{m-1,2}$, this may be rewritten as 
  \[\tag{$\dagger$}\frac{d_1t-\sum\limits_{r=1}^t\ell_r}{t}>\frac{u_{m-1,2}}{u_{m,1}}.\]
  But by construction of $D'$, we see that $d_1-\ell_r$ is the number of vertical edges immediately following the $r$-th horizontal edge of $D_m$.  
  Thus, by rewriting the numerator as $d_1t-\sum\limits_{r=1}^t\ell_r=\sum\limits_{r=1}^t(d_1-\ell_r)$ in the inequality ($\dagger$), we see that the subpath of $D_m$ containing the first $t$ horizontal edges and the vertical edges immediately following these horizontal edges will cross above the main diagonal of the rectangle $[0,u_{m,1}]\times[0,u_{m-1,2}]$, a contradiction.  
  Thus $D'$ is a Dyck path, i.e.\ it does not pass above the main diagonal.

  To see that $D'$ is maximal, suppose there exists a lattice point $(s,t)$ strictly above $D'$ which does not lie above the main diagonal.  
  Without loss of generality, we may take $s=\sum\limits_{r=1}^t\ell_r-1$ and get the inequality $\frac{t}{\sum\limits_{r=1}^t\ell_r-1}\le\frac{u_{m,1}}{u_{m+1,2}}$.  
  Using the equality $u_{m+1,2}=d_1u_{m,1}-u_{m-1,2}$, this may be rewritten as
  \[\tag{$\ddagger$}\frac{d_1t-\sum\limits_{r=1}^t\ell_r+1}{t}\le\frac{u_{m-1,2}}{u_{m,1}}.\]
  Now considering the same initial segment of $D_m$ as above, we see that the point $(t,d_1t-\sum\limits_{r=1}^t\ell_r+1)$ lies strictly above $D_m$, but the inequality ($\ddagger$) implies this point does not lie above the main diagonal of the rectangle $[0,u_{m,1}]\times[0,u_{m-1,2}]$, contradicting the maximality of $D_m$.  
  Thus we may conclude that $D'=D'_{m+1}$ is the maximal Dyck path in the rectangle $[0,u_{m+1,2}]\times[0,u_{m,1}]$.
\end{proof}

With this we obtain the recursive structure of the maximal Dyck paths $D_m$, the reader should compare Corollary~\ref{cor:recursive dyck path structure} with Example~\ref{ex:dyck paths}.
There are some subtleties in describing this recursive structure when one of $d_1$ or $d_2$ is equal to 1 as can be seen in the following.
\begin{example}
  \label{ex:dyck paths 2}
  For $d_1=1$ and $d_2=5$, we have the following maximal Dyck paths:
  \[\begin{array}{|c|c|c|c|c|}
    \hline
    m & \bfa_m & D_m & \bfa'_m & D'_m \\
    \hline
    1 & (1,0) & 
    \raisebox{0.15em}{\begin{tikzpicture}
      \begin{scope}[scale=0.5]
      \draw[step=1.0, black, thin] (0,0) grid (1,0);
      \draw[black, thin] (0,0) -- (1,0);
      \draw[black, fill=black] (0,0) circle (3pt);
      \draw[black, fill=black] (1,0) circle (3pt);
      \end{scope}
    \end{tikzpicture}}
    & (1,0) & 
    \raisebox{0.15em}{\begin{tikzpicture}
      \begin{scope}[scale=0.5]
      \draw[step=1.0, black, thin] (0,0) grid (1,0);
      \draw[black, thin] (0,0) -- (1,0);
      \draw[black, fill=black] (0,0) circle (3pt);
      \draw[black, fill=black] (1,0) circle (3pt);
      \end{scope}
    \end{tikzpicture}} \\
    \hline
    2 & (5,1) & 
    \raisebox{-0.5em}{\begin{tikzpicture}
      \begin{scope}[scale=0.5]
      \draw[step=1.0, black, thin] (0,0) grid (5,1);
      \draw[black, thin] (0,0) -- (5,1);
      \draw[black, fill=black] (0,0) circle (3pt);
      \draw[black, fill=black] (1,0) circle (3pt);
      \draw[black, fill=black] (2,0) circle (3pt);
      \draw[black, fill=black] (3,0) circle (3pt);
      \draw[black, fill=black] (4,0) circle (3pt);
      \draw[black, fill=black] (5,0) circle (3pt);
      \draw[black, fill=black] (5,1) circle (3pt);
      \end{scope}
    \end{tikzpicture}}
    & (1,1) & 
    \raisebox{-0.5em}{\begin{tikzpicture}
      \begin{scope}[scale=0.5]
      \draw[step=1.0, black, thin] (0,0) grid (1,1);
      \draw[black, thin] (0,0) -- (1,1);
      \draw[black, fill=black] (0,0) circle (3pt);
      \draw[black, fill=black] (1,0) circle (3pt);
      \draw[black, fill=black] (1,1) circle (3pt);
      \end{scope}
    \end{tikzpicture}} \\
    \hline
    3 & (4,1) & 
    \raisebox{-0.5em}{\begin{tikzpicture}
      \begin{scope}[scale=0.5]
      \draw[step=1.0, black, thin] (0,0) grid (4,1);
      \draw[black, thin] (0,0) -- (4,1);
      \draw[black, fill=black] (0,0) circle (3pt);
      \draw[black, fill=black] (1,0) circle (3pt);
      \draw[black, fill=black] (2,0) circle (3pt);
      \draw[black, fill=black] (3,0) circle (3pt);
      \draw[black, fill=black] (4,0) circle (3pt);
      \draw[black, fill=black] (4,1) circle (3pt);
      \end{scope}
    \end{tikzpicture}}
    & (4,5) &
    \raisebox{-3.3em}{\begin{tikzpicture}
      \begin{scope}[scale=0.5]
      \draw[step=1.0, black, thin] (0,0) grid (4,5);
      \draw[black, thin] (0,0) -- (4,5);
      \draw[black, fill=black] (0,0) circle (3pt);
      \draw[black, fill=black] (1,0) circle (3pt);
      \draw[black, fill=black] (1,1) circle (3pt);
      \draw[black, fill=black] (2,1) circle (3pt);
      \draw[black, fill=black] (2,2) circle (3pt);
      \draw[black, fill=black] (3,2) circle (3pt);
      \draw[black, fill=black] (3,3) circle (3pt);
      \draw[black, fill=black] (4,3) circle (3pt);
      \draw[black, fill=black] (4,4) circle (3pt);
      \draw[black, fill=black] (4,5) circle (3pt);
      \end{scope}
    \end{tikzpicture}} \\
    \hline
    4 & (15,4) & 
    \raisebox{-2.6em}{\begin{tikzpicture}
      \begin{scope}[scale=0.5]
      \draw[step=1.0, black, thin] (0,0) grid (15,4);
      \draw[black, thin] (0,0) -- (15,4);
      \draw[black, fill=black] (0,0) circle (3pt);
      \draw[black, fill=black] (1,0) circle (3pt);
      \draw[black, fill=black] (2,0) circle (3pt);
      \draw[black, fill=black] (3,0) circle (3pt);
      \draw[black, fill=black] (4,0) circle (3pt);
      \draw[black, fill=black] (4,1) circle (3pt);
      \draw[black, fill=black] (5,1) circle (3pt);
      \draw[black, fill=black] (6,1) circle (3pt);
      \draw[black, fill=black] (7,1) circle (3pt);
      \draw[black, fill=black] (8,1) circle (3pt);
      \draw[black, fill=black] (8,2) circle (3pt);
      \draw[black, fill=black] (9,2) circle (3pt);
      \draw[black, fill=black] (10,2) circle (3pt);
      \draw[black, fill=black] (11,2) circle (3pt);
      \draw[black, fill=black] (12,2) circle (3pt);
      \draw[black, fill=black] (12,3) circle (3pt);
      \draw[black, fill=black] (13,3) circle (3pt);
      \draw[black, fill=black] (14,3) circle (3pt);
      \draw[black, fill=black] (15,3) circle (3pt);
      \draw[black, fill=black] (15,4) circle (3pt);
      \end{scope}
    \end{tikzpicture}}
    & (3,4) &
    \raisebox{-2.6em}{\begin{tikzpicture}
      \begin{scope}[scale=0.5]
      \draw[step=1.0, black, thin] (0,0) grid (3,4);
      \draw[black, thin] (0,0) -- (3,4);
      \draw[black, fill=black] (0,0) circle (3pt);
      \draw[black, fill=black] (1,0) circle (3pt);
      \draw[black, fill=black] (1,1) circle (3pt);
      \draw[black, fill=black] (2,1) circle (3pt);
      \draw[black, fill=black] (2,2) circle (3pt);
      \draw[black, fill=black] (3,2) circle (3pt);
      \draw[black, fill=black] (3,3) circle (3pt);
      \draw[black, fill=black] (3,4) circle (3pt);
      \end{scope}
    \end{tikzpicture}} \\
    \hline
  \end{array}\]
\end{example}
In what follows we always assume $d_1d_2\ge4$ and take 
\[\delta_m:=\begin{cases} 1 & \text{if $d_{m-1}=1$ and $m\ne3$;}\\ 0 & \text{if $d_{m-1}\ne1$ or $m=3$.}\end{cases}\]
\begin{corollary}
  \label{cor:recursive dyck path structure}
  The maximal Dyck paths $D_m$, $m\ge1$, admit the following recursive structure:
  \begin{enumeratea}
    \item $D_1$ consists of a single horizontal edge;
    \item $D_2$ consists of $d_2$ horizontal edges followed by a single vertical edge;
    \item for $m\ge3$, $D_m$ can be constructed by concatenating $d_m-1-\delta_m$ copies of $D_{m-1}$ followed by a copy of $D_{m-1}$ with its first copy of $D_{m-2-\delta_m}$ removed.
  \end{enumeratea}
\end{corollary}
For the remainder of the paper we will understand the notation $D_{m-1}\setminus D_{m-2-\delta_m}$ to mean the terminal subpath of $D_{m-1}$ obtained by removing its first copy of $D_{m-2-\delta_m}$ as in Corollary~\ref{cor:recursive dyck path structure}(c).
\begin{remark}
  \label{rem:primed recursive structure}
  The roles of $d_1$ and $d_2$ must be interchanged when applying Corollary~\ref{cor:recursive dyck path structure} to $D'_m$.
\end{remark}
\begin{proof}
  Parts (a) and (b) are immediate from the definitions of $D_1$ and $D_2$.  
  Part (c) with $m=3$ follows from Lemma~\ref{le:Dyck path recursion} and part (b) since $D'_2$ consists of $d_1$ horizontal edges followed by a vertical edge.  

  We establish part (c) by induction on $m\ge4$.  
  Notice that by Remark~\ref{rem:primed recursive structure} the claimed recursive structures of $D_m$ and $D'_{m-1}$ are the same for $m\ge5$, thus we obtain the result for $D_m$ if we know the result for $D'_{m-1}$ by applying the construction from Lemma~\ref{le:Dyck path recursion}.  
  Hence it suffices to establish the claimed recursive structure for $D_4$.  

  If $d_3\ne1$, then $\delta_4=0$ and the structure of $D_4$ is immediately deduced from Lemma~\ref{le:Dyck path recursion} and part (c) for $D'_3$.  
  If $d_3=1$, then $D'_2$ consists of a single horizontal edge followed by a single vertical edge and $D'_3$ consists of $d_2-1=d_4-1$ copies of $D'_2$ followed by a vertical edge.  
  Applying Lemma~\ref{le:Dyck path recursion} to $D'_3$ shows that $D_4$ consists of $d_4-2$ copies of $D_3$ followed by a copy of $D_3$ with its first horizontal edge (i.e.\ its first $D_1$) removed.  
  This establishes part (c) for $D_4$ and completes the proof.
\end{proof}
\begin{corollary}
  \label{cor:initial subpath}
  For $m\ge2$, if the last edge of $D_m$ is omitted, the resulting lattice path identifies with an initial subpath of the maximal Dyck path $C_m$ obtained by concatenating $d_m$ copies of $D_{m-1}$.
\end{corollary}
\begin{proof}
  We work by induction on $m$, the case $m=2$ following immediately from Corollary~\ref{cor:recursive dyck path structure}(b).  

  Assume $m\ge3$.
  By part (c) of Corollary~\ref{cor:recursive dyck path structure}, in comparing $D_m$ to $C_m$ the first $d_m-1-\delta_m$ copies of $D_{m-1}$ inside $D_m$ may be ignored and the problem reduces to comparing the final $D_{m-1}\setminus D_{m-2-\delta_m}$ subpath of $D_m$ with the maximal Dyck path $D_{m-1}$.
  For $m=3$, removing the final edge of $D_{m-1}\setminus D_{m-2-\delta_m}$ produces $d_2-1$ consecutive horizontal edges which clearly identifies with an initial subpath of $D_{m-1}$.
  Assume $m\ge4$.
  There are two cases to consider.
  \begin{itemize}
    \item If $d_{m-1}\ne1$, $D_{m-1}$ consists of $d_{m-1}-1-\delta_{m-1}$ copies of $D_{m-2}$ followed by a copy of $D_{m-2}\setminus D_{m-3-\delta_{m-1}}$.  
      It follows that comparing $D_{m-1}\setminus D_{m-2}$ with $D_{m-1}$ reduces to comparing $D_{m-2}\setminus D_{m-3-\delta_{m-1}}$ with $D_{m-2}$.
      But by induction, we know that we obtain an initial subpath of $D_{m-2}$ by removing the last edge of $D_{m-2}\setminus D_{m-3-\delta_{m-1}}$.
    \item When $d_{m-1}=1$, the maximal Dyck path $D_{m-1}$ is just $D_{m-2}\setminus D_{m-3}$.  
      But $D_{m-2}$ consists of $d_m-2$ copies of $D_{m-3}$ followed by a copy of $D_{m-3}\setminus D_{m-5}$ and so $D_{m-1}$ consists of $d_m-3$ copies of $D_{m-3}$ followed by a copy of $D_{m-3}\setminus D_{m-5}$.  
      Hence comparing $D_{m-1}\setminus D_{m-3}$ with $D_{m-1}$ reduces to comparing $D_{m-3}\setminus D_{m-5}$ with $D_{m-3}$, but by induction we know that removing the last edge of $D_{m-3}\setminus D_{m-5}$ produces an initial subpath of $D_{m-3}$.
  \end{itemize}
  The two items above show that we get an initial subpath of $D_{m-1}$ by removing the last edge of $D_{m-1}\setminus D_{m-2-\delta_m}$ and thus removing the last edge of $D_m$ produces an initial subpath of $C_m$. 
\end{proof}

Let $E_m$ denote the edges of $D_m$, where $E_m=H_m\sqcup V_m$ for horizontal edges $H_m=\{h_1,\ldots,h_{u_{m,1}}\}$ and vertical edges $V_m=\{v_1,\ldots,v_{u_{m-1,2}}\}$, both labeled in the natural order along $D_m$.  
We may describe the structure of $D_m$ as follows.
\begin{lemma}\cite[Lemma~3.2]{rupel2}
  \label{le:height and depth}
  For $m\ge2$, the following hold.
  \begin{enumeratea}
    \item There are exactly $\hgt(h_i):=\lfloor (i-1)u_{m-1,2}/u_{m,1}\rfloor$ vertical edges of $D_m$ preceding the horizontal edge $h_i$, call this number the \emph{height} of $h_i$;
    \item There are exactly $\dpt(v_t):=\lceil tu_{m,1}/u_{m-1,2}\rceil$ horizontal edges of $D_m$ preceding the vertical edge $v_t$, call this number the \emph{depth} of $v_t$.
  \end{enumeratea}
\end{lemma}
In the natural labeling of edges, Lemma~\ref{le:height and depth} gives $h_i=i+\lfloor (i-1)u_{m-1,2}/u_{m,1}\rfloor$ for $1\le i\le u_{m,1}$, $m\ge1$ and $v_t=t+\lceil tu_{m,1}/u_{m-1,2}\rceil$ for $1\le t\le u_{m-1,2}$, $m\ge2$.  
In particular, we see that $u_{m-1,2}<u_{m,1}$ implies $D_m$ contains no consecutive vertical edges, while $u_{m,1}<u_{m-1,2}$ implies $D_m$ contains no consecutive horizontal edges.  

For the next result, recall that we work under the assumption $d_1d_2\ge4$.
\begin{corollary}
  \label{cor:edge restrictions}
  For $m\ge2$, the following hold.
  \begin{enumeratea}
    \item $D_m$ contains at most $1+\delta_1$ vertical edges of any given depth.
    \item $D_m$ contains no consecutive horizontal edges if and only if $d_2=1$.
  \end{enumeratea}
\end{corollary}
\begin{proof}
  For $D_2$, both claims are immediate from Corollary~\ref{cor:recursive dyck path structure}(b).
  There are two possibilities for $D_3$.
  If $d_2=1$, the result for $D_2$ together with Corollary~\ref{cor:recursive dyck path structure}(c) shows $D_3$ contains no consecutive horizontal edges and that the vertical edges of $D_3$ all have different depths except $v_{d_1-1}$ and $v_{d_1}$, which both have depth $d_1-1$.
  
  For $d_2>1$, the result for $D_2$ together with Corollary~\ref{cor:recursive dyck path structure}(c) shows all vertical edges of $D_3$ have different depths.
  To see that $d_2>1$ implies there are consecutive horizontal edges in $D_3$ we need to consider two case.
  When $d_1>1$, $D_3$ begins with a copy of $D_2$ by Corollary~\ref{cor:recursive dyck path structure}(c) and thus contains consecutive horizontal edges.
  When $d_1=1$, we must have $d_2\ge4$.
  But then $D_3$ is just $D_2\setminus D_1$ and, since $d_2\ge4$, it contains consecutive horizontal edges. 

  For $m\ge4$, both claims follow by induction using Corollary~\ref{cor:recursive dyck path structure}(c).
\end{proof}
Analogous statements hold for $D'_m$, with horizontal edges $H'_m=\{h'_1,\ldots,h'_{u_{m,2}}\}$ and vertical edges $V'_m=\{v'_1,\ldots,v'_{u_{m-1,1}}\}$, by interchanging the roles of $d_1$ and $d_2$.

The proof of Theorem~\ref{th:combinatorial construction} will go by induction.  Towards this aim we introduce notation, following Corollary~\ref{cor:recursive dyck path structure}, which captures the recursive structure in the edges of $D_m$, $m\ge3$.  
\begin{definition}
  \label{def:subpath edges}
  For $m\ge3$ and $1\le r\le d_m-1-\delta_m$, define the following subsets of $H_m$ and $V_m$: 
  \[H_{m,r}=\{h_{(r-1)u_{m-1,1}+1},h_{(r-1)u_{m-1,1}+2},\ldots,h_{ru_{m-1,1}}\};\]
  \[V_{m,r}=\{v_{(r-1)u_{m-2,2}+1},v_{(r-1)u_{m-2,2}+2},\ldots,v_{ru_{m-2,2}}\};\]
  we identify these, for each $r$, with the horizontal and vertical edges of $D_{m-1}$.  Also set 
  \[H_{m,d_m-\delta_m}=\{h_{(d_m-1-\delta_m)u_{m-1,1}+1},\ldots,h_{u_{m,1}-1},h_{u_{m,1}}\};\]
  \[V_{m,d_m-\delta_m}=\{v_{(d_m-1-\delta_m)u_{m-2,2}+1},\ldots,v_{u_{m-1,2}-1},v_{u_{m-1,2}}\};\]
  we identify these subsets with the horizontal and vertical edges of $D_{m-1}\setminus D_{m-2-\delta_m}$.

  As a notational convenience, for $1\le r\le d_m-1-\delta_m$ and $1\le i\le u_{m-1,1}$ we write $h_{i,r}:=h_{(r-1)u_{m-1,1}+i}$ and similarly $v_{t,r}:=v_{(r-1)u_{m-2,2}+t}$ for $1\le t\le u_{m-2,2}$.  
  For $u_{m-2-\delta_m,1}+1\le i\le u_{m-1,1}$, set $h_{i,d_m-\delta_m}:=h_{(d_m-1-\delta_m)u_{m-1,1}+i-u_{m-2-\delta_m,1}}$ and set $v_{t,d_m-\delta_m}:=v_{(d_m-1-\delta_m)u_{m-2,2}+t-u_{m-3-\delta_m,2}}$ for $u_{m-3-\delta_m,2}+1\le t\le u_{m-2,2}$.
\end{definition}

The next technical result will be useful in the proof of Theorem~\ref{th:blocking edge conditions}.
\begin{lemma}
  \label{le:Dyck tail inequality}
  Assume $m\ge3$, $d_1d_2\ge4$, and $d_1\ge2$.
  Then the terminal $D_{m-1}\setminus D_{m-2-\delta_m}$ subpath of $D_m$ contains at least $d_1$ vertical edges unless one of the following holds:
  \begin{enumerate}
    \item $m=3$ or $m=4$;
    \item $m=5$ with $d_2\le 2$.
    \item $m\ge6$ with $d_1=d_2=2$;
    \item $m\ge6$ is odd with $d_1=4$ and $d_2=1$
  \end{enumerate}
  When $m=4$ above, $D_m$ contains $d_2-1$ horizontal edges of height $u_{m-1,2}-d_1$, namely the edges 
  \[h_{u_{3,1},d_2-1},h_{u_{3,1}-1,d_2-1},\cdots,h_{u_{3,1}-d_2+2,d_2-1}.\]
  In cases (2), (3), or (4) above, $D_m$ contains only one horizontal edge of height $u_{m-1,2}-d_1$, namely the edge $h_{u_{m-1,1},d_m-1-\delta_m}$.
\end{lemma}
\begin{proof}
  For $m\ge3$, the number of vertical edges in $D_{m-1}\setminus D_{m-2-\delta_m}$ is given by $u_{m-2,2}-u_{m-3-\delta_m,2}$.
  For $m=3$, this is $u_{1,2}-u_{0,2}=1<d_1$.

  When $d_1=d_2=2$, this number is always equal to $1<d_1$.
  For $m\ge4$ in this case, the only horizontal edge of height $u_{m-1,2}-d_1$ inside $D_m$ is the horizontal edge of the terminal $D_{m-2}\setminus D_{m-3}$ subpath inside the first $D_{m-1}$ subpath of $D_m$.

  When $m\ge5$ is odd with $d_1=4$ and $d_2=1$, $u_{m-2,2}-u_{m-4,2}$ is always equal to $2<d_1$.
  Since $d_2=1$, $D_m$ contains no consecutive horizontal edges and thus contains at most one horizontal edge of any given height.
  In this case, there is a horizontal edge of height $u_{m-1,2}-d_1$ inside $D_m$, namely the last horizontal edge $h_{u_{m-1,1},d_1-2}$ inside the $(d_1-2)$-nd $D_{m-1}$ subpath of $D_m$ since this horizontal edge is followed by 2 vertical edges.

  In the case $m=4$, we have $u_{2,2}-u_{1,2}=d_1-1$ vertical edges in $D_3\setminus D_2$.
  The horizontal edges here having height $u_{m-1,2}-d_1$ are precisely the $d_2-1$ horizontal edges in the terminal $D_2\setminus D_1$ subpath of the $(d_2-1)$-st copy of $D_3$ inside $D_4$, namely the edges
  \[h_{u_{3,1},d_2-1},h_{u_{3,1}-1,d_2-1},\cdots,h_{u_{3,1}-d_2+2,d_2-1}.\]

  For $m=5$ with $d_2=1$, there are $u_{3,2}-u_{1,2}=(d_1d_2-1)-1=d_1-2<d_1$ vertical edges in $D_4\setminus D_2$.
  In this case, there is a unique horizontal edge of height $u_{m-1,2}-d_1$ which can be described in exactly the same way as the case $d_1=4$, $d_2=1$ above.
  If $m=5$ with $d_2\ge2$, we have $u_{3,2}-u_{2,2}=(d_1d_2-1)-d_1=d_1(d_2-1)-1$ vertical edges in $D_4\setminus D_3$, which is less than $d_1$ if $d_2=2$ and greater than $d_1$ if $d_2>2$ since $d_1\ge2$.
  In the case $d_2=2$ above, the last horizontal edge $h_{u_{4,1},d_1-1}$ of the $(d_1-1)$-st $D_4$ subpath of $D_5$ is immediately followed by a vertical edge and is the only horizontal edge of height $u_{m-1,2}-d_1$.

  For $m\ge6$ with $\{d_1,d_2\}\ne\{2\}$, there are several cases to consider.
  If $m$ is odd and $d_2\ge2$, we have
  \[u_{m-2,2} + u_{m-4,2} = d_1 u_{m-3,1} = d_2 u_{m-3,2} \ge 2 u_{m-3,2},\]
  where the middle equality uses the second identity from Remark~\ref{rem:chebyshev}, which is equivalent to the inequality
  \[u_{m-2,2} - u_{m-3,2} \ge u_{m-3,2} - u_{m-4,2}.\]
  In particular, the claim for $m$ follows from the claim for $m-1$.
  If $m$ is even, we have
  \[u_{m-2,2} + u_{m-4,2} = d_1 u_{m-3,1} = d_1 u_{m-3,2},\]
  where the last equality uses the first identity from Remark~\ref{rem:chebyshev}, which is equivalent to the identity
  \[\tag{$\dagger$}u_{m-2,2} - u_{m-3,2} = (d_1-1) u_{m-3,2} - u_{m-4,2}.\]
  In particular, the claim for $m$ follows from the claim for $m-1$ since $d_1\ge2$.
  For $d_2>2$, we have already seen that the claim holds for $m=5$ and thus it holds for all $m\ge6$.
  If $d_2=2$, we must have $d_1>2$ and it follows from ($\dagger$) that the claim holds for $m=6$ (and hence for all $m\ge6$) since
  \[(d_1-1) u_{3,2} - u_{2,2} > u_{3,2} - u_{2,2} = d_1-1.\]

  For $m\ge6$ with $d_2=1$, we must have $d_1\ge4$.
  Thus when $m$ is even, ($\dagger$) gives
  \[\tag{$\ddagger$}u_{m-2,2} - u_{m-3,2} \ge 3 u_{m-3,2} - u_{m-4,2} > 3 (u_{m-3,2} - u_{m-4,2})\]
  and we see again that the claim for $m$ follows from the claim for $m-1$.
  The claim for $m=6$ (and hence for all even $m\ge6$) also follows from ($\ddagger$) since 
  \[ 3 u_{3,2} - u_{2,2} = 3 (d_1d_2 - 1) - d_1 = 2 d_1 -1 > d_1. \]
  If $m$ is odd, we have
  \[u_{m-2,2} - u_{m-4,2} = u_{m-3,2} - 2 u_{m-4,2} \ge (3 u_{m-4,2} - u_{m-5,2}) - u_{m-4,2} = 2 u_{m-4,2} - u_{m-5,2} = u_{m-4,2} - u_{m-6,2}.\]
  In particular, the claim for $m$ follows from the claim for $m-2$.
  To see the claim for $m=7$ (and hence for all odd $m\ge6$) when $d_1\ge5$, we compute
  \[ u_{5,2} - u_{3,2} = (d_1^2d_2^2 - 3d_1d_2 + 1) - (d_1d_2 - 1) = d_1^2 - 4d_1 + 2 = d_1(d_1-4) + 2 > d_1.\]
  This completes the proof.
\end{proof}

\section{Combinatorics of Compatible Gradings}
\label{sec:compatible pairs}

Let $\omega:E_m\to\ZZ_{\ge0}$ be a $(d_1,d_2)$-bounded grading of $D_m$, $m\ge1$.  
It will be convenient to write $\omega_H$ and $\omega_V$ for the restrictions of $\omega$ to $H_m$ and to $V_m$ respectively.  
In the absence of a total grading $\omega$, we refer to the maps $\omega_H:H_m\to[0,d_1]$ and $\omega_V:V_m\to[0,d_2]$ respectively as \emph{horizontal gradings} and \emph{vertical gradings} of $D_m$.    
We will often consider $\omega$ to be the pair $(\omega_H,\omega_V)$ and refer to $\omega_H$ and $\omega_V$ as being \emph{compatible} if Definition~\ref{def:compatibility} is satisfied for $\omega$.  
Since the first condition~\eqref{eq:hgc} of Definition~\ref{def:compatibility} only involves $\omega_H$, we refer to it as the \emph{horizontal grading condition}.  
Similarly, we refer to the second condition~\eqref{eq:vgc} as the \emph{vertical grading condition}.

Write $\supp(\omega):=\{e\in E_m:\omega(e)\ne0\}$ and call this the \emph{support} of $\omega$.
Set $\supp(\omega_H)=\supp(\omega)\cap H$ and $\supp(\omega_V)=\supp(\omega)\cap V$.
Define $|\omega|_H:=\sum\limits_{h\in H_m} \omega_H(h)$ and $|\omega|_V:=\sum\limits_{v\in V_m} \omega_V(v)$.  

\subsection{Shadow Statistics}
\label{sec:shadows}
To begin we introduce notation to gain a more delicate grasp of the compatibility conditions \eqref{eq:hgc} and \eqref{eq:vgc} from Definition~\ref{def:compatibility}.
For a horizontal grading $\omega_H:H_m\to[0,d_1]$ and any subpath $ee'\subset D_m$, define the \emph{horizontal shadow statistic}
\[f_{\omega_H}(ee'):=-|ee'|_V+\sum\limits_{h\in(ee')_H}\omega_H(h).\]
We also define the \emph{vertical shadow statistic}
\[f_{\omega_V}(ee'):=-|ee'|_H+\sum\limits_{v\in(ee')_V}\omega_V(v)\]
for each vertical grading $\omega_V:V_m\to[0,d_2]$.  
It immediately follows from the definitions that the shadow statistics satisfy the following additivity property with respect to concatenation of paths:
\begin{equation}
  \label{eq:shadow statistic concatenation}
  f_{\omega_H}(e_1e_3)=f_{\omega_H}(e_1e_2)+f_{\omega_H}(\overline{e}_2e_3)\quad\text{and}\quad f_{\omega_V}(e_1e_3)=f_{\omega_V}(e_1e_2)+f_{\omega_V}(\overline{e}_2e_3)
\end{equation}
for edges $e_i\in E_m$ with $e_2\in e_1e_3$.

The shadow statistics give the following alternative check for compatibility, c.f. \cite[Lemma 3.9]{lee-li-zelevinsky}.
\begin{lemma}
  \label{le:compatibility inequality}
  Let $\omega:E_m\to\ZZ_{\ge0}$ be a compatible grading of $D_m$.  For $h\in H_m$ and $v\in V_m$, the following hold:
  \begin{enumeratea}
    \item if $f_{\omega_H}(hv)<0$, then the horizontal grading condition \eqref{eq:hgc} is satisfied for the path $hv$;
    \item if $f_{\omega_V}(hv)<0$, then the vertical grading condition \eqref{eq:vgc} is satisfied for the path $hv$.
  \end{enumeratea}
\end{lemma}
\begin{proof}
  We prove (a), the proof of (b) is similar.
  
  There is nothing to show when $\omega_H(h)=0$, so assume $h\in\supp(\omega_H)$.
  Then $f_{\omega_H}(hh)>0$ and as $e$ ranges from $h$ to $v$ the value of $f_{\omega_H}(he)$ either increases, stays the same, or decreases by 1 with each step.
  Since $f_{\omega_H}(hv)<0$, we see that $f_{\omega_H}(he)$ must eventually take the value 0 with $e\ne v$, i.e.\ the horizontal grading condition is satisfied for the path $hv$.
\end{proof}

Apart from their relationship to the compatibility conditions \eqref{eq:hgc} and \eqref{eq:vgc}, the shadow statistics $f_{\omega_H}$ and $f_{\omega_V}$ encode the following important information.  
For each subpath $ee'\subset D_m$, we obtain a factor $Y_{ee'}(\omega_H,\omega_V)$ of the monomial $Y_{D_m}(\omega_H,\omega_V)$ appearing in equation~\eqref{eq:total path weights} by only multiplying the weights of edges along the path $ee'$.
\begin{lemma}
  \label{le:shadows and degrees}
  The quantities $f_{\omega_H}(ee')$ and $f_{\omega_V}(ee')$ record the total $Y$-degree and the total $X$-degree respectively of the monomial $Y_{ee'}(\omega_H,\omega_V)$.
\end{lemma}
\begin{proof}
  A horizontal edge $h\in(ee')_H$ contributes a factor of $p_{1,\omega_H(h)}Y^{\omega_H(h)}X^{-1}$ to $Y_{ee'}(\omega_H,\omega_V)$ while a vertical edge $v\in(ee')_V$ contributes a factor of $p_{2,d_2-\omega_V(v)}X^{\omega_V(v)+1}Y^{-1}X^{-1}$.  
  The result now follows by comparing the total $Y$- and $X$-degrees of $Y_{ee'}(\omega_H,\omega_V)$ with the definitions of $f_{\omega_H}(ee')$ and $f_{\omega_V}(ee')$ respectively.
\end{proof}

Following \cite[Section 3]{lee-li-zelevinsky}, for a horizontal grading $\omega_H:H_m\to[0,d_1]$, define the \emph{local shadow path} $D(h;\omega_H)$ of a horizontal edge $h\in H_m$ to be the shortest nonempty subpath $he\subset D_m$ such that $f_{\omega_H}(he)=0$, if there is no such subpath we set $D(h;\omega_H)=hv_{u_{m-1,2}}$.  
Write $D_H(h;\omega_H):=D(h;\omega_H)\cap H_m$ and $D_V(h;\omega_H):=D(h;\omega_H)\cap V_m$ for the \emph{local horizontal shadow} and \emph{local vertical shadow} of $h$ with respect to $\omega_H$.  
The local shadow path $D(v;\omega_V)$ is defined similarly for a vertical edge $v\in V_m$ and a vertical grading $\omega_V:V_m\to[0,d_2]$, where $D(v;\omega_V)=h_1v$ if there is no edge $e\le v$ for which $f_{\omega_V}(ev)=0$.
The local shadows $D_H(v;\omega_V)$, $D_V(v;\omega_V)$ are defined as above.  

By definition we have $f_{\omega_H}\big(D(h;\omega_H)\big)=0$ whenever the final edge of $D(h;\omega_H)$ is not $v_{u_{m-1,2}}$.  
More importantly, writing $D(h;\omega_H)=he$, Lemma~\ref{le:compatibility inequality} together with equation~\eqref{eq:shadow statistic concatenation} imply that $f_{\omega_H}(he')>0$ and $f_{\omega_H}(e'e)<0$ for any proper subpaths $he',e'e\subset D(h;\omega_H)$.  
Thus we see for $h\in\supp(\omega_H)$ and $v\in D_V(h;\omega_H)$ that the condition \eqref{eq:hgc} is not satisfied for the path $hv$, however for any $\omega_V$ compatible with $\omega_H$ the condition \eqref{eq:vgc} is satisfied for $h$ and $v$.  
In particular, when $\omega_V$ is compatible with $\omega_H$, $D(v;\omega_V)$ is a proper subpath of $D(h;\omega_H)$ for any $v\in D_V(h;\omega_H)$.

Similar statements hold using the vertical shadow statistic $f_{\omega_V}$.

\subsection{Recursions}
\label{sec:recursions}
We introduce in this section a recursive construction of gradings analogous to the recursive operations on Dyck paths from Lemma~\ref{le:Dyck path recursion}.  
These results are direct generalizations of constructions from \cite[Section 3]{lee-li-zelevinsky}.

The \emph{shadow} of a horizontal grading $\omega_H:H_m\to[0,d_1]$ is the collection of vertical edges in the local vertical shadows of all horizontal edges, i.e.\ $\sh(\omega_H)=\bigcup_{h\in H_m} D_V(h;\omega_H)$.  
The \emph{remote shadow} of a horizontal grading $\omega_H:H_m\to[0,d_1]$ is the subset $\rsh(\omega_H)\subset\sh(\omega_H)$ obtained by excluding for each $d$ the (up to) $\omega_H(h_d)$ vertical edges of depth $d$ immediately following $h_d$.  
The shadow and remote shadow of a vertical grading $\omega_V:V_m\to[0,d_2]$ are defined similarly.
\begin{remark}
  \label{rem:remote shadows}
  The remote shadow $\rsh(\omega_H)\subset\sh(\omega_H)$ of a horizontal grading $\omega_H$ can be described as the subset consisting of those vertical edges $v\in\sh(\omega_H)$ for which there exists a vertical grading $\omega_V$ compatible with $\omega_H$ such that $\omega_V(v)>0$.
  In particular, any vertical grading $\omega_V$ compatible with $\omega_H$ must satisfy $\omega_V(v)=0$ for $v\in\sh(\omega_H)\setminus\rsh(\omega_H)$.
\end{remark}

\begin{example}
  In Example~\ref{ex:combinatorial cluster variables}, the gradings are organized according to their associated vertical gradings~$\omega_V$.
  The dashed horizontal edges are precisely those lying outside the shadow $\sh(\omega_V)$ while horizontal edges with a restriction on their weights comprise the remote shadow $\rsh(\omega_V)$.
\end{example}

In order to give a relationship between gradings of $D_m$ and gradings of $D'_{m+1}$, we need to partition the remote shadows according to which local shadow contains a given edge. 
\begin{definition}
  \label{def:remote_shadows}
  Let $\omega:E_m\to\ZZ_{\ge0}$ be a grading of $D_m$.
  \begin{enumeratea}
    \item For $1\le j\le d\le u_{m,1}$, denote by $\rsh(\omega_H)_{j;d}$ the set of $v\in\rsh(\omega_H)$ of depth $d$ such that $v\in D_V(h_j;\omega_H)$ and $h_j$ is the first horizontal edge before $v$ with this property.  
    Define the \emph{local remote shadow} of the edge $h_j$ as $\rsh(h_j;\omega_H):=\coprod\limits_{d\in[j+1,u_{m,1}]}\rsh(\omega_H)_{j;d}$.
    \item For $0\le\ell<t\le u_{m-1,2}$, denote by $\rsh(\omega_V)_{t;\ell}$ the set of $h\in\rsh(\omega_V)$ of height $\ell$ such that $h\in D_H(v_t;\omega_V)$ and $v_t$ is the first vertical edge after $h$ with this property.  
    Define the \emph{local remote shadow} of the edge $v_t$ as $\rsh(v_t;\omega_V):=\coprod\limits_{\ell\in[0,t-2]}\rsh(\omega_V)_{t;\ell}$.
  \end{enumeratea}
\end{definition}
\begin{remark}
  By the definition of the remote shadows, it is impossible to have $d=j$ or $\ell=t-1$ in Definition~\ref{def:remote_shadows}.
\end{remark}

Lemma~\ref{le:Dyck path recursion} establishes a canonical order preserving bijection between the vertical edges $V'_{m+1}$ of $D'_{m+1}$ and the horizontal edges $H_m$ of $D_m$ which we write as $\varphi=\varphi_m:V'_{m+1}\to H_m$, $\varphi(v'_i)=h_i$ for $1\le i\le u_{m,1}$.  
Thus we obtain a bijection from $d_1$-bounded horizontal gradings of $D_m$ to $d_1$-bounded vertical gradings of $D'_{m+1}$ taking a horizontal grading $\omega_H:H_m\to[0,d_1]$ to the vertical grading $\varphi^*\omega_H:V'_{m+1}\to[0,d_1]$ given by $\varphi^*\omega_H(v'_i)=d_1-\omega_H(h_i)$.
\begin{remark}
  We will abuse notation slightly and also write $\varphi^*_m$ for the bijection between horizontal gradings of $D'_m$ and vertical gradings of $D_{m+1}$ where the roles of $d_1$ and $d_2$ need to be interchanged in the definitions above, however this abuse should not lead to any confusion.
\end{remark}

The next result shows that the remote shadows for $\omega_H$ and $\varphi^*\omega_H$ are intimately related.
\begin{lemma}\cite[Corollary~4.18]{rupel2}
  \label{le:remote shadow cardinalities}
  Let $\omega_H:H_m\to[0,d_1]$ be a horizontal grading of $D_m$.
  For $1\le j<d\le u_{m,1}$, we have $|\rsh(\omega_H)_{j;d}|=|\rsh(\varphi^*\omega_H)_{d;j-1}|$.
\end{lemma}

Thus for $1\le j<d\le u_{m,1}$ we may define a bijection $\theta_{j;d}:\rsh(\omega_H)_{j;d}\to\rsh(\varphi^*\omega_H)_{d;j-1}$ which preserves the natural order determined by distance from $h_j$ and from $v'_d$ respectively.  More explicitly, as the vertical edges of $\rsh(\omega_H)_{j;d}$ are read from bottom to top the corresponding horizontal edges of $\rsh(\varphi^*\omega_H)_{d;j-1}$ are read from right to left.

For a horizontal grading $\omega_H:H_m\to[0,d_1]$, write $\cG(\omega_H)$ for the collection of all $(d_1,d_2)$-bounded gradings $\omega:E_m\to\ZZ_{\ge0}$ such that the restriction $\omega|_{H_m}$ is precisely $\omega_H$ and denote by $\cC(\omega_H)\subset\cG(\omega_H)$ the subset of compatible gradings.
Let $\cG_\rsh(\omega_H)\subset\cG(\omega_H)$ denote those gradings $\omega$ for which $\omega(v)=0$ whenever $v\in V_m\setminus\rsh(\omega_H)$ and write $\cC_\rsh(\omega_H):=\cG_\rsh(\omega_H)\cap\cC(\omega_H)$.
Define analogous collections of gradings associated to a vertical grading $\omega_V:V_m\to[0,d_2]$.

Define a map $\Omega=\Omega_m:\cG_\rsh(\omega_H)\to\cG_\rsh(\varphi^*\omega_H)$ as follows:
\[\Omega(\omega_V)(h')=
  \begin{cases}
    0 & \text{ if $h'\in H'_{m+1}\setminus\rsh(\varphi^*\omega_H)$;}\\
    \omega_V(v) & \text{ if $h'=\theta_{j;d}(v)$ for $v\in\rsh(\omega_H)_{j;d}$.}
  \end{cases}\]
Note that $\Omega$ admits an obvious inverse map.  
\begin{remark}
  Given a grading $\omega:E_m\to\ZZ_{\ge0}$ of $D_m$ where $\omega_V\notin\cG_\rsh(\omega_H)$, the map $\Omega$ may still be applied to $\omega_V$ to produce a horizontal grading in $\cG_\rsh(\varphi^*\omega_H)$.  
  This observation will be used without mention in the statements of Lemma~\ref{le:f_and_Omega} and Proposition~\ref{prop:piecewise equivalence} as well as in the proof of Corollary~\ref{cor:support and remote shadow}.
\end{remark}

The following result shows that we have some control over the shadow statistics under this operation.
It is also the essential ingredient for understanding the piecewise compatible gradings introduced in the next section.
\begin{lemma}\cite[Lemma~4.19]{rupel2}
  \label{le:f_and_Omega}
  Let $\omega:E_m\to\ZZ_{\ge0}$ be a grading on $D_m$.
  Suppose $h'=\theta_{j;d}(v)$ for a vertical edge $v\in\rsh(\omega_H)_{j;d}\cap\supp(\omega_V)$. Then $f_{\Omega(\omega_V)}(h'v'_d)=f_{\omega_V}(h_jv)$.
\end{lemma}
This crucial result also shows that $\Omega$ restricts to a map $\cC_\rsh(\omega_H)\to\cC_\rsh(\varphi^*\omega_H)$, i.e.\ that the pair $(\Omega(\omega_V),\varphi^*\omega_H)$ gives a compatible grading of $D'_{m+1}$ exactly when $\omega_V\in\cC_\rsh(\omega_H)$.
\begin{proposition}\cite[Lemma~4.20]{rupel2}
  \label{prop:compatibility equivalence}
  Let $\omega_H:H_m\to[0,d_1]$ be a horizontal grading of $D_m$.  
  For a vertical grading $\omega_V\in\cG_\rsh(\omega_H)$, we have $\omega_V\in\cC_\rsh(\omega_H)$ if and only if $\Omega(\omega_V)\in\cC_\rsh(\varphi^*\omega_H)$.
\end{proposition}

\begin{example}
  The compatible gradings (organized by their associated horizontal gradings $\omega_H$) of the maximal Dyck path $D'_2=$ 
    \raisebox{-0.5em}{\begin{tikzpicture}
      \begin{scope}[scale=0.5]
      \draw[step=1.0, black, thin] (0,0) grid (3,1);
      \draw[black, thin] (0,0) -- (3,1);
      \draw[black, fill=black] (0,0) circle (3pt);
      \draw[black, fill=black] (1,0) circle (3pt);
      \draw[black, fill=black] (2,0) circle (3pt);
      \draw[black, fill=black] (3,0) circle (3pt);
      \draw[black, fill=black] (3,1) circle (3pt);
      \end{scope}
    \end{tikzpicture}}
    from Example~\ref{ex:dyck paths} are given by
  
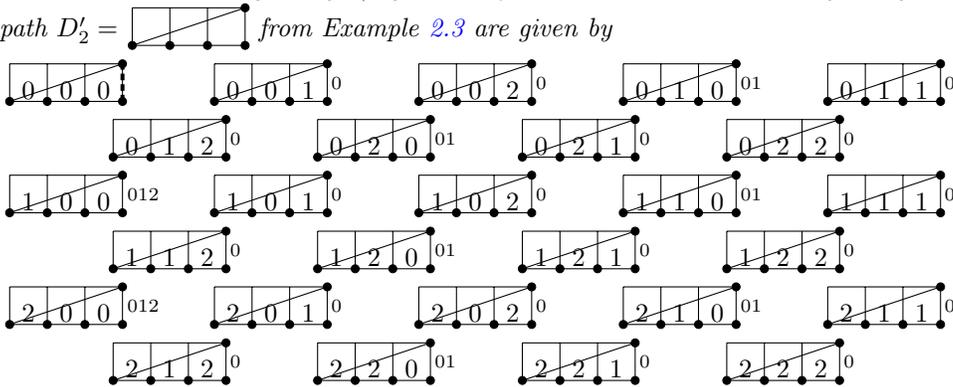
\begin{align*}
    \begin{tikzpicture}
        \begin{scope}[scale=0.5]
          \useasboundingbox (0,0) rectangle (4,1);
        \draw[step=1.0, black, thin] (0,0) grid (3,1);
        \draw[black, thin] (0,0) -- (3,1);
        \draw[black, fill=black] (0,0) circle (3pt);
        \node at (0.5,0.3) {$0$};
        \draw[black, fill=black] (1,0) circle (3pt);
        \node at (1.5,0.3) {$0$};
        \draw[black, fill=black] (2,0) circle (3pt);
        \node at (2.5,0.3) {$0$};
        \draw[black, fill=black] (3,0) circle (3pt);
        \draw[black, line width=1.5pt, dash pattern=on 2.5pt off 1.5pt] (3,0) -- (3,1);
        \draw[black, fill=black] (3,1) circle (3pt);
        \end{scope}
      \end{tikzpicture}
      \qquad
      \begin{tikzpicture}
        \begin{scope}[scale=0.5]
          \useasboundingbox (0,0) rectangle (4,1);
        \draw[step=1.0, black, thin] (0,0) grid (3,1);
        \draw[black, thin] (0,0) -- (3,1);
        \draw[black, fill=black] (0,0) circle (3pt);
        \node at (0.5,0.3) {$0$};
        \draw[black, fill=black] (1,0) circle (3pt);
        \node at (1.5,0.3) {$0$};
        \draw[black, fill=black] (2,0) circle (3pt);
        \node at (2.5,0.3) {$1$};
        \draw[black, fill=black] (3,0) circle (3pt);
        \node at (3.25,0.5) {\scriptsize $0$};
        \draw[black, fill=black] (3,1) circle (3pt);
        \end{scope}
      \end{tikzpicture}
      \qquad
      \begin{tikzpicture}
        \begin{scope}[scale=0.5]
          \useasboundingbox (0,0) rectangle (4,1);
        \draw[step=1.0, black, thin] (0,0) grid (3,1);
        \draw[black, thin] (0,0) -- (3,1);
        \draw[black, fill=black] (0,0) circle (3pt);
        \node at (0.5,0.3) {$0$};
        \draw[black, fill=black] (1,0) circle (3pt);
        \node at (1.5,0.3) {$0$};
        \draw[black, fill=black] (2,0) circle (3pt);
        \node at (2.5,0.3) {$2$};
        \draw[black, fill=black] (3,0) circle (3pt);
        \node at (3.25,0.5) {\scriptsize $0$};
        \draw[black, fill=black] (3,1) circle (3pt);
        \end{scope}
      \end{tikzpicture}
      \qquad
      \begin{tikzpicture}
        \begin{scope}[scale=0.5]
          \useasboundingbox (0,0) rectangle (4,1);
        \draw[step=1.0, black, thin] (0,0) grid (3,1);
        \draw[black, thin] (0,0) -- (3,1);
        \draw[black, fill=black] (0,0) circle (3pt);
        \node at (0.5,0.3) {$0$};
        \draw[black, fill=black] (1,0) circle (3pt);
        \node at (1.5,0.3) {$1$};
        \draw[black, fill=black] (2,0) circle (3pt);
        \node at (2.5,0.3) {$0$};
        \draw[black, fill=black] (3,0) circle (3pt);
        \node at (3.4,0.5) {\scriptsize $01$};
        \draw[black, fill=black] (3,1) circle (3pt);
        \end{scope}
      \end{tikzpicture}
      \qquad
      \begin{tikzpicture}
        \begin{scope}[scale=0.5]
          \useasboundingbox (0,0) rectangle (4,1);
        \draw[step=1.0, black, thin] (0,0) grid (3,1);
        \draw[black, thin] (0,0) -- (3,1);
        \draw[black, fill=black] (0,0) circle (3pt);
        \node at (0.5,0.3) {$0$};
        \draw[black, fill=black] (1,0) circle (3pt);
        \node at (1.5,0.3) {$1$};
        \draw[black, fill=black] (2,0) circle (3pt);
        \node at (2.5,0.3) {$1$};
        \draw[black, fill=black] (3,0) circle (3pt);
        \node at (3.25,0.5) {\scriptsize $0$};
        \draw[black, fill=black] (3,1) circle (3pt);
        \end{scope}
      \end{tikzpicture}\\
      \begin{tikzpicture}
        \begin{scope}[scale=0.5]
          \useasboundingbox (0,0) rectangle (4,1);
        \draw[step=1.0, black, thin] (0,0) grid (3,1);
        \draw[black, thin] (0,0) -- (3,1);
        \draw[black, fill=black] (0,0) circle (3pt);
        \node at (0.5,0.3) {$0$};
        \draw[black, fill=black] (1,0) circle (3pt);
        \node at (1.5,0.3) {$1$};
        \draw[black, fill=black] (2,0) circle (3pt);
        \node at (2.5,0.3) {$2$};
        \draw[black, fill=black] (3,0) circle (3pt);
        \node at (3.25,0.5) {\scriptsize $0$};
        \draw[black, fill=black] (3,1) circle (3pt);
        \end{scope}
      \end{tikzpicture}
      \qquad
      \begin{tikzpicture}
        \begin{scope}[scale=0.5]
          \useasboundingbox (0,0) rectangle (4,1);
        \draw[step=1.0, black, thin] (0,0) grid (3,1);
        \draw[black, thin] (0,0) -- (3,1);
        \draw[black, fill=black] (0,0) circle (3pt);
        \node at (0.5,0.3) {$0$};
        \draw[black, fill=black] (1,0) circle (3pt);
        \node at (1.5,0.3) {$2$};
        \draw[black, fill=black] (2,0) circle (3pt);
        \node at (2.5,0.3) {$0$};
        \draw[black, fill=black] (3,0) circle (3pt);
        \node at (3.4,0.5) {\scriptsize $01$};
        \draw[black, fill=black] (3,1) circle (3pt);
        \end{scope}
      \end{tikzpicture}
      \qquad
      \begin{tikzpicture}
        \begin{scope}[scale=0.5]
          \useasboundingbox (0,0) rectangle (4,1);
        \draw[step=1.0, black, thin] (0,0) grid (3,1);
        \draw[black, thin] (0,0) -- (3,1);
        \draw[black, fill=black] (0,0) circle (3pt);
        \node at (0.5,0.3) {$0$};
        \draw[black, fill=black] (1,0) circle (3pt);
        \node at (1.5,0.3) {$2$};
        \draw[black, fill=black] (2,0) circle (3pt);
        \node at (2.5,0.3) {$1$};
        \draw[black, fill=black] (3,0) circle (3pt);
        \node at (3.25,0.5) {\scriptsize $0$};
        \draw[black, fill=black] (3,1) circle (3pt);
        \end{scope}
      \end{tikzpicture}
      \qquad
      \begin{tikzpicture}
        \begin{scope}[scale=0.5]
          \useasboundingbox (0,0) rectangle (4,1);
        \draw[step=1.0, black, thin] (0,0) grid (3,1);
        \draw[black, thin] (0,0) -- (3,1);
        \draw[black, fill=black] (0,0) circle (3pt);
        \node at (0.5,0.3) {$0$};
        \draw[black, fill=black] (1,0) circle (3pt);
        \node at (1.5,0.3) {$2$};
        \draw[black, fill=black] (2,0) circle (3pt);
        \node at (2.5,0.3) {$2$};
        \draw[black, fill=black] (3,0) circle (3pt);
        \node at (3.25,0.5) {\scriptsize $0$};
        \draw[black, fill=black] (3,1) circle (3pt);
        \end{scope}
      \end{tikzpicture}
      \qquad
      \begin{tikzpicture}
        \begin{scope}[scale=0.25]
          \useasboundingbox (0,0) rectangle (2.5,1);
        \end{scope}
      \end{tikzpicture}\\
    \begin{tikzpicture}
        \begin{scope}[scale=0.5]
          \useasboundingbox (0,0) rectangle (4,1);
        \draw[step=1.0, black, thin] (0,0) grid (3,1);
        \draw[black, thin] (0,0) -- (3,1);
        \draw[black, fill=black] (0,0) circle (3pt);
        \node at (0.5,0.3) {$1$};
        \draw[black, fill=black] (1,0) circle (3pt);
        \node at (1.5,0.3) {$0$};
        \draw[black, fill=black] (2,0) circle (3pt);
        \node at (2.5,0.3) {$0$};
        \draw[black, fill=black] (3,0) circle (3pt);
        \node at (3.55,0.5) {\scriptsize $012$};
        \draw[black, fill=black] (3,1) circle (3pt);
        \end{scope}
      \end{tikzpicture}
      \qquad
      \begin{tikzpicture}
        \begin{scope}[scale=0.5]
          \useasboundingbox (0,0) rectangle (4,1);
        \draw[step=1.0, black, thin] (0,0) grid (3,1);
        \draw[black, thin] (0,0) -- (3,1);
        \draw[black, fill=black] (0,0) circle (3pt);
        \node at (0.5,0.3) {$1$};
        \draw[black, fill=black] (1,0) circle (3pt);
        \node at (1.5,0.3) {$0$};
        \draw[black, fill=black] (2,0) circle (3pt);
        \node at (2.5,0.3) {$1$};
        \draw[black, fill=black] (3,0) circle (3pt);
        \node at (3.25,0.5) {\scriptsize $0$};
        \draw[black, fill=black] (3,1) circle (3pt);
        \end{scope}
      \end{tikzpicture}
      \qquad
      \begin{tikzpicture}
        \begin{scope}[scale=0.5]
          \useasboundingbox (0,0) rectangle (4,1);
        \draw[step=1.0, black, thin] (0,0) grid (3,1);
        \draw[black, thin] (0,0) -- (3,1);
        \draw[black, fill=black] (0,0) circle (3pt);
        \node at (0.5,0.3) {$1$};
        \draw[black, fill=black] (1,0) circle (3pt);
        \node at (1.5,0.3) {$0$};
        \draw[black, fill=black] (2,0) circle (3pt);
        \node at (2.5,0.3) {$2$};
        \draw[black, fill=black] (3,0) circle (3pt);
        \node at (3.25,0.5) {\scriptsize $0$};
        \draw[black, fill=black] (3,1) circle (3pt);
        \end{scope}
      \end{tikzpicture}
      \qquad
      \begin{tikzpicture}
        \begin{scope}[scale=0.5]
          \useasboundingbox (0,0) rectangle (4,1);
        \draw[step=1.0, black, thin] (0,0) grid (3,1);
        \draw[black, thin] (0,0) -- (3,1);
        \draw[black, fill=black] (0,0) circle (3pt);
        \node at (0.5,0.3) {$1$};
        \draw[black, fill=black] (1,0) circle (3pt);
        \node at (1.5,0.3) {$1$};
        \draw[black, fill=black] (2,0) circle (3pt);
        \node at (2.5,0.3) {$0$};
        \draw[black, fill=black] (3,0) circle (3pt);
        \node at (3.4,0.5) {\scriptsize $01$};
        \draw[black, fill=black] (3,1) circle (3pt);
        \end{scope}
      \end{tikzpicture}
      \qquad
      \begin{tikzpicture}
        \begin{scope}[scale=0.5]
          \useasboundingbox (0,0) rectangle (4,1);
        \draw[step=1.0, black, thin] (0,0) grid (3,1);
        \draw[black, thin] (0,0) -- (3,1);
        \draw[black, fill=black] (0,0) circle (3pt);
        \node at (0.5,0.3) {$1$};
        \draw[black, fill=black] (1,0) circle (3pt);
        \node at (1.5,0.3) {$1$};
        \draw[black, fill=black] (2,0) circle (3pt);
        \node at (2.5,0.3) {$1$};
        \draw[black, fill=black] (3,0) circle (3pt);
        \node at (3.25,0.5) {\scriptsize $0$};
        \draw[black, fill=black] (3,1) circle (3pt);
        \end{scope}
      \end{tikzpicture}\\
      \begin{tikzpicture}
        \begin{scope}[scale=0.5]
          \useasboundingbox (0,0) rectangle (4,1);
        \draw[step=1.0, black, thin] (0,0) grid (3,1);
        \draw[black, thin] (0,0) -- (3,1);
        \draw[black, fill=black] (0,0) circle (3pt);
        \node at (0.5,0.3) {$1$};
        \draw[black, fill=black] (1,0) circle (3pt);
        \node at (1.5,0.3) {$1$};
        \draw[black, fill=black] (2,0) circle (3pt);
        \node at (2.5,0.3) {$2$};
        \draw[black, fill=black] (3,0) circle (3pt);
        \node at (3.25,0.5) {\scriptsize $0$};
        \draw[black, fill=black] (3,1) circle (3pt);
        \end{scope}
      \end{tikzpicture}
      \qquad
      \begin{tikzpicture}
        \begin{scope}[scale=0.5]
          \useasboundingbox (0,0) rectangle (4,1);
        \draw[step=1.0, black, thin] (0,0) grid (3,1);
        \draw[black, thin] (0,0) -- (3,1);
        \draw[black, fill=black] (0,0) circle (3pt);
        \node at (0.5,0.3) {$1$};
        \draw[black, fill=black] (1,0) circle (3pt);
        \node at (1.5,0.3) {$2$};
        \draw[black, fill=black] (2,0) circle (3pt);
        \node at (2.5,0.3) {$0$};
        \draw[black, fill=black] (3,0) circle (3pt);
        \node at (3.4,0.5) {\scriptsize $01$};
        \draw[black, fill=black] (3,1) circle (3pt);
        \end{scope}
      \end{tikzpicture}
      \qquad
      \begin{tikzpicture}
        \begin{scope}[scale=0.5]
          \useasboundingbox (0,0) rectangle (4,1);
        \draw[step=1.0, black, thin] (0,0) grid (3,1);
        \draw[black, thin] (0,0) -- (3,1);
        \draw[black, fill=black] (0,0) circle (3pt);
        \node at (0.5,0.3) {$1$};
        \draw[black, fill=black] (1,0) circle (3pt);
        \node at (1.5,0.3) {$2$};
        \draw[black, fill=black] (2,0) circle (3pt);
        \node at (2.5,0.3) {$1$};
        \draw[black, fill=black] (3,0) circle (3pt);
        \node at (3.25,0.5) {\scriptsize $0$};
        \draw[black, fill=black] (3,1) circle (3pt);
        \end{scope}
      \end{tikzpicture}
      \qquad
      \begin{tikzpicture}
        \begin{scope}[scale=0.5]
          \useasboundingbox (0,0) rectangle (4,1);
        \draw[step=1.0, black, thin] (0,0) grid (3,1);
        \draw[black, thin] (0,0) -- (3,1);
        \draw[black, fill=black] (0,0) circle (3pt);
        \node at (0.5,0.3) {$1$};
        \draw[black, fill=black] (1,0) circle (3pt);
        \node at (1.5,0.3) {$2$};
        \draw[black, fill=black] (2,0) circle (3pt);
        \node at (2.5,0.3) {$2$};
        \draw[black, fill=black] (3,0) circle (3pt);
        \node at (3.25,0.5) {\scriptsize $0$};
        \draw[black, fill=black] (3,1) circle (3pt);
        \end{scope}
      \end{tikzpicture}
      \qquad
      \begin{tikzpicture}
        \begin{scope}[scale=0.25]
          \useasboundingbox (0,0) rectangle (2.5,1);
        \end{scope}
      \end{tikzpicture}\\
    \begin{tikzpicture}
        \begin{scope}[scale=0.5]
          \useasboundingbox (0,0) rectangle (4,1);
        \draw[step=1.0, black, thin] (0,0) grid (3,1);
        \draw[black, thin] (0,0) -- (3,1);
        \draw[black, fill=black] (0,0) circle (3pt);
        \node at (0.5,0.3) {$2$};
        \draw[black, fill=black] (1,0) circle (3pt);
        \node at (1.5,0.3) {$0$};
        \draw[black, fill=black] (2,0) circle (3pt);
        \node at (2.5,0.3) {$0$};
        \draw[black, fill=black] (3,0) circle (3pt);
        \node at (3.55,0.5) {\scriptsize $012$};
        \draw[black, fill=black] (3,1) circle (3pt);
        \end{scope}
      \end{tikzpicture}
      \qquad
      \begin{tikzpicture}
        \begin{scope}[scale=0.5]
          \useasboundingbox (0,0) rectangle (4,1);
        \draw[step=1.0, black, thin] (0,0) grid (3,1);
        \draw[black, thin] (0,0) -- (3,1);
        \draw[black, fill=black] (0,0) circle (3pt);
        \node at (0.5,0.3) {$2$};
        \draw[black, fill=black] (1,0) circle (3pt);
        \node at (1.5,0.3) {$0$};
        \draw[black, fill=black] (2,0) circle (3pt);
        \node at (2.5,0.3) {$1$};
        \draw[black, fill=black] (3,0) circle (3pt);
        \node at (3.25,0.5) {\scriptsize $0$};
        \draw[black, fill=black] (3,1) circle (3pt);
        \end{scope}
      \end{tikzpicture}
      \qquad
      \begin{tikzpicture}
        \begin{scope}[scale=0.5]
          \useasboundingbox (0,0) rectangle (4,1);
        \draw[step=1.0, black, thin] (0,0) grid (3,1);
        \draw[black, thin] (0,0) -- (3,1);
        \draw[black, fill=black] (0,0) circle (3pt);
        \node at (0.5,0.3) {$2$};
        \draw[black, fill=black] (1,0) circle (3pt);
        \node at (1.5,0.3) {$0$};
        \draw[black, fill=black] (2,0) circle (3pt);
        \node at (2.5,0.3) {$2$};
        \draw[black, fill=black] (3,0) circle (3pt);
        \node at (3.25,0.5) {\scriptsize $0$};
        \draw[black, fill=black] (3,1) circle (3pt);
        \end{scope}
      \end{tikzpicture}
      \qquad
      \begin{tikzpicture}
        \begin{scope}[scale=0.5]
          \useasboundingbox (0,0) rectangle (4,1);
        \draw[step=1.0, black, thin] (0,0) grid (3,1);
        \draw[black, thin] (0,0) -- (3,1);
        \draw[black, fill=black] (0,0) circle (3pt);
        \node at (0.5,0.3) {$2$};
        \draw[black, fill=black] (1,0) circle (3pt);
        \node at (1.5,0.3) {$1$};
        \draw[black, fill=black] (2,0) circle (3pt);
        \node at (2.5,0.3) {$0$};
        \draw[black, fill=black] (3,0) circle (3pt);
        \node at (3.4,0.5) {\scriptsize $01$};
        \draw[black, fill=black] (3,1) circle (3pt);
        \end{scope}
      \end{tikzpicture}
      \qquad
      \begin{tikzpicture}
        \begin{scope}[scale=0.5]
          \useasboundingbox (0,0) rectangle (4,1);
        \draw[step=1.0, black, thin] (0,0) grid (3,1);
        \draw[black, thin] (0,0) -- (3,1);
        \draw[black, fill=black] (0,0) circle (3pt);
        \node at (0.5,0.3) {$2$};
        \draw[black, fill=black] (1,0) circle (3pt);
        \node at (1.5,0.3) {$1$};
        \draw[black, fill=black] (2,0) circle (3pt);
        \node at (2.5,0.3) {$1$};
        \draw[black, fill=black] (3,0) circle (3pt);
        \node at (3.25,0.5) {\scriptsize $0$};
        \draw[black, fill=black] (3,1) circle (3pt);
        \end{scope}
      \end{tikzpicture}\\
      \begin{tikzpicture}
        \begin{scope}[scale=0.5]
          \useasboundingbox (0,0) rectangle (4,1);
        \draw[step=1.0, black, thin] (0,0) grid (3,1);
        \draw[black, thin] (0,0) -- (3,1);
        \draw[black, fill=black] (0,0) circle (3pt);
        \node at (0.5,0.3) {$2$};
        \draw[black, fill=black] (1,0) circle (3pt);
        \node at (1.5,0.3) {$1$};
        \draw[black, fill=black] (2,0) circle (3pt);
        \node at (2.5,0.3) {$2$};
        \draw[black, fill=black] (3,0) circle (3pt);
        \node at (3.25,0.5) {\scriptsize $0$};
        \draw[black, fill=black] (3,1) circle (3pt);
        \end{scope}
      \end{tikzpicture}
      \qquad
      \begin{tikzpicture}
        \begin{scope}[scale=0.5]
          \useasboundingbox (0,0) rectangle (4,1);
        \draw[step=1.0, black, thin] (0,0) grid (3,1);
        \draw[black, thin] (0,0) -- (3,1);
        \draw[black, fill=black] (0,0) circle (3pt);
        \node at (0.5,0.3) {$2$};
        \draw[black, fill=black] (1,0) circle (3pt);
        \node at (1.5,0.3) {$2$};
        \draw[black, fill=black] (2,0) circle (3pt);
        \node at (2.5,0.3) {$0$};
        \draw[black, fill=black] (3,0) circle (3pt);
        \node at (3.4,0.5) {\scriptsize $01$};
        \draw[black, fill=black] (3,1) circle (3pt);
        \end{scope}
      \end{tikzpicture}
      \qquad
      \begin{tikzpicture}
        \begin{scope}[scale=0.5]
          \useasboundingbox (0,0) rectangle (4,1);
        \draw[step=1.0, black, thin] (0,0) grid (3,1);
        \draw[black, thin] (0,0) -- (3,1);
        \draw[black, fill=black] (0,0) circle (3pt);
        \node at (0.5,0.3) {$2$};
        \draw[black, fill=black] (1,0) circle (3pt);
        \node at (1.5,0.3) {$2$};
        \draw[black, fill=black] (2,0) circle (3pt);
        \node at (2.5,0.3) {$1$};
        \draw[black, fill=black] (3,0) circle (3pt);
        \node at (3.25,0.5) {\scriptsize $0$};
        \draw[black, fill=black] (3,1) circle (3pt);
        \end{scope}
      \end{tikzpicture}
      \qquad
      \begin{tikzpicture}
        \begin{scope}[scale=0.5]
          \useasboundingbox (0,0) rectangle (4,1);
        \draw[step=1.0, black, thin] (0,0) grid (3,1);
        \draw[black, thin] (0,0) -- (3,1);
        \draw[black, fill=black] (0,0) circle (3pt);
        \node at (0.5,0.3) {$2$};
        \draw[black, fill=black] (1,0) circle (3pt);
        \node at (1.5,0.3) {$2$};
        \draw[black, fill=black] (2,0) circle (3pt);
        \node at (2.5,0.3) {$2$};
        \draw[black, fill=black] (3,0) circle (3pt);
        \node at (3.25,0.5) {\scriptsize $0$};
        \draw[black, fill=black] (3,1) circle (3pt);
        \end{scope}
      \end{tikzpicture}
      \qquad
      \begin{tikzpicture}
        \begin{scope}[scale=0.25]
          \useasboundingbox (0,0) rectangle (2.5,1);
        \end{scope}
      \end{tikzpicture}
    \end{align*}
    The corresponding images of these under the map $\Omega$ can be seen at the end of Example~\ref{ex:combinatorial cluster variables}.
\end{example}

\subsection{Piecewise Compatibility}
\label{sec:piecewise compatibility}
Our goal in this section is to understand which gradings on $D_m$, $m\ge3$, are obtained by gluing together compatible gradings on the $D_{m-1}$ subpaths of $D_m$ found in Corollary~\ref{cor:recursive dyck path structure}(c).
\begin{definition}
  \label{def:agregate grading}
  Fix $m\ge3$.  
  Consider $(d_1,d_2)$-bounded compatible gradings $\omega_r:=(\omega_{H,r},\omega_{V,r})$ of $D_{m-1}$ for $1\le r\le d_m-\delta_m$.
  We assume
  \begin{equation}
    \label{eq:V restriction}
    \omega_{V,d_m-\delta_m}(v)=0
  \end{equation}
  for $v$ in the first $D_{m-2-\delta_m}$ subpath of $D_{m-1}$ and
  \begin{equation}
    \label{eq:H restriction}
    \omega_{H,d_m-\delta_m}(h)=\ell
  \end{equation}
  for $h$ in the first $D_{m-2-\delta_m}$ subpath of $D_{m-1}$ if $h$ is immediately followed by exactly $\ell$ vertical edges inside $D_{m-2-\delta_m}$.  

  Define a grading $\omega:E_m\to\ZZ_{\ge0}$ of $D_m$ by
  \[\omega(e)=\begin{cases}\omega_{H,r}(e) & \text{if $e\in H_{m,r}$;}\\ \omega_{V,r}(e) & \text{if $e\in V_{m,r}$;}\end{cases}\]
  where we identify subsets of edges in $D_m$ with edges of its $D_{m-1}$ subpaths as in Definition~\ref{def:subpath edges}.
  We will refer to any grading on $D_m$ obtained in this way as \emph{piecewise compatible}.
\end{definition}
\begin{remark}
  \label{rem:restricted gradings}
  Every compatible grading of $D_m$, $m\ge3$, is piecewise compatible.  
  Given any grading $\omega$ of $D_m$ and $1\le r\le d_m-\delta_m$, we will denote by $\omega_r=(\omega_{H,r},\omega_{V,r})$ the grading of $D_{m-1}$ obtained by restricting $\omega$ to the $r$-th copy of $D_{m-1}$ inside $D_m$, where $\omega_{d_m-\delta_m}=(\omega_{H,d_m-\delta_m},\omega_{V,d_m-\delta_m})$ denotes the grading on $D_{m-1}$ satisfying the conditions~\eqref{eq:V restriction} and~\eqref{eq:H restriction} of Definition~\ref{def:agregate grading}.
\end{remark}
\begin{remark}
  \label{rem:alternate restrictions}
  When considering piecewise compatible gradings $\omega:E'_m\to\ZZ_{\ge0}$ of $D'_m$, we will instead make the following assumptions on the gradings $\omega_{H',d'_m-\delta'_m}$ and $\omega_{V',d'_m-\delta'_m}$ of $D'_{m-1}$:
  \[\omega_{H',d'_m-\delta'_m}(h')=0\]
  for $h'$ in the first $D'_{m-2-\delta'_m}$ subpath of $D'_{m-1}$ and
  \[\omega_{V',d'_m-\delta'_m}(v')=d\]
  for $v'$ in the first $D'_{m-2-\delta'_m}$ subpath of $D'_{m-1}$ if $v'$ is immediately preceded by exactly $d$ horizontal edges inside $D'_{m-2-\delta'_m}$.  
\end{remark}
\begin{example}
  \label{ex:piecewise not compatible}
  For $d_1=3$ and $d_2=2$, consider the following grading $\omega$ of $D_4$:
  
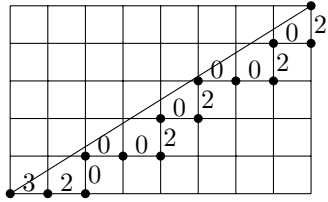
\[
    \begin{tikzpicture}
      \begin{scope}[scale=0.5]
      \draw[step=1.0, black, thin] (0,0) grid (8,5);
      \draw[black, thin] (0,0) -- (8,5);
      \draw[black, fill=black] (0,0) circle (3pt);
      \node at (0.5,0.3) {$3$};
      \draw[black, fill=black] (1,0) circle (3pt);
      \node at (1.5,0.3) {$2$};
      \draw[black, fill=black] (2,0) circle (3pt);
      \node at (2.25,0.5) {$0$};
      \draw[black, fill=black] (2,1) circle (3pt);
      \node at (2.5,1.3) {$0$};
      \draw[black, fill=black] (3,1) circle (3pt);
      \node at (3.5,1.3) {$0$};
      \draw[black, fill=black] (4,1) circle (3pt);
      \node at (4.25,1.5) {$2$};
      \draw[black, fill=black] (4,2) circle (3pt);
      \node at (4.5,2.3) {$0$};
      \draw[black, fill=black] (5,2) circle (3pt);
      \node at (5.25,2.5) {$2$};
      \draw[black, fill=black] (5,3) circle (3pt);
      \node at (5.5,3.3) {$0$};
      \draw[black, fill=black] (6,3) circle (3pt);
      \node at (6.5,3.3) {$0$};
      \draw[black, fill=black] (7,3) circle (3pt);
      \node at (7.25,3.5) {$2$};
      \draw[black, fill=black] (7,4) circle (3pt);
      \node at (7.5,4.3) {$0$};
      \draw[black, fill=black] (8,4) circle (3pt);
      \node at (8.25,4.5) {$2$};
      \draw[black, fill=black] (8,5) circle (3pt);
      \end{scope}
    \end{tikzpicture}
  \]
  Its component gradings $\omega_1$ and $\omega_2$, given by
  \[
    \omega_1=
    \raisebox{-1.9em}
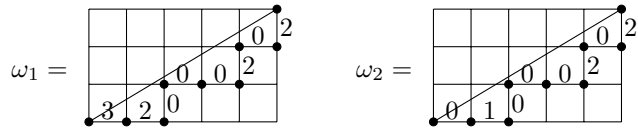
{
    \begin{tikzpicture}
      \begin{scope}[scale=0.5]
      \draw[step=1.0, black, thin] (0,0) grid (5,3);
      \draw[black, thin] (0,0) -- (5,3);
      \draw[black, fill=black] (0,0) circle (3pt);
      \node at (0.5,0.3) {$3$};
      \draw[black, fill=black] (1,0) circle (3pt);
      \node at (1.5,0.3) {$2$};
      \draw[black, fill=black] (2,0) circle (3pt);
      \node at (2.25,0.5) {$0$};
      \draw[black, fill=black] (2,1) circle (3pt);
      \node at (2.5,1.3) {$0$};
      \draw[black, fill=black] (3,1) circle (3pt);
      \node at (3.5,1.3) {$0$};
      \draw[black, fill=black] (4,1) circle (3pt);
      \node at (4.25,1.5) {$2$};
      \draw[black, fill=black] (4,2) circle (3pt);
      \node at (4.5,2.3) {$0$};
      \draw[black, fill=black] (5,2) circle (3pt);
      \node at (5.25,2.5) {$2$};
      \draw[black, fill=black] (5,3) circle (3pt);
      \end{scope}
    \end{tikzpicture}}
    \qquad
    \omega_2=
    \raisebox{-1.9em}{
    \begin{tikzpicture}
      \begin{scope}[scale=0.5]
      \draw[step=1.0, black, thin] (0,0) grid (5,3);
      \draw[black, thin] (0,0) -- (5,3);
      \draw[black, fill=black] (0,0) circle (3pt);
      \node at (0.5,0.3) {$0$};
      \draw[black, fill=black] (1,0) circle (3pt);
      \node at (1.5,0.3) {$1$};
      \draw[black, fill=black] (2,0) circle (3pt);
      \node at (2.25,0.5) {$0$};
      \draw[black, fill=black] (2,1) circle (3pt);
      \node at (2.5,1.3) {$0$};
      \draw[black, fill=black] (3,1) circle (3pt);
      \node at (3.5,1.3) {$0$};
      \draw[black, fill=black] (4,1) circle (3pt);
      \node at (4.25,1.5) {$2$};
      \draw[black, fill=black] (4,2) circle (3pt);
      \node at (4.5,2.3) {$0$};
      \draw[black, fill=black] (5,2) circle (3pt);
      \node at (5.25,2.5) {$2$};
      \draw[black, fill=black] (5,3) circle (3pt);
      \end{scope}
    \end{tikzpicture}}
    \]
    are both compatible (c.f. Example~\ref{ex:combinatorial cluster variables}) so that $\omega$ is piecewise compatible, but $\omega$ is not compatible since neither of the conditions (HGC) nor (VGC) are satisfied for the subpath $h_1v_5$ of $D_4$.
\end{example}

The next result shows that only the final edge of $D_m$ needs to be considered in order to verify (global) compatibility of a piecewise compatible grading.
\begin{lemma}
  \label{le:compatibility check}
  Let $\omega:E_m\to\ZZ_{\ge0}$ be a piecewise compatible grading of $D_m$, $m\ge3$.
  Then one of the compatibility conditions \eqref{eq:hgc} or \eqref{eq:vgc} is satisfied for every $h\in H_m$ and every $v\in V_m\setminus\{v_{u_{m-1,2}}\}$.
  In particular, a piecewise compatible grading on $D_m$ is compatible if and only if one of the compatibility conditions \eqref{eq:hgc} or \eqref{eq:vgc} holds for all paths $hv_{u_{m-1,2}}$ with $h\in\bigsqcup\limits_{r=1}^{d_m-1-\delta_m} H_{m,r}$.
\end{lemma}
\begin{proof}
  Following \cite[Remark 2.22]{rupel2}, we have a principle of non-interaction between adjacent $D_{m-1}$ subpaths of $D_m$.  
  More precisely, one of the compatibility conditions \eqref{eq:hgc} or \eqref{eq:vgc} will always be satisfied for paths $hv$ with $h\in H_{m,r}$ and $v\in V_{m,s}$ for $1\le r<s\le d_m-1-\delta_m$.  
  Since each pair $(\omega_{H,r},\omega_{V,r})$ is compatible, it only remains to verify a compatibility condition for $h\in\bigsqcup\limits_{r=1}^{d_m-1-\delta_m} H_{m,r}$ and $v\in V_{m,d_m-\delta_m}$.  

  By Corollary~\ref{cor:initial subpath}, we may again apply \cite[Remark 2.22]{rupel2} to see that one of the compatibility conditions will always be satisfied for all $v\in V_{m,d_m-\delta_m}$ with $v\ne v_{u_{m-1,2}}$.
  Thus a compatibility conditions only needs to be verified for paths $hv_{u_{m-1,2}}$ with $h\in\bigsqcup\limits_{r=1}^{d_m-1-\delta_m} H_{m,r}$ to verify the compatibility of $\omega$.
\end{proof}
\begin{corollary}
  \label{cor:automatic compatibility}
  When $d_m=1$, every piecewise compatible grading of $D_m$, $m\ge3$, is compatible.
\end{corollary}
\begin{proof}
  When $d_m=1$, the set of horizontal edges $\bigsqcup\limits_{r=1}^{d_m-1-\delta_m} H_{m,r}$ is empty and thus the compatibility condition of Lemma~\ref{le:compatibility check} is trivially satisfied.
\end{proof}

Next we observe that piecewise compatible gradings are well-behaved under the operations $\varphi^*$ and $\Omega$ introduced in Section~\ref{sec:recursions}.
\begin{proposition}
  \label{prop:piecewise equivalence}
  Let $\omega:E_m\to\ZZ_{\ge0}$ be a $(d_1,d_2)$-bounded grading on $D_m$ for $m\ge3$.  
  Then $\omega$ is piecewise compatible if and only if $(\Omega(\omega_V),\varphi^*\omega_H)$ is piecewise compatible.
\end{proposition}  
\begin{proof}
  We prove the forward implication, the other direction can be obtained by reversing the argument.  

  Assume $\omega$ is piecewise compatible and, for $1\le r\le d_m-\delta_m$,  consider $h'\in H'_{m+1,r}$ and $v'_t\in V'_{m+1,r}$ with $h'<v'_t$.  
  If $h'\notin D(v'_t;\varphi^*\omega_H)$, then the vertical grading condition \eqref{eq:vgc} is satisfied for the path $h'v'_t$.  
  So we assume $h'\in D(v'_t;\varphi^*\omega_H)$ and need to show that the horizontal grading condition \eqref{eq:hgc} is satisfied for the path $h'v'_t$.

  For $\Omega(\omega_V)(h')=0$, there is nothing to check so assume $h'\in\supp(\Omega(\omega_V))$.
  Set $j-1=\hgt(h')$ so that $h'\in\rsh(\varphi^*\omega_H)_{d;j-1}$ with $j<d\le t$.
  Then $\Omega(\omega_V)(h')=\omega_V(v)$, where $v\in\rsh(\omega_H)_{j;d}\cap\supp(\omega_V)$ with $h'=\theta_{j;d}(v)$.

  Note that $h_j\in H_{m,r}$ and $v\in V_{m,r}$, thus by piecewise compatibility the vertical grading condition \eqref{eq:vgc} is satisfied for the path $h_jv$. 
  That is, there exists $e\in\overline{h}_jv$ so that $f_{\omega_V}(ev)=0$.
  By piecewise compatibility, each vertical edge in $h_j\overline{e}$ also satisfies the vertical grading condition with $h_j$.
  It follows that $f_{\omega_V}(h_jv)<0$.

  By Lemma~\ref{le:f_and_Omega}, we thus have $f_{\Omega(\omega_V)}(h'v'_d)<0$ and so the horizontal grading condition is satisfied for the path $h'v'_d$ by Lemma~\ref{le:compatibility inequality}.
  Since $h'v'_d$ is an initial subpath of $h'v'_t$, the horizontal grading condition is also satisfied for the path $h'v'_t$.  
  Since $h'$ and $v'_t$ were arbitrary, we see that $(\Omega(\omega_V),\varphi^*\omega_H)$ is piecewise compatible.
\end{proof}  

We aim now to understand precisely when compatibility fails for a piecewise compatible grading.  
The definition below provides the necessary conditions for a piecewise compatible grading $\omega$ constructed as in Definition~\ref{def:agregate grading} to be incompatible.
\begin{definition}
  \label{def:blocking and justification}
  Let $\omega_H:H_m\to[0,d_1]$ be a horizontal grading on $D_m$, $m\ge3$.  We say a horizontal edge $h\in H_m$ is \emph{blocking for $\omega_H$} if the following hold:
  \begin{itemize}
    \item $h\in H_m\setminus H_{m,d_m-\delta_m}$;
    \item $D(h;\omega_H)=hv_{u_{m-1,2}}$;
    \item $h$ is the maximal (i.e.\ furthest to the right) horizontal edge with these properties.
  \end{itemize}
  We call $\omega_H$ \emph{left-justified} at a blocking edge $h_i\in H_m$ if there exists $k\ge i$ so that $\omega_H(h_j)>0$ for $i\le j\le k$ and $\omega_H(h_j)=0$ for $j>k$.  
  Such a horizontal grading is \emph{strongly left-justified at $h_i$} if in addition the following hold:
  \begin{itemize}
    \item $\omega_H(h_j)=d_1$ for $i\le j<k$;
    \item $f_{\omega_H}(h_iv_{u_{m-1,2}})=0$.
  \end{itemize}
  Let $\omega_V:V_m\to[0,d_2]$ be a vertical grading on $D_m$, $m\ge3$.  
  For a horizontal edge $h_i\in H_m$, $\omega_V$ is called \emph{right-justified with respect to $h_i$} if there is a vertical edge $v_s\in h_iv_{u_{m-1,2}}$ so that $\omega_V(v_t)>0$ for $s\le t\le u_{m-1,2}$ and $\omega_V(v_t)=0$ for all vertical edges $v_t\in(h_i\overline{v}_s)_V$.  
  Such a vertical grading is \emph{strongly right-justified with respect to $h_i$} if in addition the following hold:
  \begin{itemize}
    \item $\omega_V(v_t)=d_2$ for $s<t\le u_{m-1,2}$;
    \item $D(v_{u_{m-1,2}};\omega_V)=h_iv_{u_{m-1,2}}$ with $f_{\omega_V}(h_iv_{u_{m-1,2}})=0$.
  \end{itemize}
\end{definition}
\begin{example}
  The horizontal grading $\omega_H$ of $D_4$ in Example~\ref{ex:piecewise not compatible} is strongly left-justified with respect to the blocking edge $h_1$ and the vertical grading $\omega_V$ of $D_4$ is strongly right-justified with respect to $h_1$.
  The horizontal grading $\omega_{H,1}$ of $D_3$ in Example~\ref{ex:piecewise not compatible} is left-justified with respect to the blocking edge $h_1$, but it is not strongly left-justified since $f_{\omega_{H,1}}(h_1v_3)=2>0$.
  On the other hand, the vertical grading $\omega_{V,1}$ of $D_3$ in Example~\ref{ex:piecewise not compatible} is strongly right-justified with respect to $h_2$.
\end{example}

\begin{proposition}
  \label{prop:strong justification implication}
  Let $\omega:E_m\to\ZZ_{\ge0}$ be a piecewise compatible grading of $D_m$, $m\ge3$, for which $\omega_H$ admits the blocking edge $h_i\in H_m$ and with $\omega_V$ strongly right-justified with respect to $h_i$.
  Then $\omega_H$ is left-justified at $h_i$ and $\supp(\omega_H)\cap h_iv_{u_{m-1,2}}=\rsh(\omega_V)\cap h_iv_{u_{m-1,2}}$.
\end{proposition}
\begin{proof}
  Since $\omega_V$ is strongly right-justified with respect to $h_i$, we have $D(v_{u_{m-1,2}};\omega_V)=h_iv_{u_{m-1,2}}$ and thus $\supp(\omega_H)\cap h_iv_{u_{m-1,2}}\subset\rsh(\omega_V)\cap h_iv_{u_{m-1,2}}$.
  To see equality of these sets, we show that $\supp(\omega_H)\cap h_iv_{u_{m-1,2}}$ contains at least as many edges as $\rsh(\omega_V)\cap h_iv_{u_{m-1,2}}$.

  Since $\omega_V$ is strongly right-justified with respect to $h_i$, we have $|D_H(v_{u_{m-1,2}};\omega_V)|=u_{m,1}-i+1$ and so $\omega_V(v)=d_2$ for each vertical edge $v\in\overline{v}_sv_{u_{m-1,2}}$, where $s=u_{m-1,2}-\left\lfloor\frac{u_{m,1}-i+1}{d_2}\right\rfloor$.
  Moreover, this also shows $\omega_V(v_s)=u_{m,1}-i+1-d_2\left\lfloor\frac{u_{m,1}-i+1}{d_2}\right\rfloor$.
  It follows that $\omega_H(h)=0$ for each horizontal edge $h\in\overline{v}_sv_{u_{m-1,2}}$ and for each of the $u_{m,1}-i+1-d_2\left\lfloor\frac{u_{m,1}-i+1}{d_2}\right\rfloor$ horizontal edges $h$ immediately preceding $v_s$.
  Otherwise both grading conditions would fail for the path $hv$, where $v$ is the first vertical edge after $h$, in contradiction to the piecewise compatibility of $\omega$.

  Observe that the vertical edge $v_s$ has depth $\left\lceil\frac{\left(u_{m-1,2}-\left\lfloor\frac{u_{m,1}-i+1}{d_2}\right\rfloor\right)u_{m,1}}{u_{m-1,2}}\right\rceil$ by Lemma~\ref{le:height and depth} and thus we may conclude more precisely that $\omega_H(h_j)=0$ whenever $j$ is larger than the following quantity:
  \begin{align*}
    &\left\lceil\frac{\left(u_{m-1,2}-\left\lfloor\frac{u_{m,1}-i+1}{d_2}\right\rfloor\right)u_{m,1}}{u_{m-1,2}}\right\rceil-\left(u_{m,1}-i+1-d_2\left\lfloor\frac{u_{m,1}-i+1}{d_2}\right\rfloor\right)\\
    &\quad=i-1+\left\lceil\frac{-\left\lfloor\frac{u_{m,1}-i+1}{d_2}\right\rfloor u_{m,1}}{u_{m-1,2}}\right\rceil+d_2\left\lfloor\frac{u_{m,1}-i+1}{d_2}\right\rfloor\\
    &\quad=i-1+\left\lceil\frac{\left\lfloor\frac{u_{m,1}-i+1}{d_2}\right\rfloor u_{m-2,1}}{u_{m-1,2}}\right\rceil,
  \end{align*}
  where both equalities follow from the identity $\lceil n+x\rceil=n+\lceil x\rceil$ which holds for all real numbers $x$ and all integers $n$.
  This discussion also shows that $\big(\sh(\omega_V)\setminus\rsh(\omega_V)\big)\cap h_iv_{u_{m-1,2}}=(h_{i+d}v_{u_{m-1,2}})_H$, where $d=\left\lceil\frac{\left\lfloor\frac{u_{m,1}-i+1}{d_2}\right\rfloor u_{m-2,1}}{u_{m-1,2}}\right\rceil$.  
  Since $D(v_{u_{m-1,2}};\omega_V)=h_iv_{u_{m-1,2}}$, it follows that 
  \[\rsh(\omega_V)\cap h_iv_{u_{m-1,2}}=\{h_i,h_{i+1},\ldots,h_{i+d-1}\}\]
  and so 
  \[|\rsh(\omega_V)\cap h_iv_{u_{m-1,2}}|=\left\lceil\frac{\left\lfloor\frac{u_{m,1}-i+1}{d_2}\right\rfloor u_{m-2,1}}{u_{m-1,2}}\right\rceil.\]

  Now observe the inequality 
  \[\left\lceil\frac{\left\lfloor\frac{u_{m,1}-i+1}{d_2}\right\rfloor u_{m-2,1}}{u_{m-1,2}}\right\rceil\le\left\lceil\frac{(u_{m,1}-i+1)u_{m-2,1}}{d_2u_{m-1,2}}\right\rceil=\left\lceil\frac{(u_{m,1}-i+1)u_{m-2,2}}{d_1u_{m-1,1}}\right\rceil,\] 
  where the equality can be deduced from the identities in Remark~\ref{rem:chebyshev}.
  By Lemma~\ref{le:Dyck path inequality}, the last expression above is not larger than  
  \begin{equation}
    \label{eq:strongly justified support}
    \left\lceil\frac{(u_{m,1}-i+1)u_{m-1,2}}{d_1u_{m,1}}\right\rceil=\left\lceil\frac{u_{m-1,2}-\left\lfloor\frac{(i-1)u_{m-1,2}}{u_{m,1}}\right\rfloor}{d_1}\right\rceil,
  \end{equation}
  where the equality follows from right to left using the identities $-\lfloor x\rfloor=\lceil-x\rceil$, $\lceil n+x\rceil=n+\lceil x\rceil$, and $\left\lceil\frac{\lceil x\rceil}{n}\right\rceil=\left\lceil\frac{x}{n}\right\rceil$ which hold for all real numbers $x$ and all positive integers $n$.
  
  But $h_i$ is blocking and $\omega_H$ is $d_1$-bounded so that
  \[|\supp(\omega_H)\cap h_iv_{u_{m-1,2}}|\ge\left\lceil\frac{|h_iv_{u_{m-1,2}}|_V}{d_1}\right\rceil=\left\lceil\frac{u_{m-1,2}-\left\lfloor\frac{(i-1)u_{m-1,2}}{u_{m,1}}\right\rfloor}{d_1}\right\rceil.\]
  Combining this observation with the inequalities leading up to equation~\eqref{eq:strongly justified support}, we see that
  \[|\supp(\omega_H)\cap h_iv_{u_{m-1,2}}|\ge\left\lceil\frac{u_{m-1,2}-\left\lfloor\frac{(i-1)u_{m-1,2}}{u_{m,1}}\right\rfloor}{d_1}\right\rceil\ge\left\lceil\frac{\left\lfloor\frac{u_{m,1}-i+1}{d_2}\right\rfloor u_{m-2,1}}{u_{m-1,2}}\right\rceil=|\rsh(\omega_V)\cap h_iv_{u_{m-1,2}}|.\]
  But either inequality being strict is impossible since $\supp(\omega_H)\cap h_iv_{u_{m-1,2}}\subset\rsh(\omega_V)\cap h_iv_{u_{m-1,2}}$. 
  Thus we have
  \begin{equation}
    \label{eq:incompatible support condition}
    |\supp(\omega_H)\cap h_iv_{u_{m-1,2}}|=\left\lceil\frac{u_{m-1,2}-\left\lfloor\frac{(i-1)u_{m-1,2}}{u_{m,1}}\right\rfloor}{d_1}\right\rceil=\left\lceil\frac{\left\lfloor\frac{u_{m,1}-i+1}{d_2}\right\rfloor u_{m-2,1}}{u_{m-1,2}}\right\rceil=|\rsh(\omega_V)\cap h_iv_{u_{m-1,2}}|,
  \end{equation}
  in particular $\supp(\omega_H)\cap h_iv_{u_{m-1,2}}=\rsh(\omega_V)\cap h_iv_{u_{m-1,2}}=\{h_i,h_{i+1},\ldots,h_{i+d-1}\}$ which shows $\omega_H$ must be left-justified at $h_i$.
\end{proof}

\begin{remark}
  The middle equality of equation~\eqref{eq:incompatible support condition} does not generally hold for all $i$, this equality is a consequence of the hypotheses and thus provides a necessary condition for the existence of a piecewise compatible grading as in Proposition~\ref{prop:strong justification implication}.
  The next result will show that this condition is also sufficient and that such gradings are the only piecewise compatible gradings which are not compatible.
\end{remark} 

\begin{theorem}
  \label{th:blocking edge conditions}
  Let $\omega:E_m\to\ZZ_{\ge0}$ be a piecewise compatible grading of $D_m$, $m\ge3$.
  \begin{enumeratea}
    \item If $\omega_H$ does not admit a blocking edge, then $\omega$ is compatible.
    \item Suppose $\omega_H$ admits a blocking edge $h_i$, but $\omega$ is not compatible.
      Then the following hold:
      \begin{enumeratei}
        \item $D(v_{u_{m-1,2}};\omega_V)=h_iv_{u_{m-1,2}}$ with $f_{\omega_V}(h_iv_{u_{m-1,2}})=0$;
        \item $\omega_H$ is left-justified at $h_i$ and $\omega_V$ is strongly right-justified with respect to $h_i$.
      \end{enumeratei}
      If in addition $m\ge4$, the following also hold:
      \begin{enumeratei}\addtocounter{enumii}{2}
        \item $f_{\omega_H}(h_iv_{u_{m-1,2}})=0$;
        \item $\omega_H$ must be strongly left-justified at $h_i$;
        \item $\hgt(h_{i+1})=\hgt(h_i)+\delta_1$ when $|\supp(\omega_H)\cap h_iv_{u_{m-1,2}}|>1$.
      \end{enumeratei}
  \end{enumeratea}
\end{theorem}
\begin{remark}
  When $d_m=1$, the hypotheses of Theorem~\ref{th:blocking edge conditions}(b) cannot apply by Corollary~\ref{cor:automatic compatibility}.
\end{remark}
\begin{proof}
  If $\omega_H$ does not admit a blocking edge, any horizontal edge $h\in H_m$ has a local shadow path of the form $D(h;\omega_H)=he$ with $e<v_{u_{m-1,2}}$, i.e.\ the horizontal grading condition is satisfied for $h$ and $v_{u_{m-1,2}}$.  By Lemma~\ref{le:compatibility check}, this implies $\omega$ is compatible, establishing (a).  
  
  Now assume $\omega_H$ admits a blocking edge $h_i$ and $\omega$ is not compatible.
  It follows that $d_m\ne1$ by Corollary~\ref{cor:automatic compatibility}.  
  
  First consider the case $\hgt(h_i)\ge u_{m-1,2}-d_1$.
  By definition $h_i$ is not contained in the final $D_{m-1}\setminus D_{m-2-\delta_m}$ subpath of $D_m$ (which contains at least one vertical edge) and so we must have $d_1\ge2$.
  Lemma~\ref{le:Dyck tail inequality} gives the conditions when such a blocking edge can exist.

  Let $\ell=\hgt(h_i)$.
  Since $d_1\ne1$, each vertical edge in $\overline{v}_{\ell+1+\delta_m}\overline{v}_{u_{m-1,2}}$ is immediately preceded by exactly $d_2$ horizontal edges inside $h_iv_{u_{m-1,2}}$ while $v_{u_{m-1,2}}$ is immediately preceded by exactly $d_2-1$ horizontal edges.
  In particular, we see that the vertical grading condition fails for the path $h_iv_{u_{m-1,2}}$ exactly when:
  \begin{itemize}
    \item $\omega_V(v_{\ell+1})=\dpt(v_{\ell+1})-i$ and $\omega_V(v)=d_2$ for $v\in(\overline{v}_{\ell+1}v_{u_{m-1,2}})_V$ if $d_2\ge2$;
    \item $\omega_V(v_{\ell+1})=0$, $\omega_V(v_{\ell+2})=0$, and $\omega_V(v)=d_2$ for $v\in(\overline{v}_{\ell+2}v_{u_{m-1,2}})_V$ if $d_2=1$;
  \end{itemize}
  In either case we have $D(v_{u_{m-1,2}};\omega_V)=h_iv_{u_{m-1,2}}$ with $f_{\omega_V}(h_iv_{u_{m-1,2}})=0$.
  Note that in each of the cases above, $\omega_H$ is left-justified at $h_i$ with $k=i$ in Definition~\ref{def:blocking and justification} and $\omega_V$ is strongly right-justified with respect to $h_i$.
  This establishes the claims in the first part of (b) for these cases.
  Observe that our assumptions when $m\ge4$ imply $f_{\omega_H}(h_iv_{u_{m-1,2}})=0$ and that $\omega_H$ is strongly left-justified at $h_i$.
  Since $\supp(\omega_H)\cap h_iv_{u_{m-1,2}}=\{h_i\}$, this establishes the second part of (b) in these cases.
  
  Next consider the case $\hgt(h_i)<u_{m-1,2}-d_1$ which requires $m\ge4$.  
  Then there must exist $j>i$ so that $D(h_j;\omega_H)=h_jv_{u_{m-1,2}-\ell}$ with $1\le\ell\le\omega_H(h_i)-\delta_1$
  (the extra $\delta_1$ must be included here since $d_2=1$ implies all horizontal edges of $D_m$ have different heights, in other words $d_2=1$ implies $h_i$ is immediately followed by a vertical edge). 
  Assume that $j$ is chosen so that $\ell$ is minimal, in particular when $d_1=1$ we must have $\ell=1$.

  By Lemma~\ref{le:compatibility check}, the vertical grading condition must be satisfied for the paths $h_iv$ with $v\in(h_i\overline{v}_{u_{m-1,2}})_V$.
  For each such $v$, we have $D(v;\omega_V)=h_{j(v)}v$ for some $j(v)>i$, in particular $f_{\omega_V}(h_{j(v)}v)=0$.  
  Since $h_i$ is blocking, it cannot be contained in the shadow of any of these vertical edges.
  Moreover, when $d_1=1$, the edge $h_j$ will also not be contained in the shadow of any of these vertical edges. 
  Thus we see that there are at least $1+\delta_2$ horizontal edges of the path $h_iv_{u_{m-1,2}-1}$ lying outside the shadows of its vertical edges and applying equation~\eqref{eq:shadow statistic concatenation} shows $f_{\omega_V}(h_iv_{u_{m-1,2}-1})\le-(1+\delta_2)$.
  But by Corollary~\ref{cor:recursive dyck path structure} there are $d_2-1-\delta_2$ horizontal edges immediately preceding $v_{u_{m-1,2}}$ and, since $\omega_V(v_{u_{m-1,2}})\le d_2$, we must have $f_{\omega_V}(h_iv_{u_{m-1,2}})\le0$.  
  We conclude that one of the following holds:
  \begin{itemize}
    \item $D(v_{u_{m-1,2}};\omega_V)$ is a proper subpath of $h_iv_{u_{m-1,2}}$ by Lemma~\ref{le:compatibility inequality} and thus $\omega$ is compatible;
    \item $D(v_{u_{m-1,2}};\omega_V)=h_iv_{u_{m-1,2}}$ with $f_{\omega_V}(h_iv_{u_{m-1,2}})=0$ and both compatibility conditions fail for the path $h_iv_{u_{m-1,2}}$.
  \end{itemize}
  This establishes claim (i) of (b) in this case.  
  When $d_1=1$, we must have $f_{\omega_H}(h_iv_{u_{m-1,2}})=0$ for otherwise $h_i$ could not be blocking.
  This gives claim (iii) of (b) when $d_1=1$.  
  To complete the proof of (iii) for $d_1>1$ and $m\ge4$, we observe that $h_i$ being a blocking edge implies $f_{\omega_H}(h_iv_{u_{m-1,2}})\ge0$.
  Our aim then is to show that $f_{\omega_H}(h_iv_{u_{m-1,2}})>0$ implies the second situation above is impossible.

  Indeed, $f_{\omega_H}(h_iv_{u_{m-1,2}})>0$ can only occur if we take $\ell\le\omega_H(h_i)-1-\delta_1$ above.
  But, assuming $d_1>1$ and $m\ge4$, there are $d_2$ horizontal edges of $D_m$ immediately preceding each of the $d_1-2-\delta_1$ vertical edges 
  \[v_{u_{m-1,2}-d_1+2+\delta_1},v_{u_{m-1,2}-d_1+3+\delta_1},\ldots,v_{u_{m-1,2}-1},\]
  and $d_2-1$ horizontal edges immediately preceding $v_{u_{m-1,2}}$ (by Corollary~\ref{cor:recursive dyck path structure}, the terminal subpath of $D_m$ containing all these edges identifies with the terminal subpath $D_3\setminus D_2$ inside $D_3$).  
  It follows that $D(v_{u_{m-1,2}};\omega_V)$ must be a subpath of $h_jv_{u_{m-1,2}}$ and so the vertical grading condition is satisfied for the path $h_iv_{u_{m-1,2}}$.  
  In particular, $\omega$ is compatible by Lemma~\ref{le:compatibility check}, this completes the proof of (iii).

  The arguments above also establish the following when $m\ge4$, $d_m\ne1$, and $\hgt(h_i)<u_{m-1,2}-d_1$:
  \begin{itemize}
    \item if $\omega_H(h_i)<d_1$ or $h_i$ is immediately followed by $1+\delta_1$ vertical edges, then either $d_1=1$ and $h_i$ cannot possibly be blocking or there must exist a horizontal edge $h_j$ as in the previous paragraph and compatibility again holds, this gives (v) once we have established (iv), i.e.\ once we know that $\omega_H$ is strongly left-justified at $h_i$; 
    \item if $\omega_V(v_{u_{m-1,2}-t})<d_2$ for any $0\le t\le d_1-\delta_1$, then the piecewise compatible grading $\omega$ must be compatible.
  \end{itemize}

  We prove (ii) and (iv) by induction on $m\ge3$, $d_m\ne1$.
  The base case $m=3$ of (ii) was established in the first part of the proof.
  Suppose $m\ge4$ and $\omega$ is not compatible.  
  By Proposition~\ref{prop:piecewise equivalence} the grading $\big((\varphi^*_{m-1})^{-1}\omega_V,\Omega_{m-1}^{-1}(\omega_H)\big)=:(\omega_{H'},\omega_{V'})$ of $D'_{m-1}$ is piecewise compatible, but not compatible by Proposition~\ref{prop:compatibility equivalence}.  
  By part (a), there must be a blocking edge $h'_j$ for $\omega_{H'}$.
  Applying (ii) to the grading $(\omega_{H'},\omega_{V'})$ we see that $\omega_{H'}$ is left-justified at $h'_j$ and $\omega_{V'}$ is strongly right-justified with respect to this blocking edge.

  When $m=4$, we have $\supp(\omega_{H'})\cap h'_jv'_{u'_{m-2,2}}=\{h'_j\}$ and from the definition of $\varphi^*_{m-1}$ we see that $\omega_V$ is strongly right-justified with respect to $h_i$.
  This requires the extra observation above that we had to take $k=i$ in the definition of left-justification for the case $m=3$.
  For $m\ge5$, claim (iv) applied to the grading $(\omega_{H'},\omega_{V'})$ shows that $\omega_{H'}$ is strongly left-justified at $h'_j$ and again the definition of $\varphi^*_{m-1}$ shows that $\omega_V$ is strongly right-justified with respect to $h_i$.  
  By Proposition~\ref{prop:strong justification implication}, we see that $\omega_H$ must be left-justified at $h_i$.  

  It remains to argue that $\omega_H$ is strongly left-justified at $h_i$, but this is immediate from Lemma~\ref{le:remote shadow cardinalities} and the definition of the maps $\theta$.  
  Indeed, since $\omega_{H'}$ is strongly left-justified at its blocking edge $h'_j$, the remote shadows of the horizontal edges in $h'_jv'_{u_{m-2,1}}$ are linearly ordered in the opposite order to the horizontal edges in $\supp(\omega_{H'})\cap h'_jv'_{u_{m-2,1}}$.  
  Since $\omega_V$ is strongly right-justified with respect to $h_i$, analogous statements can be made about the remote shadows of the vertical edges in $h_iv_{u_{m-1,2}}$.  
  But the maps $\theta$ are compatible with these orderings and so $\omega_{V'}$ being strongly right-justified with respect to $h'_j$ forces $\omega_H=\Omega_{m-1}(\omega_{V'})$ to be strongly left-justified at $h_i$.
  This completes the proof of (ii) and (iv).
\end{proof}

The next result severely restricts which horizontal edges can be blocking.
\begin{corollary}
  Let $\omega:E_m\to\ZZ_{\ge0}$ be a piecewise compatible grading of $D_m$, $m\ge5$, which is not compatible.
  Write $h_i\in H_m$ for the blocking edge of $\omega_H$.
  Then either $i=1$ or $h_i$ is immediately preceded by a vertical edge.
\end{corollary}
\begin{proof}
  By Proposition~\ref{prop:piecewise equivalence} and Proposition~\ref{prop:compatibility equivalence}, the grading $(\omega_{H'},\omega_{V'}):=\big((\varphi^*_{m-1})^{-1}\omega_V,\Omega_{m-1}^{-1}\omega_H)$ of $D'_{m-1}$ is piecewise compatible but not compatible.
  Let $h'_j\in H'_{m-1}$ denote the blocking edge of $\omega_{H'}$.
  Then since $m\ge5$, we have $|\rsh(h'_j;\omega_{H'})|=d_2-\ell$, where $\ell$ is the number of vertical edges immediately following $h'_j$.
  By Lemma~\ref{le:remote shadow cardinalities}, this implies there are $d_2-\ell$ horizontal edges of height $j-1$ in the remote shadow of $\omega_V$ and the leftmost of these is the leftmost edge in $\rsh(v_{u_{m-1,2}};\omega_V)$.
  But there are exactly $d_2-\ell$ horizontal edges of height $j-1$ inside $D_m$ by Lemma~\ref{le:Dyck path recursion}.
  Since $D(v_{u_{m-1,2}};\omega_V)=h_iv_{u_{m-1,2}}$, the edge $h_i$ is the leftmost horizontal edge of height $j-1$, this gives the result.
\end{proof}

We also obtain the following analogue of Proposition~\ref{prop:strong justification implication}.
\begin{corollary}
  \label{cor:support and remote shadow}
  Let $\omega:E_m\to\ZZ_{\ge0}$ be a piecewise compatible grading of $D_m$, $m\ge3$, which is not compatible.
  If $h_i\in H_m$ denotes the blocking edge for $\omega_H$, then $\supp(\omega_V)\cap h_iv_{u_{m-1,2}}=\rsh(\omega_H)\cap h_iv_{u_{m-1,2}}$.
\end{corollary}
\begin{proof}
  Since $\omega$ is not compatible, the grading $(\Omega_m(\omega_V),\varphi^*_m\omega_H)=:(\omega_{H'},\omega_{V'})$ of $D'_{m+1}$ is not compatible by Proposition~\ref{prop:compatibility equivalence}, but is piecewise compatible by Proposition~\ref{prop:piecewise equivalence}.
  By Theorem~\ref{th:blocking edge conditions} the grading $(\omega_{H'},\omega_{V'})$ satisfies the hypotheses of Proposition~\ref{prop:strong justification implication} and so $\supp(\omega_{H'})\cap h'_jv'_{u'_{m,2}}=\rsh(\omega_{V'})\cap h'_jv'_{u'_{m,2}}$, where $h'_j$ denotes the blocking edge of $\omega_{H'}$.

  By piecewise compatibility, we must have $\supp(\omega_V)\cap h_iv_{u_{m-1,2}}\subset\rsh(\omega_H)\cap h_iv_{u_{m-1,2}}$ since every vertical edge in $\supp(\omega_V)\cap h_iv_{u_{m-1,2}}$ is contained in the shadow of $\omega_H$.
  If there exists $v\in\rsh(\omega_H)\cap h_iv_{u_{m-1,2}}$ with $\omega_V(v)=0$, by Lemma~\ref{le:remote shadow cardinalities} there will be a horizontal edge $h'\in\rsh(\omega_{V'})\cap h'_jv'_{u'_{m,2}}$ with $\omega_{H'}(h')=0$, a contradiction.
  Therefore we must have $\supp(\omega_V)\cap h_iv_{u_{m-1,2}}=\rsh(\omega_H)\cap h_iv_{u_{m-1,2}}$.
\end{proof}

As a final consequence we show that the piecewise compatible gradings which are not compatible satisfy a certain upper bound property with respect to compatible gradings.
\begin{corollary}
  \label{cor:maximality}
  Suppose $\omega:E_m\to\ZZ_{\ge0}$ is a piecewise compatible grading of $D_m$, $m\ge3$, which is not compatible.
  Write $h_i$ for the blocking edge of $\omega_H$.
  Then the following hold:
  \begin{enumeratea}
    \item for any vertical grading $\chi_V\in\cC(\omega_H)$ and any edge $v\in(h_iv_{u_{m-1,2}})_V$, we have $\chi_V(v)\le\omega_V(v)$;
    \item for any horizontal grading $\chi_H\in\cC(\omega_V)$ and any edge $h\in(h_iv_{u_{m-1,2}})_H$, we have $\chi_H(h)\le\omega_H(h)$.
  \end{enumeratea}
\end{corollary}
\begin{proof}
  We begin by making a few basic observations which allow to deduce part (b) for $D_m$ from part (a) for $D'_{m-1}$.
 
  Consider a horizontal grading $\chi_H\in\cC(\omega_V)$ and suppose $\omega_H(h)<\chi_H(h)$ for some $h\in(h_iv_{u_{m-1,2}})_H$.
  This implies $\omega_H(h)<d_1$ since we only consider $d_1$-bounded horizontal gradings.
  By Theorem~\ref{th:blocking edge conditions}, we have $D(v_{u_{m-1,2}};\omega_V)=h_iv_{u_{m-1,2}}$ and so every edge of $h_iv_{u_{m-1,2}}$ is in the shadow of $\omega_V$.
  Thus we have 
  \[\supp(\chi_H)\cap h_iv_{u_{m-1,2}}\subset\rsh(\omega_V)\cap h_iv_{u_{m-1,2}}=\supp(\omega_H)\cap h_iv_{u_{m-1,2}},\]
  where the equality comes from Proposition~\ref{prop:strong justification implication}.
  By Theorem~\ref{th:blocking edge conditions}, $\omega_H$ is strongly left-justified at $h_i$ and so the only edge $h\in\supp(\omega_H)\cap h_iv_{u_{m-1,2}}$ which could satisfy $\omega_H(h)<d_1$ is $h=h_{i-1+d}$, where $d=|\supp(\omega_H)\cap h_iv_{u_{m-1,2}}|$.

  For $m=3$, we have $\supp(\omega_H)\cap h_iv_{u_{m-1,2}}=\{h_i\}$.
  Since the horizontal grading condition~\eqref{eq:hgc} of $\omega_H$ is not satisfied for the path $h_iv_{u_{m-1,2}}$, the inequality $\omega_H(h_i)<\chi_H(h_i)$ implies the horizontal grading condition of $\chi_H$ is also not satisfied for the path $h_iv_{u_{m-1,2}}$.
  In particular, $(\chi_H,\omega_V)$ is not compatible, a contradiction. 
 
  For $m\ge4$, consider the compatible grading $(\omega_{H'},\chi_{V'}):=\big((\varphi^*_{m-1})^{-1}\omega_V,\Omega_{m-1}^{-1}\chi_H\big)$ of $D'_{m-1}$ (see Proposition~\ref{prop:compatibility equivalence}) and the piecewise compatible grading $(\omega_{H'},\omega_{V'}):=\big((\varphi^*_{m-1})^{-1}\omega_V,\Omega_{m-1}^{-1}\omega_H\big)$ of $D'_{m-1}$ (see Proposition~\ref{prop:piecewise equivalence}).
  By the definition of $\Omega$, we have
  \[\chi_{V'}(\theta^{-1}h_{i-1+d})=\chi_H(h_{i-1+d})>\omega_H(h_{i-1+d})=\omega_{V'}(\theta^{-1}h_{i-1+d}).\]
  This contradicts part (a) applied to the grading $(\omega_{H'},\omega_{V'})$ of $D'_{m-1}$ and so there can be no grading $\chi_H$ as above.
  Thus part (b) holds for $m$ once we have established part (a) for $m-1$, $m\ge4$.

  To continue we suppose there exists a vertical grading $\chi_V\in\cC(\omega_H)$ such that $\chi_V(v)>\omega_V(v)$ for some $v\in(h_iv_{u_{m-1,2}})_V$.
  As above, this implies $0<\omega_V(v)<d_2$ and thus $v=v_{u_{m-1,2}-t+1}$, where $t=|\supp(\omega_V)\cap h_iv_{u_{m-1,2}}|$.
  In particular, we must have $d_2\ge2$ and by Corollary~\ref{cor:edge restrictions} the Dyck path $D_m$ has no consecutive vertical edges.
  
  Note that, by Proposition~\ref{prop:strong justification implication}, there are only two possibilities for the height of the edge $h_{i-1+d}$.
  Either $\hgt(h_{i-1+d})=u_{m-1,2}-t$ so that $v_{u_{m-1,2}-t+1}\in\rsh(h_{i-1+d};\omega_H)$ or $\hgt(h_{i-1+d})=u_{m-1,2}-t-1$ with $h_{i-1+d}$ immediately followed by a single vertical edge.
  In the latter case, $\omega_H(h_{i-1+d})>1$ also implies $v_{u_{m-1,2}-t+1}\in\rsh(h_{i-1+d};\omega_H)$.

  If $v_{u_{m-1,2}-t+1}\in\rsh(h_{i-1+d};\omega_H)$, the horizontal grading condition \eqref{eq:hgc} is not satisfied for the path $h_{i-1+d}v_{u_{m-1,2}-t+1}$ and we have $D(v_{u_{m-1,2}-t+1};\omega_V)=h_{i+d}v_{u_{m-1,2}-t+1}$ by Proposition~\ref{prop:strong justification implication}.
  But then for $\chi_V$ as above, the vertical grading condition \eqref{eq:vgc} is not satisfied for the path $h_{i-1+d}v_{u_{m-1,2}-t+1}$.
  In particular, this implies $(\omega_H,\chi_V)$ is not compatible, a contradiction.

  Thus we must have $\hgt(h_{i-1+d})=u_{m-1,2}-t-1$ with $h_{i-1+d}$ immediately followed by exactly one vertical edge and $\omega_H(h_{i-1+d})=1$.
  Then, since $\omega_H$ is strongly left-justified at $h_i$, we have $v_{u_{m-1,2}-t+1}\in\rsh(h_{i-2+d};\omega_H)$ and so the horizontal grading condition \eqref{eq:hgc} is not satisfied for the path $h_{i-2+d}v_{u_{m-1,2}-t+1}$.
  If $h_{i-1+d}\in\rsh(v_{u_{m-1,2}-t+1};\omega_V)$, we must have $D(v_{u_{m-1,2}-t+1};\omega_V)=h_{i-1+d}v_{u_{m-1,2}-t+1}$.
  But then for $\chi_V$ as above, the vertical grading condition \eqref{eq:vgc} is not satisfied for the path $h_{i-2+d}v_{u_{m-1,2}-t+1}$.
  In particular, this implies $(\omega_H,\chi_V)$ is not compatible, a contradiction.

  Thus the horizontal edge $h_{i-1+d}$ must lie beyond the shadow of $v_{u_{m-1,2}-t+1}$.
  By Proposition~\ref{prop:strong justification implication}, there can be no horizontal edges of height $u_{m-1,2}-t$ in the remote shadow of $\omega_V$ and so we must have $\omega_V(v_{u_{m-1,2}-t+1})=\ell$, where $\ell<d_2$ is the number of horizontal edges immediately preceding $v_{u_{m-1,2}-t+1}$. 
  For $m=3$, this can only occur for $t=1$, but $v_{u_{m-1,2}}$ is immediately preceded by $d_2-1$ horizontal edges inside $D_3$ and thus $\omega$ is compatible, a contradiction.

  So we must have $m\ge4$.
  Consider the piecewise compatible grading $(\omega_{H'},\omega_{V'}):=\big((\varphi^*_{m-1})^{-1}\omega_V,\Omega_{m-1}^{-1}\omega_H\big)$ of $D'_{m-1}$ (see Proposition~\ref{prop:piecewise equivalence}).
  Since $\omega_V$ is strongly right-justified and $\omega_V(v_{u_{m-1,2}-t+1})<d_2$, the last horizontal edge in $\supp(\omega_{H'})$ must be $h'_{u_{m-1,2}-t+1}$ with $\omega_{H'}(h'_{u_{m-1,2}-t+1})=d_2-\ell$, this being exactly the number of vertical edges immediately following $h'_{u_{m-1,2}-t+1}$ by Lemma~\ref{le:Dyck path recursion}.
  Moreover, by Lemma~\ref{le:remote shadow cardinalities}, the first vertical edge $v'$ in $\rsh(\omega_{H'})$ lies in the remote shadow of $h'_{u_{m-1,2}-t}$ and $\omega_{V'}(v')=1$.
  By Corollary~\ref{cor:support and remote shadow}, $v'$ cannot be immediately preceded by a vertical edge.
  But then there exists $\chi_{V'}$ with $\chi_{V'}(v')=2$ compatible with $\omega_{H'}$, a contradiction with part (b) for $D'_{m-1}$.
  
  This contradiction shows there can be no vertical $\chi_V$ as above and thus proves (a).
\end{proof}

\section{Proof of Main Theorem}
\label{sec:proof of main}

We begin this section with a general statement about non-commutative weights associated to certain gradings of an arbitrary (i.e.\ not necessarily maximal) Dyck path, here we make no boundedness assumptions on the gradings.

\begin{proposition}
  \label{prop:noncommutative collapse}
  Let $D$ be any Dyck path with edges $E=H\sqcup V$, where $H=\{h_1,\ldots,h_{a_1}\}$ with $a_1\ge1$ and $V=\{v_1,\ldots,v_{a_2}\}$ denote the sets of horizontal and vertical edges of $D$.  
  Write $E=\{1,2,\ldots,a_1+a_2\}$ for the edges of $D$ taken in the natural order.
  Let $\omega:E\to\ZZ_{\ge0}$ be any grading of $D$.
  Given $q_{i,j}\in\kk$ for $i\in\{1,2\}$ and $j\in\ZZ_{\ge0}$, define non-commutative weights
  \begin{equation}
    \label{eq:general edge weights}
    \wt_\omega(e)=\begin{cases}
                    q_{1,\omega(e)}Y^{\omega(e)}X^{-1} & \text{if $e\in H$;}\\
                    q_{2,\omega(e)}X^{\omega(e)+1}Y^{-1}X^{-1} & \text{if $e\in V$;}\\
                  \end{cases}
  \end{equation}
  and let $Y_D(\omega)=\wt_\omega(1)\wt_\omega(2)\cdots\wt_\omega(a_1+a_2)$.
  Assume $\omega$ is compatible and satisfies the following:
  \begin{enumeratea}
    \item the local shadow path $D(h_1;\omega_H)=D$ with $f_{\omega_H}(D)=0$;
    \item for any other vertical grading $\chi_V\in\cC(\omega_H)$ and any vertical edge $v_t\in V$ so that $\chi_V(v_s)=\omega_V(v_s)$ for $s<t$, we have $\chi_V(v_t)\le\omega_V(v_t)$.
  \end{enumeratea}
  Then $Y_D(\omega)=pX^{-1}$, where $p=\prod_{i=1}^{a_1}q_{1,\omega(h_i)}\cdot\prod_{t=1}^{a_2}q_{2,\omega(v_t)}$.
\end{proposition}
\begin{proof}
  We first note that the coefficient $p$ is immediate from the definition of the non-commutative edge weights in equation~\eqref{eq:general edge weights}.  
  Thus we assume all $q_{i,j}=1$ for the remainder of the proof.

  We work by induction on $a_2$.  
  For $a_2=0$, assumption (a) implies $a_1=1$ and $\omega_H(h_1)=0$.  
  The claim follows in this case directly from the definition of the non-commutative edge weights in equation~\eqref{eq:general edge weights}.

  Suppose $a_2\ge1$ and consider $h_i\in\supp(\omega_H)$ with $i$ maximal.
  Let $v_r\in V$ denote the next vertical edge after $h_i$, i.e.\ the path $h_iv_r$ consists of several consecutive horizontal edges, say $d$ of them, followed by a single vertical edge.
  By assumption (a), we have $r<a_2$.
  By assumption (b), we have $\omega_V(v_r)=d-1$ so that 
  \[\tag{$\dagger$}Y_{h_iv_r}(\omega)=\big(Y^{\omega_H(h_i)}X^{-1}\big)\big(X^{-1}\big)^{d-1}\big(X^dY^{-1}X^{-1}\big)=Y^{\omega_H(h_i)-1}X^{-1}.\]
  Let $\widetilde D$ be the Dyck path obtained from $D$ by replacing the path $h_iv_r$ by a single horizontal edge.
  Write $\widetilde E=\widetilde H\sqcup\widetilde V$ for the edges of $\widetilde D$, where $\widetilde H=\{\tilde h_1,\ldots,\tilde h_{a_1-d+1}\}$ and $\widetilde V=\{\tilde v_1,\ldots,\tilde v_{a_2-1}\}$ denote the horizontal and vertical edges of $\widetilde D$.
  Define a grading $\widetilde\omega:\widetilde E\to\ZZ_{\ge0}$ by
  \[\widetilde\omega_H(\tilde h_j)=\begin{cases} \omega_H(h_j) & \text{if $j<i$;}\\ \omega_H(h_i)-1 & \text{if $j=i$;}\\ 0 & \text{if $j>i$;}\end{cases}
    \hspace{1in}
    \widetilde\omega_V(\tilde v_s)=\begin{cases} \omega_V(v_s) & \text{if $s<r$;}\\ \omega_V(v_{s+1}) & \text{if $s\ge r$.}\end{cases}\]
  It is not hard to see that $\widetilde\omega$ satisfies assumptions (a) and (b), thus by induction we have $Y_{\widetilde D}(\widetilde\omega)=X^{-1}$.
  By ($\dagger$), we have $Y_D(\omega)=Y_{\widetilde D}(\widetilde\omega)$ and so $Y_D(\omega)=X^{-1}$ as desired.
\end{proof}

Now we turn to the proof of Theorem~\ref{th:combinatorial construction} and return to our standard boundedness assumptions on gradings.
\begin{lemma}
  \label{le:noncommutative collapse}
  Let $\omega:E_m\to\ZZ_{\ge0}$ be a piecewise compatible grading of $D_m$, $m\ge3$, which is not compatible.
  Denote by $h_i\in H_m$ the blocking edge of $\omega_H$.  
  Set 
  \[d=|\supp(\omega_H)\cap h_iv_{u_{m-1,2}}|\quad\text{ and }\quad t=|\supp(\omega_V)\cap h_iv_{u_{m-1,2}}|.\]
  Then for any $h\in(\overline{h}_ih_{i-1+d})_H$, we have $Y_{D(h;\omega_H)}(\omega)=pX^{-1}$, where $p=p_{1,\omega_H(h_{i-1+d})}p_{2,d_2-\omega_V(v_{u_{m-1,2}-t+1})}$.
\end{lemma}
\begin{proof}
  Since $h_i$ is blocking, no local shadow path $D(h;\omega_H)$ for $h\in(\overline{h}_ih_{i-1+d})_H$ contains $v_{u_{m-1,2}}$.
  Thus Lemma~\ref{le:compatibility check} shows $\omega|_{D(h;\omega_H)}$ is compatible.
  By definition of local shadow paths, $\omega|_{D(h;\omega_H)}$ satisfies condition (a) of Proposition~\ref{prop:noncommutative collapse}.
  Condition (b) follows directly from Corollary~\ref{cor:maximality}.
  The conclusion immediately follows since the only edges in $D(h;\omega_H)$ for $h\in(\overline{h}_ih_{i-1+d})_H$ whose non-commutative weights have nontrivial coefficients are $h_{i-1+d}$ and $v_{u_{m-1,2}-t+1}$.
\end{proof}

This leads to the following result which is key to our induction argument.
\begin{corollary}
  \label{cor:incompatible collapse}
  Let $\omega:E_m\to\ZZ_{\ge0}$ be a piecewise compatible grading of $D_m$, $m\ge3$, which is not compatible.  
  Write $h_i$ for the blocking edge of $\omega_H$ and assume $f_{\omega_H}(h_iv_{u_{m-1,2}})=0$. 
  Set 
  \[d=|\supp(\omega_H)\cap h_iv_{u_{m-1,2}}|\quad\text{ and }\quad t=|\supp(\omega_V)\cap h_iv_{u_{m-1,2}}|.\]
  Then $Y_{h_iv_{u_{m-1,2}}}(\omega)=pYXY^{-1}X^{-1}$, where $p=p_{1,\omega_H(h_{i-1+d})}p_{2,d_2-\omega_V(v_{u_{m-1,2}-t+1})}$.
\end{corollary}
\begin{proof}
  We distinguish two cases as in the proof of Theorem~\ref{th:blocking edge conditions}.
  First consider the case $\hgt(h_i)\ge u_{m-1,2}-d_1$.
  In each of the possible cases from Lemma~\ref{le:Dyck tail inequality} we have $\supp(\omega_H)\cap h_iv_{u_{m-1,2}}=\{h_i\}$ and by assumption $\omega_H(h_i)=u_{m-1,2}-\hgt(h_i)$.
  We use the description of $\omega$ from the proof of Theorem~\ref{th:blocking edge conditions} in each case.

  For $m=3$, set $\delta=1$ if $h_i$ is immediately followed by a vertical edge and $\delta=0$ otherwise.  
  Then we have
  \begin{align*}
    Y_{h_iv_{u_{2,2}}}(\omega)
    &=\big(p_{1,\omega_H(h_i)}Y^{\omega_H(h_i)} X^{-1}\big)\big(XY^{-1}X^{-1}\big)^\delta\big(X^{-1}\big)^{\omega_V(v_{u_{m-1,2}-t+1})}\times\\
    &\quad\times\big(p_{2,d_2-\omega_V(v_{u_{m-1,2}-t+1})}X^{\omega_V(v_{u_{m-1,2}-t+1})+1}Y^{-1}X^{-1}\big)\Big[\big(X^{-1}\big)^{d_2}\big(X^{d_2+1}Y^{-1}X^{-1}\big)\Big]^{\omega_H(h_i)-\delta-2}\times\\
    &\quad\times\big(X^{-1}\big)^{d_2-1}\big(X^{d_2+1}Y^{-1}X^{-1}\big)\\
    &=p\big(Y^{\omega_H(h_i)-\delta-1}X^{-1}\big)\Big[XY^{-\omega_H(h_i)+\delta+2}X^{-1}\Big]\big(X^2Y^{-1}X^{-1}\big)\\
    &=pYXY^{-1}X^{-1}.
  \end{align*}
  For $m=4$, set $\delta=1$ if $h_i$ is immediately followed by a vertical edge and $\delta=0$ otherwise.
  Then we have
  \begin{align*}
    Y_{h_iv_{u_{3,2}}}(\omega)
    &=\big(Y^{d_1} X^{-1}\big)\big(XY^{-1}X^{-1}\big)^\delta\times\\
    &\quad\times\Big[\big(X^{-1}\big)^{\omega_V(v_{u_{m-1,2}-t+1})}\big(p_{2,d_2-\omega_V(v_{u_{m-1,2}-t+1})}X^{\omega_V(v_{u_{m-1,2}-t+1})+1}Y^{-1}X^{-1}\big)\Big]^{1-\delta}\times\\
    &\quad\times\Big[\big(X^{-1}\big)^{d_2}\big(X^{d_2+1}Y^{-1}X^{-1}\big)\Big]^{d_1-2}\big(X^{-1}\big)^{d_2-1}\big(X^{d_2+1}Y^{-1}X^{-1}\big)\\
    &=p\big(Y^{d_1-1}X^{-1}\big)\Big[XY^{-d_1+2}X^{-1}\Big]\big(X^2Y^{-1}X^{-1}\big)\\
    &=pYXY^{-1}X^{-1}.
  \end{align*}
  For $m\ge5$ in Lemma~\ref{le:Dyck tail inequality}, we have $p=1$ and so
  \begin{align*}
    Y_{h_iv_{u_{m-1,2}}}(\omega)
    &=\big(Y^{d_1}X^{-1}\big)\big(XY^{-1}X^{-1}\big)^{1+\delta_1}\Big[\big(X^{-1}\big)^{d_2}\big(X^{d_2+1}Y^{-1}X^{-1}\big)\Big]^{d_1-2-\delta_1}\big(X^{-1}\big)^{d_2-1}\big(X^{d_2+1}Y^{-1}X^{-1}\big)\\
    &=\big(Y^{d_1-1-\delta_1}X^{-1}\big)\Big[XY^{-d_1+2+\delta_1}X^{-1}\Big]\big(X^2Y^{-1}X^{-1}\big)\\
    &=YXY^{-1}X^{-1}.
  \end{align*}

  Now suppose $\hgt(h_i)<u_{m-1,2}-d_1$ so that $|\supp(\omega_H)\cap h_iv_{u_{m-1,2}}|>1$ and, by Theorem~\ref{th:blocking edge conditions}(c), $\hgt(h_{i+1})=\hgt(h_i)+\delta_1$.
  By Lemma~\ref{le:noncommutative collapse}, we have $Y_{D(h_{i+1};\omega_H)}(\omega)=pX^{-1}$.
  As in the proof of Theorem~\ref{th:blocking edge conditions}, we have $\omega_H(h_i)=d_1$ and $\omega_V(v_{u_{m-1,2}-\ell})=d_2$ for $0\le\ell\le d_1-\delta_1$.
  Combining these observations, we get
  \begin{align*}
    Y_{h_iv_{u_{m-1,2}}}(\omega)  
    &=\big(Y^{d_1}X^{-1}\big)\big(XY^{-1}X^{-1}\big)^{\delta_1}\big[pX^{-1}\big]\Big[\big(X^{-1}\big)^{d_2-1}\big(X^{d_2+1}Y^{-1}X^{-1}\big)\Big]\times\\
    &\qquad\times\Big[\big(X^{-1}\big)^{d_2}\big(X^{d_2+1}Y^{-1}X^{-1}\big)\Big]^{d_1-2-\delta_1}\big(X^{-1}\big)^{d_2-1}\big(X^{d_2+1}Y^{-1}X^{-1}\big)\\
    &=p\big(Y^{d_1-1-\delta_1}X^{-1}\big)\Big[XY^{-d_1+2+\delta_1}X^{-1}\Big]\big(X^2Y^{-1}X^{-1}\big)\\
    &=pYXY^{-1}X^{-1}.
  \end{align*}
\end{proof}

For $m\ge1$, we consider summands of $Y_{D_m}$ given as follows:
\begin{equation}
  \label{eq:Y decomposition}
  Y_{D_m}=\sum_{\omega_H:H_m\to[0,d_1]}Y_{D_m}(\omega_H),\qquad Y_{D_m}(\omega_H):=\sum_{\omega_V\in\cC(\omega_H)}Y_{D_m}(\omega_H,\omega_V).
\end{equation}
Our goal will be to understand the action of $F_{P_0}$ on each of these summands.
The first step is given by the following factorization results which allow for an induction argument.
\begin{lemma}
  \label{le:piecewise compatible factorization 1}
  Let $\omega_H:H_m\to[0,d_1]$ be a horizontal grading of $D_m$, $m\ge3$.
  Write 
  \[Y_{D_m}^{pc}(\omega_H)=\sum\limits_{\omega:E_m\to\ZZ_{\ge0}}Y_{D_m}(\omega),\]
  where the sum ranges over piecewise compatible gradings $\omega$ of $D_m$ for which $\omega|_{H_m}=\omega_H$.
  Then there is the following factorization:
  \begin{equation}
    \label{eq:piecewise compatible factorization 1}
    Y_{D_m}^{pc}(\omega_H)=Y_{D_{m-1}}(\omega_{H,1})Y_{D_{m-1}}(\omega_{H,2})\cdots Y_{D_{m-1}}(\omega_{H,d_m-1-\delta_m})pX^{|H_{m-2-\delta_m}|} Y_{D_{m-1}}(\omega_{H,d_m-\delta_m}),
  \end{equation}
  where 
  \[p=\begin{cases} p_{1,1}^{|V_{m-2-\delta_m}|-2|H_{m-2-\delta_m}|}p_{1,2}^{|H_{m-2-\delta_m}|-|V_{m-2-\delta_m}|} & \text{if $d_2=1$ and $m>3$;}\\ p_{1,1}^{-|V_{m-2-\delta_m}|} & \text{if $d_2>1$ or $m=3$.}\end{cases}\]
\end{lemma}
\begin{proof}
  By the assumptions on the horizontal grading $\omega_{H,d_m-\delta_m}$ of $D_{m-1}$ from Definition~\ref{def:agregate grading}, each term contributing to $Y_{D_{m-1}}(\omega_{H,d_m-\delta_m})$ begins with the monomial $p^{-1}X^{-|H_{m-2-\delta_m}|}$ associated to the initial $D_{m-2-\delta_m}$ subpath of $D_{m-1}$. 
  To see the coefficient $p^{-1}$, we observe the following:
  \begin{itemize}
    \item when $d_2>1$, there are $|V_{m-2-\delta_m}|$ horizontal edges of $D_{m-2-\delta_m}$ which are immediately followed by a single vertical edge and all other horizontal edges are not immediately followed by any vertical edges;
    \item when $d_2=1$, each horizontal edge is immediately followed by a vertical edge and so there are $|V_{m-2-\delta_m}|-|H_{m-2-\delta_m}|$ horizontal edges of $D_{m-2-\delta_m}$ which are immediately followed by exactly two vertical edges (see Corollary~\ref{cor:edge restrictions}) and the remaining $2|H_{m-2-\delta_m}|-|V_{m-2-\delta_m}|$ horizontal edges are immediately followed by a single vertical edge.
  \end{itemize}

  Using the notation of Definition~\ref{def:subpath edges}, for any grading $\omega:E_m\to\ZZ_{\ge0}$ there is the factorization
  \[Y_{D_m}(\omega)=Y_{h_{1,1}v_{u_{m-2,2},1}}(\omega)\cdots Y_{h_{1,d_m-1-\delta_m}v_{u_{m-2,2},d_m-1-\delta_m}}(\omega)Y_{h_{u_{m-2-\delta_m,1}+1,d_m-\delta_m}v_{u_{m-2,2},d_m-\delta_m}}(\omega).\]
  The result then immediately follows from the definition of piecewise compatible gradings in Definition~\ref{def:agregate grading}.
\end{proof}

Using Remark~\ref{rem:alternate restrictions} instead of Definition~\ref{def:agregate grading}, we obtain a similar factorization for piecewise compatible gradings of $D'_{m+1}$.
Below we use the notation $Y'_{D'_m}(\omega_{V'}):=\sum_{\omega_{H'}\in\cC(\omega_{V'})}Y'_{D'_m}(\omega_{H'},\omega_{V'})$ for a vertical grading $\omega_{V'}:V'_m\to[0,d_1]$.
Note that $d'_{m+1}=d_m$, $d'_m=d_{m-1}$, and so $\delta'_{m+1}=\delta_m$ when $m\ge3+\delta'_{m+1}$.
\begin{lemma}
  \label{le:piecewise compatible factorization 2}
  Let $\omega_{V'}:V'_{m+1}\to[0,d_1]$ be a vertical grading of $D'_{m+1}$ for $m\ge3+\delta'_{m+1}$.
  Write 
  \[{Y'}_{D'_{m+1}}^{pc}(\omega_{V'})=\sum\limits_{\omega':E'_{m+1}\to\ZZ_{\ge0}} Y'_{D'_{m+1}}(\omega'),\]
  where the sum ranges over piecewise compatible gradings $\omega'$ of $D'_{m+1}$ for which $\omega'|_{V'_{m+1}}=\omega_{V'}$.
  Then there is the following factorization:
  \begin{equation}
    \label{eq:piecewise compatible factorization 2}
    {Y'}_{D'_{m+1}}^{pc}(\omega_{V'})=Y'_{D'_m}(\omega_{V',1})Y'_{D'_m}(\omega_{V',2})\cdots Y'_{D'_m}(\omega_{V',d_m-1-\delta_m})pXY^{|V'_{m-1-\delta_m}|}X^{-1} Y'_{D'_m}(\omega_{V',d_m-\delta_m}),
  \end{equation}
  where 
  \[p=\begin{cases} p_{1,1}^{|H_{m-2-\delta_m}|-2|H_{m-2-\delta_m}|}p_{1,2}^{|H_{m-2-\delta_m}|-|V_{m-2-\delta_m}|} & \text{if $d_2=1$;}\\ p_{1,1}^{-|V_{m-2-\delta_m}|} & \text{if $d_2>1$.}\end{cases}\]
\end{lemma}
\begin{proof}
  By the assumptions on the vertical grading $\omega_{V',d_m-\delta_m}$ from Remark~\ref{rem:alternate restrictions}, each term contributing to $Y'_{D'_m}(\omega_{V',d_m-\delta_m})$ begins with the monomial $p^{-1}XY^{-|V'_{m-1-\delta_m}|}X^{-1}$ associated to the initial $D'_{m-1-\delta_m}$ subpath of $D'_m$.
  The coefficient $p^{-1}$ here can be seen as follows.  
  Applying Lemma~\ref{le:Dyck path recursion}(a), we see that the structure of $D'_{m-1-\delta_m}$ is determined by the structure of $D_{m-2-\delta_m}$ observed in the last part of the previous proof.
  More precisely, we have the following:
  \begin{itemize}
    \item when $d_2>1$, there are $|V_{m-2-\delta_m}|$ vertical edges of $D'_{m-1-\delta_m}$ which are immediately preceded by $d_1-1$ horizontal edges and all other vertical edges are immediately followed by $d_1$ horizontal edges;
    \item when $d_2=1$, there are $|V_{m-2-\delta_m}|-|H_{m-2-\delta_m}|$ vertical edges of $D'_{m-1-\delta_m}$ which are immediately preceded by $d_1-2$ horizontal edges and the remaining $2|H_{m-2-\delta_m}|-|V_{m-2-\delta_m}|$ vertical edges are immediately preceded by $d_1-1$ horizontal edges.
  \end{itemize}
  Then observe that in the computation of $Y'_{D'_m}(\omega_{V',d_m-\delta_m})$ the coefficients are given by $p'_{2,d_1-k}=p_{1,k}$ for $k=1,2$.
\end{proof}

The analogous factorization in the special case where $m=3$ and $\delta'_4=1$ is handled in the following result which is proven exactly as Lemma~\ref{le:piecewise compatible factorization 1}.
\begin{lemma}
  \label{le:piecewise compatible factorization 3}
  Suppose $d_2=1$.  Let $\omega_{V'}:V'_4\to[0,d_1]$ be a vertical grading of $D'_4$.
  Write 
  \[{Y'}_{D'_4}^{pc}(\omega_{V'})=\sum\limits_{\omega':E'_4\to\ZZ_{\ge0}} Y'_{D'_4}(\omega'),\]
  where the sum ranges over piecewise compatible gradings $\omega'$ of $D'_4$ for which $\omega'|_{V'_4}=\omega_{V'}$.
  Then there is the following factorization:
  \begin{equation}
    \label{eq:piecewise compatible factorization 3}
    {Y'}_{D'_4}^{pc}(\omega_{V'})=Y'_{D'_3}(\omega_{V',1})Y'_{D'_3}(\omega_{V',2})\cdots Y'_{D'_3}(\omega_{V',d_1-2})X Y'_{D'_3}(\omega_{V',d_1-1}).
  \end{equation}
\end{lemma}

The factorizations above concerned sums over piecewise compatible gradings.  Our goal is to understand sums over compatible gradings, however it will be easier to first focus on piecewise compatible gradings which are not compatible.
\begin{lemma}
  \label{le:incompatible factorization 1}
  Let $\omega_H:H_m\to[0,d_1]$ be a horizontal grading of $D_m$, $m\ge3$, for which there exists a vertical grading $\omega^*_V:V_m\to[0,d_2]$ of $D_m$ so that $(\omega_H,\omega^*_V)$ is piecewise compatible but not compatible.
  Write 
  \[Y_{D_m}^{nc}(\omega_H)=\sum\limits_{\omega:E_m\to\ZZ_{\ge0}}Y_{D_m}(\omega),\]
  where the sum ranges over piecewise compatible gradings $\omega$ of $D_m$ which are not compatible and satisfy $\omega|_{H_m}=\omega_H$.
  Let $h_i\in H_m$ denote the blocking edge of $\omega_H$ and set 
  \[d=|\supp(\omega_H)\cap h_iv_{u_{m-1,2}}|\quad\text{ and }\quad t=|\supp(\omega^*_V)\cap h_iv_{u_{m-1,2}}|.\]
  Let $s=\left\lceil\frac{i}{u_{m-1,1}}\right\rceil$ denote the index so that $h_i\in H_{m,s}$.
  Define a horizontal grading $\chi_H:H_{m-1}\to[0,d_1]$ of $D_{m-1}$ with $\chi_H(h)=\omega_{H,s}(h)$ for $h\in (h_{1,s}\overline{h}_i)_H$ and $\chi_H(h)=\ell$ for $h\in (h_iv_{u_{m-2,2},s})_H$ if $h$ is immediately followed by exactly $\ell$ vertical edges in this copy of $D_{m-1}$.
  Then there is the following factorization:
  \begin{equation}
    \label{eq:incompatible factorization 1}
    Y_{D_m}^{nc}(\omega_H)=Y_{D_{m-1}}(\omega_{H,1})\cdots Y_{D_{m-1}}(\omega_{H,s-1})Y_{D_{m-1}}(\chi_H)p_1X^{|h_iv_{u_{m-2,2},s}|_H} p_2YXY^{-1}X^{-1},
  \end{equation}
  where $p_2=p_{1,\omega_H(h_{i-1+d})}p_{2,d_2-\omega^*_V(v_{u_{m-1,2}-t+1})}$ and
  \[p_1=\begin{cases} p_{1,1}^{|h_iv_{u_{m-2,2},s}|_V-2|h_iv_{u_{m-2,2},s}|_H}p_{1,2}^{|h_iv_{u_{m-2,2},s}|_H-|h_iv_{u_{m-2,2},s}|_V} & \text{ if $d_2=1$;}\\ 
    p_{1,1}^{-|h_iv_{u_{m-2,2},s}|_V} & \text{if $d_2>1$.}\end{cases}\]
\end{lemma}
\begin{proof}
  Using the notation of Definition~\ref{def:subpath edges}, for any grading $\omega:E_m\to\ZZ_{\ge0}$ there is the factorization
  \[Y_{D_m}(\omega)=Y_{h_{1,1}v_{u_{m-2,2},1}}(\omega)\cdots Y_{h_{1,s-1}v_{u_{m-2,2},s-1}}(\omega)Y_{h_{1,s}\overline{h}_i}(\omega)Y_{h_iv_{u_{m-1,2}}}(\omega).\]
  By definition of $\chi_H$, every term of $Y_{D_{m-1}}(\chi_H)$ ends with the monomial $p_1^{-1}X^{-|h_iv_{u_{m-2,2},s}|_H}$.
  Theorem~\ref{th:blocking edge conditions} shows that any piecewise compatible grading $\omega:E_m\to\ZZ_{\ge0}$ of $D_m$ which is not compatible agrees with $(\omega_H,\omega^*_V)$ on the path $h_iv_{u_{m-1,2}}$.
  The result then immediately follows from Corollary~\ref{cor:incompatible collapse} and the definition of piecewise compatible gradings in Definition~\ref{def:agregate grading}.
\end{proof}

The next result gives an analogous factorization for sums over piecewise compatible gradings of $D'_{m+1}$ which are not compatible.
\begin{lemma}
  \label{le:incompatible factorization 2}
  Let $\omega_{V'}:V'_{m+1}\to[0,d_1]$ be a vertical grading of $D'_{m+1}$, $m\ge3$, for which there exists a horizontal grading $\omega^*_{H'}:H'_{m+1}\to[0,d_2]$ of $D'_{m+1}$ so that $(\omega^*_{H'},\omega_{V'})$ is piecewise compatible but not compatible.
  Write 
  \[{Y'}_{D'_{m+1}}^{nc}(\omega_{V'})=\sum\limits_{\omega:E'_{m+1}\to\ZZ_{\ge0}}Y'_{D'_{m+1}}(\omega),\]
  where the sum ranges over piecewise compatible gradings $\omega$ of $D'_{m+1}$ which are not compatible and satisfy $\omega|_{V'_{m+1}}=\omega_{V'}$.
  Let $h'_j\in H'_{m+1}$ denote the blocking edge of $\omega^*_{H'}$, where $\hgt(h'_j)=i-1$, and set 
  \[d=|\supp(\omega_{V'})\cap h'_jv'_{u'_{m,2}}|\quad\text{ and }\quad t=|\supp(\omega^*_H)\cap h'_jv'_{u'_{m,2}}|.\]
  Let $s=\left\lceil\frac{i}{u_{m-1,1}}\right\rceil$ denote the index so that $h'_j\in H'_{m+1,s}$.
  Define a horizontal grading $\chi_{V'}:V'_m\to[0,d_1]$ of $D'_m$ with $\chi_{V'}(v')=\omega_{V',s}(v')$ for $v'\in (h'_{1,s}\overline{h}'_j)_V$ and $\chi_{V'}(v')=\ell$ for $v'\in (h'_jv'_{u'_{m-1,2},s})_V$ if $v'$ is immediately preceded by exactly $\ell$ horizontal edges in this copy of $D'_m$.
  Then there is the following factorization:
  \begin{equation}
    \label{eq:incompatible factorization 2}
    {Y'}_{D'_{m+1}}^{nc}(\omega_{V'})=Y'_{D'_m}(\omega_{V',1})\cdots Y'_{D'_m}(\omega_{V',s-1})Y'_{D'_m}(\chi_{V'})p_1XY^{|h'_jv_{u'_{m-1,2},s}|_V}X^{-1} p_2YXY^{-1}X^{-1},
  \end{equation}
  where $p_2=p_{1,d_1-\omega_V(v'_{u'_{m,2}-d+1})}p_{2,d_2-\omega^*_H(h'_{j-1+t})}$ and
  \[p_1=\begin{cases} p_{1,1}^{|h_iv_{u_{m-2,2},s}|_V-2|h_iv_{u_{m-2,2},s}|_H}p_{1,2}^{|h_iv_{u_{m-2,2},s}|_H-|h_iv_{u_{m-2,2},s}|_V} & \text{ if $d_2=1$;}\\ 
    p_{1,1}^{-|h_iv_{u_{m-2,2},s}|_V} & \text{if $d_2>1$;}\end{cases}\]
  with $h_iv_{u_{m-2,2,s}}$ being the subpath in the $s$-th copy of $D_{m-1}$ inside $D_m$.
\end{lemma}
\begin{proof}
  By definition of $\chi_{V'}$, every term of $Y'_{D'_m}(\chi_{V'})$ ends with the monomial $p_1^{-1}XY^{-|h'_jv'_{u'_{m-1,2},s}|_V}X^{-1}$.
  To see the coefficient $p_1^{-1}$, note that by Lemma~\ref{le:Dyck path recursion} the structure of $D'_m$ is determined by the structure of $D_{m-1}$.

  Theorem~\ref{th:blocking edge conditions} shows that any piecewise compatible grading $\omega:E'_{m+1}\to\ZZ_{\ge0}$ of $D'_{m+1}$ which is not compatible agrees with $(\omega^*_{H'},\omega_{V'})$ on the path $h'_jv'_{u'_{m,2}}$.
  The result then immediately follows from Corollary~\ref{cor:incompatible collapse} and the definition of piecewise compatible gradings in Definition~\ref{def:agregate grading}.
\end{proof}

\subsection{Proof of Main Theorem}
We work by induction on $m\ge1$.
From the definition of the non-commutative weights in equation~\eqref{eq:edge weights}, we immediately see 
\[Y_{D_1}=P_1(Y)X^{-1}=F_{P_1}(Y)=Y_1\]
and
\[Y_{D_2}=\sum\limits_{\ell=0}^{d_2} (P_1(Y)X^{-1})^{d_2-\ell}(X^{-1})^\ell(p_{2,d_2-\ell}X^{\ell+1}Y^{-1}X^{-1})=P_2(P_1(Y)X^{-1})XY^{-1}X^{-1}=F_{P_1}F_{P_2}(Y)=Y_2.\]

Write $Y_{D_2}=\sum\limits_{\omega_H:H_2\to[0,d_1]} Y_{D_2}(\omega_H)$.
We will show that $F_{P_0}(Y_{D_2}(\omega_H))=Y'_{D'_3}(\varphi^*_2\omega_H)$ for each horizontal grading $\omega_H:H_2\to[0,d_1]$.
Fix a horizontal grading $\omega_H:H_2\to[0,d_1]$.
If $\supp(\omega_H)=\emptyset$, then 
\[Y_{D_2}(\omega_H)=(X^{-1})^{d_2}P_0(X)XY^{-1}X^{-1}\]
and so 
\begin{align*}
  F_{P_0}(Y_{D_2}(\omega_H))
  &=(XY^{-1}X^{-1})^{d_2}(XP_0(Y)X^{-1})(XYX^{-1})(P_0(Y)X^{-1})^{-1}(XY^{-1}X^{-1})\\
  &=(XY^{-1}X^{-1})^{d_2}XYXY^{-1}X^{-1}\\
  &=(XY^{-1}X^{-1})^{d_2-1}X^2Y^{-1}X^{-1}\\
  &=Y'_{D'_3}(\varphi^*_2\omega_H).
\end{align*}
Suppose $\supp(\omega_H)\ne\emptyset$.
Let $h_i$ be the last horizontal edge in $\supp(\omega_H)$.
Then $(\omega_H,\omega_V)$ will be compatible if and only if $\omega_V(v_1)\le d_2-i$.
This gives
\[Y_{D_2}(\omega_H)=(p_{1,\omega_H(h_1)}Y^{\omega_H(h_1)}X^{-1})\cdots(p_{1,\omega_H(h_i)}Y^{\omega(h_i)}X^{-1})(X^{-1})^{d_2-i}\sum\limits_{\ell=0}^{d_2-i} p_{2,d_2-\ell}X^{\ell+1}Y^{-1}X^{-1}.\]
Applying $F_{P_0}$ gives
\begin{align*}
  F_{P_0}(Y_{D_2}(\omega_H))
  &=\big(p_{1,\omega_H(h_1)}(P_0(Y)X^{-1})^{\omega_H(h_1)}XY^{-1}X^{-1}\big)\cdots\big(p_{1,\omega_H(h_i)}(P_0(Y)X^{-1})^{\omega_H(h_i)}XY^{-1}X^{-1}\big)\times\\
  &\quad\times[XY^{-1}X^{-1}]^{d_2-i}\sum\limits_{\ell=0}^{d_2-i} p_{2,d_2-\ell}(XYX^{-1})^{\ell+1}(P_0(Y)X^{-1})^{-1}(XY^{-1}X^{-1})\\
  &=(P_0(Y)X^{-1})^{\omega_H(h_1)}(X^{-1})^{d_1-\omega_H(h_1)}(p_{1,\omega_H(h_1)}X^{d_1-\omega_H(h_1)+1}Y^{-1}X^{-1})\cdots (P_0(Y)X^{-1})^{\omega_H(h_i)-1}\times\\
  &\quad\times\left(\sum\limits_{\ell=0}^{d_2-i} p_{2,d_2-\ell} Y^\ell X^{-1}\right)(X^{-1})^{d_1-\omega_H(h_i)}(p_{1,\omega_H(h_i)}X^{d_1-\omega_H(h_i)+1}Y^{-1}X^{-1})\times\\
  &\quad\times\big[(X^{-1})^{d_1}(X^{d_1+1}Y^{-1}X^{-1})\big]^{d_2-i-1}(X^{-1})^{d_1-1}(X^{d_1+1}Y^{-1}X^{-1})\\
  &=Y'_{D'_3}(\varphi^*_2\omega_H).
\end{align*}

Suppose $m\ge3$ and let $\omega_H:H_m\to[0,d_1]$ be a horizontal grading of $D_m$.
Following Theorem~\ref{th:blocking edge conditions}, there are two cases to consider.
\begin{enumeratea}
  \item Suppose that $(\omega_H,\omega_V)$ is compatible for every piecewise compatible grading $(\omega_H,\omega_V)$ of $D_m$.
    Then Lemma~\ref{le:piecewise compatible factorization 1} shows there is the factorization
    \[\tag{$\dagger$}Y_{D_m}(\omega_H)=Y_{D_{m-1}}(\omega_{H,1})Y_{D_{m-1}}(\omega_{H,2})\cdots Y_{D_{m-1}}(\omega_{H,d_m-1-\delta_m})pX^{|H_{m-2-\delta_m}|} Y_{D_{m-1}}(\omega_{H,d_m-\delta_m}),\]
    where 
    \[p=\begin{cases} p_{1,1}^{|V_{m-2-\delta_m}|-2|H_{m-2-\delta_m}|}p_{1,2}^{|H_{m-2-\delta_m}|-|V_{m-2-\delta_m}|} & \text{if $d_2=1$;}\\ p_{1,1}^{-|V_{m-2-\delta_m}|} & \text{if $d_2>1$.}\end{cases}\]
    If $m\ge3+\delta'_{m+1}$, we may apply Lemma~\ref{le:piecewise compatible factorization 2} to conclude by induction that $F_{P_0}(Y_{D_m}(\omega_H))=Y'_{D'_{m+1}}(\varphi^*_m\omega_H)$.

    It remains to consider the case $m=3$ with $\delta'_4=1$, i.e. $d_2=1$.
    In this case $D_2$ consists of a single horizontal edge followed by a single vertical edge and, by Corollary~\ref{cor:recursive dyck path structure}(c), $D_3$ consists of $d_1-1$ copies of $D_2$ followed by a single vertical edge, in particular $D_3$ ends with two consecutive vertical edges.
    The factorization ($\dagger$) still holds in this case and by induction we have $F_{P_0}(Y_{D_2}(\omega_{H,r}))=Y'_{D'_3}(\varphi^*_2\omega_{H,r})$ for $1\le r\le d_1$.
    In particular, to see that $F_{P_0}(Y_{D_3}(\omega_H))=Y'_{D'_4}(\varphi^*_3\omega_H)$ it suffices by Lemma~\ref{le:piecewise compatible factorization 3} to compare $F_{P_0}\big(Y_{D_2}(\omega_{H,d_1-1})pXY_{D_2}(\omega_{H,d_1})\big)$ with $XY'_{D'_3}(\omega_{V',d_1-1})$, where we write $\omega_{V'}=\varphi^*_3\omega_H$.

    There are two cases to consider.  If $\omega(h_{d_1-1})=0$, we have $Y_{D_2}(\omega_{H,d_1-1})=(X^{-1})P_0(X)XY^{-1}X^{-1}$ and so
    \[F_{P_0}(Y_{D_2}(\omega_{H,d_1-1}))=(XY^{-1}X^{-1})(XP_0(Y)X^{-1})(XYX^{-1})(P_0(Y)X^{-1})^{-1}XY^{-1}X^{-1}=X^2Y^{-1}X^{-1}.\]
    The same calculation shows $F_{P_0}\big(Y_{D_2}(\omega_{H,d_1})\big)=X^2Y^{-1}X^{-1}$ by the assumptions on $\omega_{H,d_1}$ in Definition~\ref{def:agregate grading}.
    But then, since $p=1$ in this case, we have
    \begin{align*}
      F_{P_0}\big(Y_{D_2}(\omega_{H,d_1-1})pXY_{D_2}(\omega_{H,d_1})\big)
      &=(X^2Y^{-1}X^{-1})(XYX^{-1})(X^2Y^{-1}X^{-1})\\
      &=X^3Y^{-1}X^{-1}\\
      &=X(X^{-1})^{d_1-1}(X^{d_1+1}Y^{-1}X^{-1}),
    \end{align*}
    which is exactly $XY'_{D'_3}(\omega_{V',d_1-1})$.
    
    When $h_{d_1-1}\in\supp(\omega_H)$, we have $Y_{D_2}(\omega_{H,d_1-1})=p_{1,\omega(h_{d_1-1}}Y^{\omega(h_{d_1-1})-1}X^{-1}$ so that
    \[F_{P_0}\big(Y_{D_2}(\omega_{H,d_1-1})\big)=(P_0(Y)X^{-1})^{\omega(h_{d_1-1})-1}(X^{-1})^{d_1-\omega(h_{d_1-1})}p_{1,\omega(h_{d_1-1})}X^{d_1-\omega(h_{d_1-1})+1}Y^{-1}X^{-1}.\]
    We saw above that $F_{P_0}\big(Y_{D_2}(\omega_{H,d_1})\big)=X^2Y^{-1}X^{-1}$ and so
    \begin{align*}
      &F_{P_0}\big(Y_{D_2}(\omega_{H,d_1-1})pXY_{D_2}(\omega_{H,d_1})\big)\\
      &\qquad=(P_0(Y)X^{-1})^{\omega(h_{d_1-1})-1}(X^{-1})^{d_1-\omega(h_{d_1-1})-1}p_{1,\omega(h_{d_1-1})}X^{d_1-\omega(h_{d_1-1})+1}Y^{-1}X^{-1},
    \end{align*}
    which is exactly $XY'_{D'_3}(\omega_{V',d_1-1})$.
    The claim then follows by induction from Lemma~\ref{le:piecewise compatible factorization 3}.
  \item Suppose there exists a vertical grading $\omega_V^*:V_m\to[0,d_2]$ of $D_m$ so that $(\omega_H,\omega_V^*)$ is piecewise compatible, but not compatible. 
    By Theorem~\ref{th:blocking edge conditions}, there must exist a blocking edge $h_i$ for $\omega_H$.
    Set 
    \[d=|\supp(\omega_H)\cap h_iv_{u_{m-1,2}}|\quad\text{ and }\quad t=|\supp(\omega_V^*)\cap h_iv_{u_{m-1,2}}|.\]

    By Proposition~\ref{prop:compatibility equivalence} and Proposition~\ref{prop:piecewise equivalence}, $(\Omega_m\omega^*_V,\varphi^*_m\omega_H)$ is a piecewise compatible grading of $D'_{m+1}$ which is not compatible.
    Let $D(v'_{u'_{m,2}};\varphi^*_m\omega_H)=h'_jv'_{u'_{m,2}}$ and observe that $\hgt(h'_j)=i-1$ by definition of $\Omega_m$.
    By Proposition~\ref{prop:strong justification implication}, Lemma~\ref{le:remote shadow cardinalities}, and Corollary~\ref{cor:support and remote shadow}, we have
    \[|\supp(\Omega_m\omega^*_V)\cap h'_jv'_{u'_{m,2}}|=|\rsh(\varphi^*_m\omega_H)\cap h'_jv'_{u'_{m,2}}|=|\rsh(\omega_H)\cap h_iv_{u_{m-1,2}}|=|\supp(\omega_V)\cap h_iv_{u_{m-1,2}}|.\]
    Moreover, we have
    \[|\supp(\varphi^*_m\omega_H)\cap h'_jv'_{u'_{m,2}}|=u_{m,1}-i+1-|\supp(\omega_H)\cap h_iv_{u_{m-1,2}}|+\delta,\]
    where $\delta=0$ if $\omega_H(h_{i-1+d})=d_1$ and $\delta=1$ otherwise.
    If then follows from the definitions of $\varphi^*_m$ and $\Omega_m$ that the coefficients $p_2$ agree in Lemma~\ref{le:incompatible factorization 1} and Lemma~\ref{le:incompatible factorization 2}.

    Using the notation of Lemma~\ref{le:piecewise compatible factorization 1} and Lemma~\ref{le:incompatible factorization 1}, we have $Y_{D_m}(\omega_H)=Y_{D_m}^{pc}(\omega_H)-Y_{D_m}^{nc}(\omega_H)$.
    By induction we have $F_{P_0}(Y_{D_{m-1}}(\omega_{H,r}))=Y'_{D'_m}(\varphi^*_{m-1}\omega_{H,r})$ for $1\le r\le d_m-\delta_m$ and $F_{P_0}(Y_{D_{m-1}}(\chi_H))=Y'_{D'_m}(\chi_{V'})$.
    It follows that $F_{P_0}\big(Y_{D_m}^{pc}(\omega_H)\big)={Y'}_{D'_{m+1}}^{pc}(\varphi^*_m\omega_H)$ and $F_{P_0}\big(Y_{D_m}^{nc}(\omega_H)\big)={Y'}_{D'_{m+1}}^{nc}(\varphi^*_m\omega_H)$.
    Since $Y'_{D'_{m+1}}(\varphi^*_m\omega_H)={Y'}_{D'_{m+1}}^{pc}(\varphi^*_m\omega_H)-{Y'}_{D'_{m+1}}^{nc}(\varphi^*_m\omega_H)$, the result follows.
\end{enumeratea}

\begin{remark}
  Our proof of the Main Theorem developed a combinatorial model for the analogue \eqref{eq:non-commutative initial cluster mutation} of initial cluster mutations.
  It would be interesting and highly non-trivial to understand the direct combinatorial interpretation for the non-commutative exchange relations \eqref{eq:non-commutative exchange relation}.
\end{remark}

\section{Specializations}
\label{sec:specialization}

In this section we consider the specialization to quantum generalized cluster variables.
Assume $v\in\kk$ is transcendental over $\QQ$.
Define the quantum torus algebra $\cT:=\cT_v=\kk\langle Z_1,Z_2:Z_1Z_2=v^2Z_2Z_1\rangle$ and let $\cF$ denote the skew-field of fractions of $\cT$.
It will be convenient to consider elements $Z^\bfa:=v^{-a_1a_2}Z_1^{a_1}Z_2^{a_2}$ for $\bfa=(a_1,a_2)\in\ZZ^2$, these form a $\kk$-basis of $\cT$.

Recall the notation \eqref{eq:reversed polynomials} for the polynomials $P_k$, $k\in\ZZ$.
Consider \emph{quantum generalized cluster variables} $Z^{(\alpha)}_k\in\cF$, $\alpha,k\in\ZZ$, defined recursively by
\begin{equation}
  \label{eq:exchange relations}
  Z^{(\alpha)}_1=Z_1,\quad Z^{(\alpha)}_2=Z_2,\quad Z^{(\alpha)}_{k-1}Z^{(\alpha)}_{k+1}=P_{\alpha+k}(vZ^{(\alpha)}_k).
\end{equation}
Observe that equation~\eqref{eq:exchange relations} immediately implies $Z^{(\alpha)}_kZ^{(\alpha)}_{k+1}=v^2Z^{(\alpha)}_{k+1}Z^{(\alpha)}_k$ for all $\alpha,k\in\ZZ$.

For a fixed $\alpha\in\ZZ$, the \emph{quantum generalized cluster algebra} $\cA^{(\alpha)}_v(P_1,P_2)\subset\cF$ is the $\kk$-subalgebra generated by the $Z^{(\alpha)}_k$, $k\in\ZZ$.
Although they are defined as elements of $\cF$, the quantum generalized cluster variables actually live in $\cT$.
We give a direct proof here, however the combinatorial construction below provides an alternate proof.
See \cite{bai-chen-ding-xu} for a proof of this result in the special case when $P_1=\overline{P}_1$ and $P_2=\overline{P}_2$.
\begin{theorem}
  Each quantum generalized cluster variable $Z^{(\alpha)}_k$ is an element of $\cT\subset\cF$.
\end{theorem}
\begin{proof}
  Consider the monomial $v^{d_{\alpha+k}}Z^{(\alpha)}_{k-1}(Z^{(\alpha)}_{k+2})^{d_{\alpha+k}}$.
  Expanding $Z^{(\alpha)}_{k-1}$ in terms of $Z^{(\alpha)}_k$ and $Z^{(\alpha)}_{k+1}$ using equation~\eqref{eq:exchange relations} gives $v^{d_{\alpha+k}}Z^{(\alpha)}_{k-1}(Z^{(\alpha)}_{k+2})^{d_{\alpha+k}}$ as
  \begin{align*}
    &v^{d_{\alpha+k}}P_{\alpha+k}(vZ^{(\alpha)}_k)(Z^{(\alpha)}_{k+1})^{-1}(Z^{(\alpha)}_{k+2})^{d_{\alpha+k}}\\
    &\quad=v^{-d_{\alpha+k}}P_{\alpha+k}(vZ^{(\alpha)}_k)(Z^{(\alpha)}_{k+2})^{d_{\alpha+k}}(Z^{(\alpha)}_{k+1})^{-1}\\
    &\quad=\sum\limits_{i=0}^{d_{\alpha+k}}p_{\alpha+k,i}v^{-d_{\alpha+k}+i}\big[(Z^{(\alpha)}_k)^i(Z^{(\alpha)}_{k+2})^i-1\big](Z^{(\alpha)}_{k+2})^{d_{\alpha+k}-i}(Z^{(\alpha)}_{k+1})^{-1}+P_{\alpha+k+2}(v^{-1}Z^{(\alpha)}_{k+2})(Z^{(\alpha)}_{k+1})^{-1}\\
    &\quad=\sum\limits_{i=0}^{d_{\alpha+k}}p_{\alpha+k,i}v^{d_{\alpha+k}-i}\big[(Z^{(\alpha)}_k)^i(Z^{(\alpha)}_{k+2})^i-1\big](Z^{(\alpha)}_{k+1})^{-1}(Z^{(\alpha)}_{k+2})^{d_{\alpha+k}-i}+(Z^{(\alpha)}_{k+1})^{-1}P_{\alpha+k+2}(vZ^{(\alpha)}_{k+2}).
  \end{align*}
  But for $0\le i\le d_{\alpha_k}$, the term $(Z^{(\alpha)}_k)^i(Z^{(\alpha)}_{k+2})^i-1$ above is a polynomial in $Z^{(\alpha)}_{k+1}$ with no constant term and so $\big[(Z^{(\alpha)}_k)^i(Z^{(\alpha)}_{k+2})^i-1\big](Z^{(\alpha)}_{k+1})^{-1}$ is a polynomial in $Z^{(\alpha)}_{k+1}$.
  Thus we may solve for $Z^{(\alpha)}_{k+3}=(Z^{(\alpha)}_{k+1})^{-1}P_{\alpha+k+2}(vZ^{(\alpha)}_{k+2})$ above and see that this generalized cluster variable can be written as a polynomial in $\kk[Z^{(\alpha)}_{k-1},Z^{(\alpha)}_k,Z^{(\alpha)}_{k+1},Z^{(\alpha)}_{k+2}]$.
  A similar calculation shows $Z^{(\alpha)}_{k-2}\in\kk[Z^{(\alpha)}_{k-1},Z^{(\alpha)}_k,Z^{(\alpha)}_{k+1},Z^{(\alpha)}_{k+2}]$.
  Then by induction we see $Z^{(\alpha)}_k\in\kk[Z^{(\alpha)}_0,Z^{(\alpha)}_1,Z^{(\alpha)}_2,Z^{(\alpha)}_3]\subset\cT$ for all $\alpha,k\in\ZZ$.
\end{proof}
\begin{remark}
  The proof above actually shows more.
  We see from this proof that 
  \[\cA_v^{(\alpha)}(P_1,P_2)=\kk[Z^{(\alpha)}_{k-1},Z^{(\alpha)}_k,Z^{(\alpha)}_{k+1},Z^{(\alpha)}_{k+2}]\]
  for each $\alpha,k\in\ZZ$.
\end{remark}

Define the quantum specialization $\pi_v:\KK\to\cF$ by
\begin{equation}
  \label{eq:quantum specializations}
  \pi_v(X)=vZ_1,\quad\pi_v(Y)=v^{-1}Z_2.
\end{equation}
Note that for $Q=XYX^{-1}Y^{-1}$ we have $\pi_v(Q)=v^2$.
For $\alpha\in\ZZ$, set $X^{(\alpha)}_0=X$ and for $m\ge1$ define elements $X^{(\alpha)}_m,X^{(\alpha)}_{-m}\in\KK$ by
\[X^{(\alpha)}_m=F_{P_{\alpha+1}}F_{P_{\alpha+2}}\cdots F_{P_{\alpha+m}}(X)\quad\text{and}\quad X^{(\alpha)}_{-m}=F_{P_{\alpha}}^{-1}F_{P_{\alpha-1}}^{-1}\cdots F_{P_{\alpha-m+1}}^{-1}(X)\]
and observe that Theorem~\ref{th:combinatorial construction} provides a combinatorial construction of each $X^{(\alpha)}_m$.
The following specialization result will provide a combinatorial construction of the quantum generalized cluster variables $Z^{(\alpha)}_m$.
\begin{theorem}
  \label{th:quantum specialization}
  For $m,\alpha\in\ZZ$, we have $\pi_v(X^{(\alpha+1)}_m)=vZ^{(\alpha)}_{m+1}$.
\end{theorem}
\begin{proof}
  We work by induction on $m$.
  Since $X^{(\alpha)}_0=X$ and $X^{(\alpha)}_1=QY$ for all $\alpha\in\ZZ$, the cases $m=0,1$ follow immediately from equation~\eqref{eq:quantum specializations}.

  For any nonzero polynomial $P\in\kk[z]$, define a $\kk$-linear automorphism $\mu_P:\cF\to\cF$ given by $\mu_P(Z_1)=Z_2$ and $\mu_P(Z_2)=Z_1^{-1}P(vZ_2)$.
  These satisfy the functional identities $\pi_v\circ F_P=\mu_P\circ\pi_v$.
  Note that $\mu_P^{-1}(Z_1)=P(vZ_1)Z_2^{-1}$ and $\mu_P^{-1}(Z_2)=Z_1$ so that $\pi_v\circ F_P^{-1}=\mu_P^{-1}\circ\pi_v$.
  
  Moreover, observe that $\mu_{P_{\alpha+2}}(Z_2)=Z^{(\alpha)}_3$ and $\mu_{P_{\alpha+1}}^{-1}(Z_1)=Z^{(\alpha)}_0$ for $\alpha\in\ZZ$.
  By the symmetry of the exchange relations \eqref{eq:exchange relations}, these imply $\mu_{P_{\alpha+2}}(Z^{(\alpha+1)}_m)=Z^{(\alpha)}_{m+1}$ and $\mu_{P_{\alpha+1}}^{-1}(Z^{(\alpha-1)}_{m+2})=Z^{(\alpha)}_{m+1}$ for any $\alpha,m\in\ZZ$.
  Indeed, by induction on $m\ge3$ we have
  \[Z^{(\alpha)}_{m+1}=(Z^{(\alpha)}_{m-1})^{-1}P_{\alpha+m}(vZ^{(\alpha)}_m)=\mu_{P_{\alpha+2}}\big((Z^{(\alpha+1)}_{m-2})^{-1}P_{\alpha+m}(vZ^{(\alpha+1)}_{m-1})\big)=\mu_{P_{\alpha+2}}(Z^{(\alpha+1)}_m).\]
  Similarly, by induction on $m\le-2$ we have
  \[Z^{(\alpha)}_{m+1}=(Z^{(\alpha)}_{m-1})^{-1}P_{\alpha+m}(vZ^{(\alpha)}_m)=\mu_{P_{\alpha+1}}^{-1}\big((Z^{(\alpha-1)}_m)^{-1}P_{\alpha+m}(vZ^{(\alpha-1)}_{m+1})\big)=\mu_{P_{\alpha+1}}^{-1}(Z^{(\alpha-1)}_{m+2}).\]

  Thus, by induction on $m\ge1$ we see
  \[\pi_v(X^{(\alpha+1)}_m)=\pi_v(F_{P_{\alpha+2}}(X^{(\alpha+2)}_{m-1}))=\mu_{P_{\alpha+2}}(\pi_v(X^{(\alpha+2)}_{m-1}))=\mu_{P_{\alpha+2}}(vZ^{(\alpha+1)}_m)=vZ^{(\alpha)}_{m+1}\]
  and by induction on $m\le-1$ we see
  \[\pi_v(X^{(\alpha+1)}_m)=\pi_v(F_{P_{\alpha+1}}^{-1}(X^{(\alpha)}_{m+1}))=\mu_{P_{\alpha+1}}^{-1}(\pi_v(X^{(\alpha)}_{m+1}))=\mu_{P_{\alpha+1}}^{-1}(vZ^{(\alpha-1)}_{m+2})=vZ^{(\alpha)}_{m+1}.\]
\end{proof}
Applying the quantum specialization $\pi_v$ to Theorem~\ref{th:combinatorial construction}, Theorem~\ref{th:quantum specialization} gives the following combinatorial construction of the quantum generalized cluster variables $Z^{(\alpha)}_m$ as pseudo-positive Laurent polynomials.
For notational convenience, we restrict to the quantum generalized cluster variables $Z_m:=Z^{(3)}_m$.

\begin{corollary}\mbox{}
  \label{cor:quantum combinatorial}
  \begin{enumeratea}
    \item For $m\ge3$, the quantum generalized cluster variable $Z_m$ is computed as follows:
      \begin{equation}
        \label{eq:combinatorial quantum 1}
        Z_m=\sum\limits_{\omega:E_{m-2}\to\ZZ_{\ge0}}p_\omega v^{1-u_{m-2,1}-u_{m-3,2}+\gamma_\omega+\beta_\omega}Z^{(-u_{m-2,1}+|\omega_V|,-u_{m-3,2}+|\omega_H|)},
      \end{equation}
      where 
      \begin{itemize}
        \item the sum ranges over $(d_1,d_2)$-bounded compatible gradings $\omega$ of $D_{m-2}$;
        \item $p_\omega=\prod_{i=1}^{u_{m-2,1}}p_{1,\omega_H(h_i)}\prod_{t=1}^{u_{m-3,2}}p_{2,d_2-\omega_V(v_t)}$;
        \item $\gamma_\omega=\sum\limits_{e<e'\in E_{m-2}}\gamma_\omega(e,e')$ for
        \begin{equation}
          \label{eq:quantum weight 1}
          \gamma_\omega(e,e')=\begin{cases}
            0 & \text{if $e\in H_{m-2}\setminus\supp(\omega_H)$ or $e'\in V_{m-2}\setminus\supp(\omega_V)$;}\\ 
            -2\omega(e)\omega(e') & \text{if $e\in\supp(\omega_H)$ and $e'\in\supp(\omega_V)$;}\\ 
            2\omega(e) & \text{if $e\in\supp(\omega_H)$ and $e'\in H_{m-2}$;}\\ 
            2\omega(e') & \text{if $e\in V_{m-2}$ and $e'\in\supp(\omega_V)$;}\\ 
            -2 & \text{if $e\in V_{m-2}$ and $e'\in H_{m-2}$;}
          \end{cases}
        \end{equation}
        \item $\beta_\omega=\sum\limits_{e<e'\in E_{m-2}}\beta_\omega(e,e')$ for
        \begin{equation}
          \label{eq:quantum weight 2}
          \beta_\omega(e,e')=\begin{cases}
            \omega(e)\omega(e')+1 & \text{if $e\in H_{m-2}$ and $e'\in V_{m-2}$ or $e\in V_{m-2}$ and $e'\in H_{m-2}$;}\\ 
            -(\omega(e)+\omega(e')) & \text{if $e,e'\in H_{m-2}$ or $e,e'\in V_{m-2}$.} 
          \end{cases}
        \end{equation}
      \end{itemize}
    \item For $m\le0$, the quantum generalized cluster variable $Z_m$ is computed as follows:
      \begin{equation}
        \label{eq:combinatorial quantum 2}
        Z_m=\sum\limits_{\omega:E'_{-m+1}\to\ZZ_{\ge0}}p'_\omega v^{-1+u'_{-m+1,1}+u'_{-m,2}+\gamma'_\omega+\beta'_\omega}Z^{(-u'_{-m,2}+|\omega|_{H'},-u'_{-m+1,1}+|\omega|_{V'})},
      \end{equation}
      where 
      \begin{itemize}
        \item the sum ranges over $(d_2,d_1)$-bounded compatible gradings $\omega$ of $D'_{-m+1}$;
        \item $p'_\omega=\prod_{i=1}^{u'_{-m+1,1}}p_{2,d_2-\omega_{H'}(h'_i)}\prod_{t=1}^{u'_{-m,2}}p_{1,\omega_{V'}(v'_t)}$;
        \item $\gamma'_\omega=\sum\limits_{e<e'\in E'_{-m+1}}\gamma'_\omega(e,e')$ for
        \begin{equation*}
          \gamma'_\omega(e,e')=\begin{cases}
            0 & \text{if $e\in V'_{-m+1}\setminus\supp(\omega_{V'})$ or $e'\in H'_{-m+1}\setminus\supp(\omega_{H'})$;}\\ 
            -2\omega(e)\omega(e') & \text{if $e\in\supp(\omega_{V'})$ and $e'\in\supp(\omega_{H'})$;}\\ 
            2\omega(e) & \text{if $e\in\supp(\omega_{V'})$ and $e'\in V'_{-m+1}$;}\\ 
            2\omega(e') & \text{if $e\in H'_{-m+1}$ and $e'\in\supp(\omega_{H'})$;}\\ 
            -2 & \text{if $e\in H'_{-m+1}$ and $e'\in V'_{-m+1}$;}
          \end{cases}
        \end{equation*}
        \item $\beta'_\omega=\sum\limits_{e<e'\in E'_{-m+1}}\beta'_\omega(e,e')$ for
        \begin{equation*}
          \beta'_\omega(e,e')=\begin{cases}
            \omega(e)\omega(e')+1 & \text{if $e\in H'_{-m+1}$ and $e'\in V'_{-m+1}$ or $e\in V'_{-m+1}$ and $e'\in H'_{-m+1}$;}\\ 
            -(\omega(e)+\omega(e')) & \text{if $e,e'\in H'_{-m+1}$ or $e,e'\in V'_{-m+1}$.} 
          \end{cases}
        \end{equation*}
    \end{itemize}
  \end{enumeratea}
\end{corollary}
\begin{proof}
  We prove part (a), the proof of part (b) is essentially the same where the roles of $X$ and $Y$ are interchanged in equation~\eqref{eq:edge weights}.

  First note that we have $X_{m-1}=QY_{m-2}$ so that $vZ_m=\pi_v(X_{m-1})=v^2\pi_v(Y_{m-2})$, in particular this accounts for the $1$ appearing in the exponent of $v$ in equation~\eqref{eq:combinatorial quantum 1}.
  By Theorem~\ref{th:combinatorial construction}, we may compute $Y_{m-2}$ by considering compatible gradings on the maximal Dyck path $D_{m-2}$ and thus $Z_m$ can be computed by applying the quantum projection $\pi_v$ to equation~\eqref{eq:total path weights}.
  Then the exponents of $Z_1$ and $Z_2$ in equation~\eqref{eq:combinatorial quantum 1} are immediate from Lemma~\ref{le:shadows and degrees}.
  The coefficient $p_\omega$ also follows directly from the definition of the non-commutative edge weights in equation~\eqref{eq:edge weights}, so for the remainder of the proof we assume $p_{i,j}=1$ for all $i$ and $j$.

  Note that 
  \begin{equation}
    \label{eq:monomial multiplication}
    Z^{(a_1,a_2)}Z^{(b_1,b_2)}=v^{a_1b_2-a_2b_1}Z^{(a_1+b_1,a_2+b_2)}
  \end{equation}
  for $a_i,b_i\in\ZZ$, $i=1,2$.
  The rest of the exponent of $v$ in equation~\eqref{eq:combinatorial quantum 1} can be seen as follows:
  \begin{enumeratei}
    \item for an edge $e\in E_{m-2}$ we have 
      \begin{align*}
        \pi_v(\wt_\omega(e))
        &=\begin{cases}\pi_v(Y^{\omega(e)}X^{-1}) & \text{if $e\in H_{m-2}$}\\ \pi_v(X^{\omega(e)+1}Y^{-1}X^{-1}) & \text{if $e\in V_{m-2}$} \end{cases}\\
        &=\begin{cases}v^{-\omega_H(e)-1}Z_2^{\omega_H(e)}Z_1^{-1}i & \text{if $e\in H_{m-2}$}\\ v^{\omega(e)+1}Z_1^{\omega(e)+1}Z_2^{-1}Z_1^{-1} & \text{if $e\in V_{m-2}$}\end{cases}\\
        &=\begin{cases}v^{-1}Z^{(-1,\omega_H(e))} & \text{if $e\in H_{m-2}$;}\\ v^{-1}Z^{(\omega(e),-1)} & \text{if $e\in V_{m-2}$;}\end{cases}
      \end{align*}
      The $v^{-1}$ in each possibility above accounts for the terms $-u_{m,1}$ and $-u_{m-1,2}$ in equation~\eqref{eq:combinatorial quantum 1}.
    \item for $e,e'\in E_{m-2}$, the quantity $\gamma_\omega(e,e')$ from equation~\eqref{eq:quantum weight 1} records the power of $v$ which appears when commuting powers of $Z_2$ appearing in $\pi_v(\wt_\omega(e))$ past powers of $Z_1$ appearing in $\pi_v(\wt_\omega(e'))$;
    \item for $e,e'\in E_{m-2}$, the quantity $\beta_\omega(e,e')$ from equation~\eqref{eq:quantum weight 2} records the power of $v$ so that 
      \[\pi_v(\wt_\omega(e))\pi_v(\wt_\omega(e'))=v^{\gamma_\omega(e,e')+\beta_\omega(e,e')-2}Z^{(a_1,a_2)}\]
      for appropriate $a_1,a_2\in\ZZ$ depending on $e,e'\in E_{m-2}$ (the $-2$ here accounts for part (i) above).
  \end{enumeratei}
  Since we have $Z_m=v\pi_v(Y_{D_{m-2}})$, the result follows by combining the observations above.
\end{proof}

Let $\kk=\QQ(v)$ for an indeterminate $v$.  
When $p_{i,j}=0$ for $i=1,2$ and $1\le j\le d_i-1$, the expansions of the quantum generalized cluster variables as elements of $\cT$ have been computed \cite{rupel0} using the representation theory of valued quivers as follows.
In this case, we drop the adjective ``generalized'' and refer to the $Z^{(\alpha)}_k$ simply as \emph{quantum cluster variables}.

Let $d=\gcd(d_1,d_2)$.
Consider the quiver $\Lambda$ with vertices $\Lambda_0=\{1,2\}$ with $d$ arrows $a_j:2\to1$, $1\le j\le d$.
Write $\FF_q$ for the finite field with $q$ elements and fix an algebraic closure $\overline{\FF}$ of $\FF_q$.
Let $\FF_{q^{d_1}},\FF_{q^{d_2}},\FF_{q^d}\subset\overline{\FF}$ denote the extension fields of $\FF_q$ of degree $d_1,d_2,d$, respectively.
Note that $\FF_{q^{d_1}}$ and $\FF_{q^{d_2}}$ are naturally identified as vector spaces over $\FF_{q^d}$.

A valued representation $V=(V_1,V_2,V_{a_j})$ of $\Lambda$ consists of $\FF_{q^{d_i}}$-vector spaces $V_i$ for $i=1,2$ and $\FF_{q^d}$-linear maps $V_{a_j}:V_2\to V_1$ for $1\le j\le d$.
For representations $V=(V_1,V_2,V_{a_j})$ and $W=(W_1,W_2,W_{a_j})$, a morphism $\theta:V\to W$ consists of $\FF_{q^{d_i}}$-linear maps $\theta_i:V_i\to W_i$ for $i=1,2$ such that the following diagram commutes for $1\le j\le d$:
\[\begin{tikzcd}
    V_1\arrow["\theta_1"']{d}& V_2 \arrow["V_{a_j}"']{l}\arrow["\theta_2"]{d}\\
    W_1 & W_2 \arrow["W_{a_j}"]{l}
  \end{tikzcd}\]
Thus the finite-dimensional valued representations of $\Lambda$ form a category $\rep(\Lambda)$.
In fact, this category is well-known to be abelian, $\FF_q$-linear, and Krull-Schmidt.
Write $\cK(\Lambda)$ for the Grothendieck group of the category $\rep(\Lambda)$, then $\cK(\Lambda)\cong\ZZ^2$ where the class $[V]=(\dim_{\FF_{q^{d_1}}}\!\!\!V_1,\dim_{\FF_{q^{d_2}}}\!\!\!V_2)$ of a valued representation $V$ of $\Lambda$ gives its \emph{dimension vector}.
Define a $\ZZ$-bilinear pairing $\langle\cdot,\cdot\rangle:\cK(\Lambda)\times\cK(\Lambda)\to\ZZ$ on the natural basis $\alpha_1=(1,0)$ and $\alpha_2=(0,1)$ by
\[\langle\alpha_i,\alpha_i\rangle=d_i,\qquad\langle\alpha_1,\alpha_2\rangle=0,\qquad\langle\alpha_2,\alpha_1\rangle=-d_1d_2.\]

For a valued representation $V$ of $\Lambda$ and a dimension vector $\bfe=(e_1,e_2)\in\cK(\Lambda)$, write $Gr_\bfe(V)$ for the \emph{Grassmannian of subrepresentations} of $V$ with dimension vector $\bfe$:
\[Gr_\bfe(V)=\{E\subset V:[E]=\bfe\}.\]
The \emph{quiver Grassmannian} $Gr_\bfe(V)$ naturally embeds as a closed subvariety in the product $Gr_{e_1}(V_1)\times Gr_{e_2}(V_2)$, in particular it is a projective variety.
When $V$ is rigid, i.e.\ $\Ext^1(V,V)=0$, Caldero and Reineke have shown \cite{caldero-reineke} that $Gr_\bfe(V)$ is smooth.

Since the field $\FF_q$ is finite, each Grassmannian $Gr_\bfe(V)$ is a finite set.
For $V$ rigid, a result of \cite{rupel3} shows that the number of points in $Gr_\bfe(V)$ can be computed by evaluating a polynomial $P_{\bfe,V}(t)\in\ZZ[t]$ at $q=|\FF_q|$.
Note that since $V$ is rigid, it is uniquely determined up to isomorphism by its dimension vector $[V]\in\cK(\Lambda)$. 
\begin{theorem}\cite[Corollary 1.2]{rupel3}
  Let $V$ be a rigid valued representations of $\Lambda$.
  For each dimension vector $\bfe\in\cK(\Lambda)$, there exists a polynomial $P_{\bfe,V}(t)\in\ZZ[t]$ depending only on the dimension vector of $V$ so that 
  \[|Gr_\bfe(V)|=P_{\bfe,V}(q).\]
\end{theorem}

It was conjectured in \cite{rupel3} that for a rigid representation $V$ the counting polynomials $P_{\bfe,V}(t)$ have positive coefficients and are unimodal.
Corollary~\ref{cor:combinatorial polynomials} proves the positivity conjecture by giving a positive combinatorial construction of these counting polynomials.
It remains an interesting open question to see how this combinatorics can be used to establish unimodality.

Define the \emph{quantum cluster character} of a rigid valued representation $V$ of $\Lambda$ by
\[Z_V=\sum\limits_{\bfe\in\cK(\Lambda)} v^{-\langle\bfe,\bfv-\bfe\rangle}P_{\bfe,V}(v^2)Z^{(-v_1+d_2e_2,-v_2+d_1(v_1-e_1))},\]
where $[V]=\bfv=(v_1,v_2)$ and $\bfe=(e_1,e_2)$.
Write $P_m$ (resp. $I_m$), $m\ge1$, for the preprojective (resp. preinjective) valued representations of $\Lambda$ (definitions can be found in \cite{rupel0} where it is shown that $[P_m]=\bfa_m$ and $[I_m]=\bfa'_m$).
Then the Laurent expansions of the non-initial quantum cluster variables $Z_m$, $m\in\ZZ\setminus\{1,2\}$, can be computed as follows.
\begin{theorem}\cite{rupel0}
  \label{th:quantum categorification}
  Assume the intermediate exchange coefficients $p_{i,j}=0$ for $i=1,2$ and $1\le j\le d_i-1$.  Then the following hold:
  \begin{enumeratea}
    \item for $m\ge3$, the quantum cluster variable $Z_m$ is equal to $Z_{P_{m-2}}$;
    \item for $m\le0$, the quantum cluster variable $Z_m$ is equal to $Z_{I_{-m+1}}$.
  \end{enumeratea}
\end{theorem}

Combining Corollary~\ref{cor:quantum combinatorial} with Theorem~\ref{th:quantum categorification}, we obtain a combinatorial construction of the counting polynomials for Grassmannians of subrepresentations in rigid valued quiver representations.
\begin{corollary}
  \label{cor:combinatorial polynomials}
  For $m\ge1$, the counting polynomials $P_{\bfe,P_m}(t)$ and $P_{\bfe,I_m}(t)$ are given by
  \begin{equation}
    \label{eq:preprojective counting polynomial}
    P_{\bfe,P_m}(t)=\sum\limits_{\omega:E_m\to\ZZ_{\ge0}} t^{\overline{\gamma}_\omega},
  \end{equation}
  where 
  \begin{itemize}
    \item the sum ranges over $(d_1,d_2)$-bounded compatible gradings $\omega$ of $D_m$ such that $\omega(H_m)\subset\{0,d_1\}$, $\omega(V_m)\subset\{0,d_2\}$, $|\supp(\omega_H)|=u_{m,1}-e_1$, and $|\supp(\omega_V)|=e_2$;
    \item $\overline{\gamma}_\omega=\sum\limits_{e<e'\in E_m}\overline{\gamma}_\omega(e,e')$ for
    \begin{equation}
      \label{eq:polynomial exponents}
      \overline{\gamma}_\omega(e,e')=\begin{cases}
        -d_1d_2 & \text{if $e\in\supp(\omega_H)$ and $e'\in\supp(\omega_V)$;}\\ 
        d_1 & \text{if $e\in\supp(\omega_H)$ and $e'\in H_m\setminus\supp(\omega_H)$;}\\ 
        d_2 & \text{if $e\in V_m\setminus\supp(\omega_V)$ and $e'\in\supp(\omega_V)$;}\\ 
        0 & \text{otherwise;}\\ 
      \end{cases}
    \end{equation} 
  \end{itemize}
  and
  \begin{equation}
    \label{eq:preinjective counting polynomial}
    P_{\bfe,I_m}(t)=\sum\limits_{\omega:E'_m\to\ZZ_{\ge0}} t^{\overline{\gamma}'_\omega},
  \end{equation}
  where
  \begin{itemize}
    \item the sum ranges over $(d_2,d_1)$-bounded compatible gradings $\omega$ of $D'_m$ such that $\omega(H'_m)\subset\{0,d_2\}$, $\omega(V'_m)\subset\{0,d_1\}$, $|\supp(\omega_{H'})|=e_2$, and $|\supp(\omega_{V'})|=u'_{m,1}-e_1$;
    \item $\overline{\gamma}'_\omega=\sum\limits_{e<e'\in E'_m}\overline{\gamma}'_\omega(e,e')$ for
    \begin{equation*}
      \overline{\gamma}'_\omega(e,e')=\begin{cases}
        -d_1d_2 & \text{if $e\in\supp(\omega_{V'})$ and $e'\in\supp(\omega_{H'})$;}\\ 
        d_1 & \text{if $e\in\supp(\omega_{V'})$ and $e'\in V'_m\setminus\supp(\omega_{V'})$;}\\ 
        d_2 & \text{if $e\in H'_m\setminus\supp(\omega_{H'})$ and $e'\in\supp(\omega_{H'})$;}\\ 
        0 & \text{otherwise.}
      \end{cases}
    \end{equation*}
  \end{itemize}
\end{corollary}
\begin{proof}
  We prove equation~\eqref{eq:preprojective counting polynomial}, the proof of equation~\eqref{eq:preinjective counting polynomial} is essentially the same. 

  By Corollary~\ref{cor:quantum combinatorial} and Theorem~\ref{th:quantum categorification}, we have
  \[P_{\bfe,P_m}(v^2)=\sum\limits_{\omega:E_m\to\ZZ_{\ge0}} v^{-\langle\bfe,\bfa_m-\bfe\rangle+1-u_{m,1}-u_{m-1,2}+\gamma_\omega+\beta_\omega},\]
  where the sum ranges over all $(d_1,d_2)$-bounded compatible gradings of $D_m$ with $\omega(H_m)\subset\{0,d_1\}$, $\omega(V_m)\subset\{0,d_2\}$, $|\supp(\omega_H)|=u_{m,1}-e_1$, and $|\supp(\omega_V)|=e_2$.
  But observe that
  \[\langle\bfe,\bfa_m-\bfe\rangle=d_1e_1(u_{m,1}-e_1)+d_2e_2(u_{m-1,2}-e_2)-d_1d_2e_2(u_{m,1}-e_1)\]
  and under the assumptions on $\omega$ we have
  \[\beta_\omega=d_1d_2e_2(u_{m,1}-e_1)+u_{m,1}u_{m-1,2}-d_1e_1(u_{m,1}-e_1)-2d_1{u_{m,1}-e_1\choose 2}-d_2e_2(u_{m-1,2}-e_2)-2d_2{e_2\choose 2}.\]
  Canceling like terms gives
  \[P_{\bfe,P_m}(v^2)=\sum\limits_{\omega:E_m\to\ZZ_{\ge0}} v^{(u_{m,1}-1)(u_{m-1,2}-1)-2d_1{u_{m,1}-e_1\choose 2}-2d_2{e_2\choose 2}+\gamma_\omega}.\]
  When $|\supp(\omega_H)|=0$ and $|\supp(\omega_V)|=0$, we have $\gamma_\omega=-2|\{e,e'\in E_m:e<e',e\in V_m, e'\in H_m\}|$.
  But these assumptions imply $\bfe=(u_{m,1},0)$ so that $P_{\bfe,P_m}(t)=1$ and thus 
  \[(u_{m,1}-1)(u_{m-1,2}-1)=2|\{e,e'\in E_m:e<e',e\in V_m, e'\in H_m\}|.\]
  In particular, the case $e\in V_m$ and $e'\in H_m$ can be ignored when computing $\gamma_\omega$ if we omit the term $(u_{m,1}-1)(u_{m-1,2}-1)$ from the exponent of $v$.
  Since $|\supp(\omega_H)|=u_{m,1}-e_1$ and $|\supp(\omega_V)|=e_2$, the cases $e,e'\in\supp(\omega_H)$ and $e,e'\in\supp(\omega_V)$ can also be ignored giving
  \[P_{\bfe,P_m}(v^2)=\sum\limits_{\omega:E_m\to\ZZ_{\ge0}} v^{2\overline{\gamma}_\omega}.\]
  This gives the result since $v$ was an indeterminate.
\end{proof}
\begin{remark}
  The exponents in equation~\eqref{eq:preprojective counting polynomial} are not manifestly positive, however equation~\eqref{eq:polynomial exponents} giving the exponents can be refined as follows.
  Consider $e\in\supp(\omega_H)$ and $e'\in\supp(\omega_V)$ with $e<e'$ which contributes a term $-d_1d_2$ in equation~\eqref{eq:polynomial exponents}.
  The $d_2$ horizontal edges preceding $e'$ cannot be in the support of $\omega_H$ by compatibility, moreover each such horizontal edge $h$ satisfies $e<h$.
  In particular, these pairs $e<h$ together contribute a term $d_1d_2$ in equation~\eqref{eq:polynomial exponents}.
  Thus the negative contribution to $\overline{\gamma}_\omega$ will always cancel and equation~\eqref{eq:preprojective counting polynomial} indeed gives $P_{\bfe,P_m}(t)$ as a polynomial in $t$.
\end{remark}


\end{document}